\documentclass[11pt,a4paper,reqno]{amsart}
\pdfoutput=1
\usepackage{amsfonts,amsmath,amssymb,amsthm,amsxtra,mathrsfs}
\usepackage{color}
\usepackage{float}
\usepackage[colorlinks,linkcolor=blue!72,anchorcolor=orange, citecolor=red,urlcolor=Emerald,bookmarksopen,bookmarksdepth=2]{hyperref}
\usepackage[usenames,dvipsnames,svgnames]{xcolor}
\usepackage{enumitem}
\usepackage{geometry,array}
\usepackage{graphicx}
\usepackage{subfigure}
\usepackage{adjustbox}
\usepackage{url}
\usepackage[all]{xy}
\usepackage{mathdots}
\usepackage{changepage}
\usepackage{dsfont}
\usepackage{bm}

\usepackage{tikz}\usetikzlibrary{matrix}\usetikzlibrary{trees}
\usetikzlibrary{decorations.pathreplacing,decorations.markings}
 \tikzset{
  on each segment/.style={
    decorate,
    decoration={
      show path construction,
      moveto code={},
      lineto code={
        \path [#1]
        (\tikzinputsegmentfirst) -- (\tikzinputsegmentlast);
      },
      curveto code={
        \path [#1] (\tikzinputsegmentfirst)
        .. controls
        (\tikzinputsegmentsupporta) and (\tikzinputsegmentsupportb)
        ..
        (\tikzinputsegmentlast);
      },
      closepath code={
        \path [#1]
        (\tikzinputsegmentfirst) -- (\tikzinputsegmentlast);
      },
    },
  },
  mid arrow/.style={postaction={decorate,decoration={
        markings,
        mark=at position .5 with {\arrow[#1]{stealth}}
      }}},
}
\usetikzlibrary{arrows}
\newcommand{\midarrow}{\tikz \draw[-triangle 90] (0,0) -- +(.1,0);}

\numberwithin{figure}{section}
\numberwithin{equation}{section}

\newcommand*\circled[1]{\raisebox{.5pt}{\textcircled{\raisebox{-.9pt} {#1}}}}

\def\BLUE{blue} 
\def\RED{red} 

\def\hua{\mathcal}
\def\kong{\mathbb}
\def\<{\langle}
\def\>{\rangle}
\def\NN{\mathbb{N}}
\def\ZZ{\mathbb{Z}}
\def\QQ{\mathbb{Q}}
\def\CC{\mathbb{C}}
\def\Mat{\mathbf{Mat}}
\def\thick{\operatorname{thick}}
\def\SURF{\mathbf{S}}
\def\dac{\mathbf{A}}
\newcommand{\pd}{\operatorname{proj. dim.}}

\newcommand{\Pic}{Figure~}
\newcommand{\proj}{\mathsf{proj}}
\newcommand{\add}{\mathsf{add}}
\newcommand{\per}{\mathsf{per}}
\newcommand{\modcat}{\mathsf{mod}}
\newcommand{\Hom}{\mathsf{Hom}}
\newcommand{\Ext}{\mathsf{Ext}}
\newcommand{\kk}{\mathds{k}}
\newcommand{\A}{\overrightarrow{\mathbb{A}}}
\newcommand{\Surf}{\mathcal{S}}
\newcommand{\Marked}{\mathcal{M}}
\newcommand{\Y}{\mathcal{Y}}
\newcommand{\Ribbon}{\mathcal{R}}
\newcommand{\F}{\lambda}
\newcommand{\ta}{\widetilde{a}}
\newcommand{\tb}{\widetilde{b}}
\newcommand{\tc}{\widetilde{c}}
\newcommand{\txi}{\widetilde{\xi}}
\newcommand{\tsigma}{\widetilde{\sigma}}
\newcommand{\ttau}{\widetilde{\tau}}
\newcommand{\teta}{\widetilde{\eta}}
\newcommand{\tgamma}{\widetilde{\gamma}}
\newcommand{\Int}{\mathrm{Int}}
\newcommand{\ii}{\mathfrak{i}}
\newcommand{\pSurf}{\partial\mathcal{S}}
\newcommand{\innerSurf}{\mathcal{S}\backslash\partial\mathcal{S}}
\newcommand{\gbullet}{{\color{ForestGreen}\bullet}}
\newcommand{\rbullet}{{\color{red}\circ}}
\newcommand{\Dgreen}{\Delta_{{\color{ForestGreen}\bullet}}}
\newcommand{\tDgreen}{\tilde{\Delta}_{{\color{ForestGreen}\bullet}}}
\newcommand{\Dred}{\Delta_{{\color{red}\circ}}}
\newcommand{\tDred}{\tilde{\Delta}_{{\color{red}\circ}}}
\newcommand{\PP}{\mathcal{P}}
\newcommand{\KK}{\mathcal{K}}
\newcommand{\AS}{\mathrm{AS}}
\newcommand{\SSS}{\mathfrak{S}}
\newcommand{\ccc}{\mathfrak{c}}
\newcommand{\CCC}{\mathfrak{C}}
\newcommand{\longline}{-\!\!\!-\!\!\!-\!\!\!-\!\!\!-}
\newcommand{\ind}{\mathsf{ind}}

\newcommand{\FASgreen}{\mathrm{FAS}_{{\color{ForestGreen}\bullet}}}
\newcommand{\FASred}{\mathrm{FAS}_{{\color{red}\circ}}}
\newcommand{\FFASgreen}{\mathrm{FFAS}_{{\color{ForestGreen}\bullet}}}
\newcommand{\FFASred}{\mathrm{FFAS}_{{\color{red}\circ}}}
\newcommand{\npFASgreen}{\mathrm{npFAS}_{{\color{ForestGreen}\bullet}}}
\newcommand{\npFFASgreen}{\mathrm{npFFAS}_{{\color{ForestGreen}\bullet}}}
\newcommand{\OC}{\mathrm{OC}}
\newcommand{\OCM}{\mathrm{OC}_{\mathcal{M}}}
\newcommand{\GOC}{\widetilde{\mathrm{OC}}}
\newcommand{\GOCM}{\widetilde{\mathrm{A}}_{\mathcal{M}}}
\newcommand{\SGOCM}{\widetilde{\mathrm{SA}}_{\mathcal{M}}}
\newcommand{\nonselfGOCM}{\widetilde{\mathrm{OC}}{ }^{\pmb{\times}}_{\mathcal{M}}}
\newcommand{\nonselfOCM}{\mathrm{OC}^{\pmb{\times}}_{\mathcal{M}}}
\newcommand{\bfC}{($\mathbf{C}$)}
\newcommand{\bfI}{($\mathbf{I}$)}
\newcommand{\bfII}{($\mathbf{II}$)}
\newcommand{\bfIII}{($\mathbf{III}$)}
\newcommand{\bfsquare}{\pmb{\square}}

\def\wc{\widetilde{c}}
\def\grad{\lambda}
\def\wa{\widetilde{a}}
\def\SOA{\operatorname{SOA}}
\def\ws{\widetilde{\sigma}}
\def\wt{\widetilde{\tau}}
\def\walpha{\widetilde{\alpha}}
\def\alg{\Lambda}

\newcommand{\num}[2]{l_{#1}^{#2}}
\newcommand{\snum}[2]{l_{#1}^{}}
\newcommand{\ip}[2]{z_{#1}^{#2}}
\newcommand{\sip}[2]{z_{#1}^{}}
\newcommand{\dii}[2]{d_{#1}^{#2}}
\newcommand{\sdii}[2]{d_{#1}^{}}

\begin{document}

\newcommand{\nc}{\newcommand}


\allowdisplaybreaks[4]

\newtheorem{theorem}{Theorem}[section]
\newtheorem{lemma}[theorem]{Lemma}
\newtheorem{corollary}[theorem]{Corollary}
\newtheorem{proposition}[theorem]{Proposition}
\newtheorem{simp}[theorem]{Lemma/Simplification}
\theoremstyle{definition}
\newtheorem{definition}[theorem]{Definition}
\newtheorem{construction}[theorem]{Construction}
\newtheorem{question}[theorem]{Question}
\newtheorem{remark}[theorem]{Remark}
\newtheorem{example}[theorem]{Example}
\newtheorem{notation}[theorem]{Notation}
\newtheorem{statement}[theorem]{Statement}
\newtheorem*{question1}{Rank Question for presilting complexes}
\newtheorem*{question2}{Complement Question for presilting complexes}

\usetikzlibrary{arrows}

\numberwithin{equation}{section}

\renewcommand{\thefootnote}{\fnsymbol{footnote}}
\setcounter{footnote}{0}

\title{A negative answer to Complement Question for presilting complexes}

\date{\today}

\author{Yu-Zhe Liu}
\address{School of Mathematics and statistics,
    Guizhou University, 550025,
    Guiyang, China}
\email{yzliu3@163.com}

\author{Yu Zhou}
\address{Yau Mathematical Sciences Center,
	Tsinghua University, 100084,
	Beijing, China}
\email{yuzhoumath@gmail.com}

\thanks{The work was supported by National Natural Science Foundation of China (Grants Nos. 11961007, 12171207 and 12271279).}


\maketitle

\begin{abstract}
In this paper using a geometric model we show that there is a presilting complex over a finite dimensional algebra, which is not a direct summand of a silting complex.
\end{abstract}

\section*{\bf Introduction}

Tilting theory plays a central role in the representation theory of algebras. Let $\kk$ be an algebraically closed field and $\alg$ a finite dimensional $\kk$-algebra. Denote by $\modcat \alg$ the category of finitely generated right $\alg$-modules. A $\alg$-module $T\in \modcat\alg$ is called {\it (generalized) tilting} \cite{Miy1986} if
\begin{itemize}
  \item[(T1)] its projective dimension $\pd T<\infty$,
  \item[(T2)] $\Ext_\alg^{i}(T,T)=0$ for all $i>0$, and
  \item[(T3)] there exists an exact sequence
  $$0 \to \alg \to T_0 \to \cdots \to T_n \to 0$$
  with $T_i\in \add T$, $0\le i\le n$, where $\add T$ is the full subcategory of $\modcat \alg$ whose objects are all finite direct sums of direct summands of $T$.
\end{itemize}

A tilting $\alg$-module $T$ is called \emph{classical tilting} \cite{APR1979, BB1979, HR1982} if $\pd T \leq 1$. Bongartz \cite{Bong1982} showed that for classical tilting modules, the condition (T3) can be replaced by
\begin{itemize}
    \item[(T3')] $|T|=|\alg|$, where $|X|$ denotes the number of non-isomorphic indecomposable direct summands of $X\in\modcat \alg$.
\end{itemize}
The ideal which Bongartz used is to show that any $\alg$-module $T$ with $\pd T \leq 1$ and satisfying condition (T2) is a direct summand of a classical tilting module. However, this is not true in the general case, see \cite{RiSch1989} for a counter-example. The question whether (T3) can be replaced by (T3') is open in the general case.

Let $D^b(\modcat \alg)$ be the bounded derived category of $\modcat \alg$ and $\thick \alg$ the thick subcategory of $D^b(\modcat \alg)$ containing $\alg$. 
The category $\thick \alg$ is equivalent to the bounded homotopy category $K^b(\proj \alg)$ of finitely generated projective $\alg$-modules. A complex $T^\bullet\in D^b(\modcat \alg)$ is called \emph{tilting} \cite{Ri1989} if
\begin{itemize}
  \item[(TC1)] $T^\bullet\in\thick \alg$,
  \item[(TC2)] $\Hom_{D^b(\modcat \alg)}(T^\bullet, T^\bullet[i])=0$ for all $i\ne 0$, and
  \item[(TC3)] $\thick T^\bullet=\thick \alg$.
\end{itemize}
Note that each tilting $\alg$-module, when regarded as a stalk complex, is a tilting complex in $D^b(\modcat \alg)$.
There is a counter-example given in \cite{Ri1989} showing that not every complex in $D^b(\modcat \alg)$ satisfying (TC1) and (TC2) is a direct summand of a tilting complex. It was pointed out in \cite{LVY2014} that the counter-example in \cite{RiSch1989} mentioned above is indeed also a counter-example in this case. The question whether condition (TC3) can be replaced with (TC3') $|T^\bullet|=|\alg|$ is open in general.

Silting complexes were introduced by Keller and Vossieck in \cite{KV1988} as a generalization of tilting complexes, and were recently found to have rich interplay with cluster theory, torsion theory, simple-minded systems and Bridgeland stability conditions, see e.g. \cite{AI2012,AIR2014,KY2014,QW2018}.
A complex $S^\bullet\in D^b(\modcat \alg)$ is called \emph{silting} if
\begin{itemize}
  \item[(S1)] $S^\bullet\in\thick \alg$,
  \item[(S2)] $\Hom_{D^b(\modcat \alg)}(S^\bullet, S^\bullet[i])=0$ for all $i>0$, and
  \item[(S3)] $\thick S^\bullet=\thick \alg$.
\end{itemize}
By definition, each tilting complex is silting but the converse is not true in general. A complex $S^\bullet\in D^b(\modcat\alg)$ is called \emph{presilting} if $S^\bullet$ satisfies (S1) and (S2). Similarly as in the tilting case, we have the following two questions.
\begin{question2}\label{quest. silt. comp. complete}
Is any presilting complex in $D^b(\modcat \alg)$ is a direct summand of a silting complex?
\end{question2}

\begin{question1}\label{quest. silt. comp.}
Is a presilting $S^\bullet\in D^b(\modcat\alg)$ that satisfies $|S^\bullet|=|\alg|$ always silting?
\end{question1}

In the case that $S^\bullet$ of 2-term or $\alg$ piecewise hereditary, the answer to these two questions are positive, see \cite{Wei2013, AIR2014, BY2013, LL2019, XY2020, DF2022}. It was pointed out in \cite{Wei2020} that the counter-examples in \cite{RiSch1989} and \cite{Ri1989} mentioned above are not counter-examples in the silting case. In this paper, we give a negative answer to Complement Question for presilting complexes. Let $\alg = \kk Q/I$, where
\[Q = \xymatrix{1 \ar@/^0.5pc/[r]^{x_1} \ar@/_0.5pc/[r]_{y_1} & 2 \ar@/^0.5pc/[r]^{x_2} \ar@/_0.5pc/[r]_{y_2} & 3 },\ \text{and}\
I = \langle x_1x_2, y_1y_2 \rangle. \]

\begin{theorem}\label{main}
There is a presilting complex in $D^b(\modcat \alg)$ which is not a direct summand of any silting complex.
\end{theorem}

The main ingredient in the proof is the geometric model of derived categories of gentle algebras introduced in \cite{HKK2017,OPS2018}.

We remark that the global dimension of $\alg$ in this example is two, and the presilting complex which we construct is of 3-term (see Remark~\ref{rmk:app}). We also remark that the answer to Rank Question for presilting complexes in this example is positive, because $\alg$ is a gentle algebra and hence one can apply \cite[Proposition~5.7]{APS2023} to it.

\subsection*{Acknowledgement}

We would like to thank Wen Chang, Xiao-Wu Chen, Changjian Fu, Martin Kalck and Zhengfang Wang for their helpful discussions.


\section{A geometric model of the algebra}

Let $\SURF=(\Surf,\Y,\Marked,\lambda)$ be a marked surface, where
\begin{itemize}
    \item $\Surf$ is a torus whose boundary $\partial\Surf$ has exactly one component,
    \item $\Y=\{r,s\}\subset\partial\Surf$ and $\Marked=\{p,q\}\subset\partial\Surf$ such that $r,p,s,q$ are in the anticlockwise order around the boundary, and
    \item $\grad$ is a section of the projectivized tangent bundle $\mathbb{P}T(\Surf)$ of $\Surf$ as shown in Figure~\ref{exp. presilting}.
\end{itemize}

\begin{figure}[htpb]
\begin{center}
\definecolor{ffqqqq}{rgb}{1,0,0}
\definecolor{qqwuqq}{rgb}{0,0,1}
\begin{tikzpicture}
\filldraw[black!20] (0,-0.8) circle (0.44cm);
\draw [line width=1.2pt][rotate around={0:(0,0)}] (0,0) ellipse (2.83cm and 2cm);
\draw [line width=1.2pt][shift={(-0.01,2.12)}] plot[domain=4.23:5.21,variable=\t]({1*2.12*cos(\t r)+0*2.12*sin(\t r)},{0*2.12*cos(\t r)+1*2.12*sin(\t r)});
\draw [line width=1.2pt][shift={(0.01,-0.91)}] plot[domain=0.98:2.17,variable=\t]({1*1.24*cos(\t r)+0*1.24*sin(\t r)},{0*1.24*cos(\t r)+1*1.24*sin(\t r)});
\draw [line width=1.2pt] (0,-0.8) circle (0.44cm);
\begin{scriptsize}
\fill [color=qqwuqq] (-0.38,-0.58) circle (2.5pt);
\draw [qqwuqq] (-0.38,-0.58) node[left]{$p$};
\fill [color=qqwuqq] (0.38,-1.02) circle (2.5pt);
\draw [qqwuqq] (0.38,-1.02) node[right]{$q$};
\filldraw [color=white, draw=ffqqqq, line width=0.5pt] (0.38,-0.58) circle (2.5pt);
\draw [ffqqqq] (0.38,-0.58) node[right]{$r$};
\filldraw [color=white, draw=ffqqqq, line width=0.5pt] (-0.38,-1.02) circle (2.5pt);
\draw [ffqqqq] (-0.38,-1.02) node[left]{$s$};
\end{scriptsize}
\end{tikzpicture}
\ \ \ \
\begin{tikzpicture} [scale=0.8]
\draw [pink] (-2.5, 0.0)-- (2.5, 0.0);
\draw [pink] (-2.5, 0.3)-- (2.5, 0.3);
\draw [pink] (-2.5, 0.6)-- (2.5, 0.6);
\draw [pink] (-2.5, 0.9)-- (2.5, 0.9);
\draw [pink] (-2.5, 1.2)-- (2.5, 1.2);
\draw [pink] (-2.5, 1.5)-- (2.5, 1.5);
\draw [pink] (-2.5, 1.8)-- (2.5, 1.8);
\draw [pink] (-2.5, 2.1)-- (2.5, 2.1);
\draw [pink] (-2.5, 2.4)-- (2.5, 2.4);
\draw [pink] (-2.5,-0.3)-- (2.5,-0.3);
\draw [pink] (-2.5,-0.6)-- (2.5,-0.6);
\draw [pink] (-2.5,-0.9)-- (2.5,-0.9);
\draw [pink] (-2.5,-1.2)-- (2.5,-1.2);
\draw [pink] (-2.5,-1.5)-- (2.5,-1.5);
\draw [pink] (-2.5,-1.8)-- (2.5,-1.8);
\draw [pink] (-2.5,-2.1)-- (2.5,-2.1);
\draw [pink] (-2.5,-2.4)-- (2.5,-2.4);
\draw [pink] (2.8,1) node{$\F$};
\filldraw [black!20] (-2, 2) circle (0.5cm);
\filldraw [black!20] ( 2, 2) circle (0.5cm);
\filldraw [black!20] ( 2,-2) circle (0.5cm);
\filldraw [black!20] (-2,-2) circle (0.5cm);
\begin{scriptsize}
\draw [color=qqwuqq] (-1.8, 1.8) node{$q$}; \draw [color=qqwuqq] ( 1.8,-1.8) node{$p$};
\draw [color=ffqqqq] (-1.8,-1.8) node{$r$}; \draw [color=ffqqqq] ( 1.8, 1.8) node{$s$};
\end{scriptsize}
\draw [line width=1pt] (-2, 2) circle (0.5cm);
\draw [line width=1pt] ( 2, 2) circle (0.5cm);
\draw [line width=1pt] ( 2,-2) circle (0.5cm);
\draw [line width=1pt] (-2,-2) circle (0.5cm);
\draw [dash pattern=on 2pt off 2pt] (-2, 2)-- ( 2, 2);
\draw [dash pattern=on 2pt off 2pt] ( 2, 2)-- ( 2,-2);
\draw [dash pattern=on 2pt off 2pt] ( 2,-2)-- (-2,-2);
\draw [dash pattern=on 2pt off 2pt] (-2, 2)-- (-2,-2);
\begin{scriptsize}
\fill [color=qqwuqq] (-1.65, 1.65) circle (2.5pt);
\fill [color=qqwuqq] ( 1.65,-1.65) circle (2.5pt);
\filldraw [color=white, draw=ffqqqq, line width=0.5pt] (-1.65, 2.35) circle (2.5pt); \draw [color=ffqqqq] (-1.65, 2.35) node[above]{$r$};
\fill [color=qqwuqq] (-2.35, 2.35) circle (2.5pt); \draw [color=qqwuqq] (-2.35, 2.35) node[left]{$p$};
\filldraw [color=white, draw=ffqqqq, line width=0.5pt] (-1.65,-1.65) circle (2.5pt);
\filldraw [color=white, draw=ffqqqq, line width=0.5pt] (-2.35,-2.35) circle (2.5pt); \draw [color=ffqqqq] (-2.35,-2.35) node[left]{$s$};
\fill [color=qqwuqq] (-2.35,-1.65) circle (2.5pt); \draw [color=qqwuqq] (-2.35,-1.65) node[left]{$p$};
\fill [color=qqwuqq] (-1.65,-2.35) circle (2.5pt); \draw [color=qqwuqq] (-1.65,-2.35) node[right]{$q$};
\filldraw [color=white, draw=ffqqqq, line width=0.5pt] (-2.35, 1.65) circle (2.5pt); \draw [color=ffqqqq] (-2.35, 1.65) node[left]{$s$};
\fill [color=qqwuqq] ( 1.65, 2.35) circle (2.5pt); \draw [color=qqwuqq] ( 1.65, 2.35) node[left]{$p$};
\filldraw [color=white, draw=ffqqqq, line width=0.5pt] ( 1.65, 1.65) circle (2.5pt);
\filldraw [color=white, draw=ffqqqq, line width=0.5pt] ( 2.35, 2.35) circle (2.5pt); \draw [color=ffqqqq] ( 2.35, 2.35) node[right]{$r$};
\fill [color=qqwuqq] ( 2.35, 1.65) circle (2.5pt); \draw [color=qqwuqq] ( 2.35, 1.65) node[right]{$q$};
\filldraw [color=white, draw=ffqqqq, line width=0.5pt] ( 2.35,-1.65) circle (2.5pt); \draw [color=ffqqqq] ( 2.35,-1.65) node[right]{$r$};
\fill [color=qqwuqq] ( 2.35,-2.35) circle (2.5pt); \draw [color=qqwuqq] ( 2.35,-2.35) node[right]{$q$};
\filldraw [color=white, draw=ffqqqq, line width=0.5pt] ( 1.65,-2.35) circle (2.5pt); \draw [color=ffqqqq] ( 1.65,-2.35) node[below]{$s$};
\end{scriptsize}
\end{tikzpicture}
\caption{A marked torus whose boundary has exactly one component}
\label{exp. presilting}
\end{center}
\end{figure}
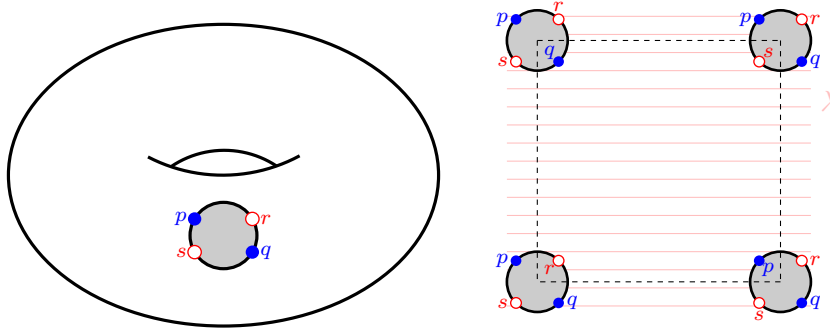

An arc on $\SURF$ is an immersion $c:[0,1]\to \Surf$. The opposite direction $\underleftarrow{c}$ of $c$ is an arc given by $\underleftarrow{c}(t)=c(1-t)$, $t\in[0,1]$. We always consider arcs up to direction and homotopy.

A grading $\wc$ on an arc $c$ is given by a homotopy class of paths in $\mathbb{P}T_{c(t)}(\Surf)$ from $\grad(c(t))$ to $\dot{c}(t)$, varying continuously with $t\in [0,1]$. The pair $(c,\wc)$ (or $\wc$ for short) is called a graded arc. The shift $\wc[d]$ of $\wc$ by an integer $d\in\mathbb{Z}$ is the graded arc whose underlying arc is the same as $\wc$ and whose grading is the composition of $\wc(t):\grad(c(t))\to\dot{c}(t)$ and the path from $\dot{c}(t)$ to itself given by clockwise rotation by $d\pi$.

Let $\wc_1,\wc_2$ be two graded arcs on $\SURF$ in a minimal position. An intersection $\ip{}{}$ between $\wc_1$ and $\wc_2$ is called an oriented intersection from $\wc_1$ to $\wc_2$ if there is a small arc in $\Surf\setminus\partial\Surf$ around $\ip{}{}$ from a point in $\wc_1$ to a point in $\wc_2$ clockwise, see Figure~\ref{fig:int}.
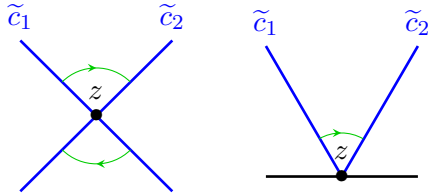
\begin{figure}[htpb]
\begin{center}
\definecolor{ffqqqq}{rgb}{1,0,0}
\definecolor{qqwuqq}{rgb}{0,0,1}
\definecolor{intcorlor}{rgb}{0,0.8,0}
\begin{tikzpicture}
\draw [qqwuqq][line width=1pt] (-1, 1) node[above]{$\tc_1$} -- ( 1,-1);
\draw [qqwuqq][line width=1pt] ( 1, 1) node[above]{$\tc_2$} -- (-1,-1);
\draw [black] (0,0) node{$\bullet$};
\draw [black] (0,0) node[above]{$\ip{}{}$};
\draw [intcorlor][postaction={on each segment={mid arrow=intcorlor}}]
      (-0.45, 0.45) arc(135:45:0.45*1.41);
\draw [intcorlor][postaction={on each segment={mid arrow=intcorlor}}]
      ( 0.45,-0.45) arc(-45:-135:0.45*1.41);
\end{tikzpicture}
\ \ \ \
\begin{tikzpicture}
\draw [black][line width=1pt] (-1,0) -- (1,0); 
\draw [qqwuqq][line width=1pt] (0,0) -- (-1,1.73) node[above]{$\tc_1$};
\draw [qqwuqq][line width=1pt] (0,0) -- ( 1,1.73) node[above]{$\tc_2$};
\draw [black] (0,0) node{$\bullet$};
\draw [black] (0,0) node[above]{$\ip{}{}$};
\draw [intcorlor][postaction={on each segment={mid arrow=intcorlor}}]
      (-0.3, 0.5) arc(120:60:0.58);
\end{tikzpicture}
\caption{Oriented intersections from $\tc_1$ to $\tc_2$}
\label{fig:int}
\end{center}
\end{figure}
For any oriented intersection $\ip{}{}=c_1(t_1)=c_2(t_2)$ from $\wc_1$ to $\wc_2$, the intersection index of $\wc_1$ and $\wc_2$ at $\ip{}{}$ is defined to be
\[\ii_{\ip{}{}}(\wc_1,\wc_2)=\wc_1(t_1)\cdot\kappa\cdot\wc_2^{-1}(t_2)\ \in\pi_1(\mathbb{P}T_{\ip{}{}}(\Surf))\cong \ZZ,\]
where $\kappa$ is (the homotopy class of) the path in $\mathbb{P}T_{\ip{}{}}(\Surf)$ from $\dot{c}_1(t_1)$ to $\dot{c}_2(t_2)$
given by clockwise rotation by an angle smaller than $\pi$. We have (cf. \cite[(2.4)]{HKK2017})
\begin{equation}\label{eq:hkk}
\ii_{\ip{}{}}(\wc_1,\wc_2)+\ii_{\ip{}{}}(\wc_2,\wc_1)=1.
\end{equation}
Let $\Int^d(\wc_1,\wc_2)$ be the number of oriented intersections from $\wc_1$ to $\wc_2$ of index $d$.


From now on, we take a collection $\dac=\{\wa_1,\wa_2,\wa_3\}$ of graded arcs whose endpoints are in $\Y$ as shown in Figure~\ref{exp. grading}. The marked surface is divided by the arcs in $\dac$ into two quadrilaterals. When taking a graded arc whose endpoints are in $\Marked$, we always assume that it is in a minimal position with the arcs in $\dac$.

\begin{figure}[htpb]
\begin{center}
\definecolor{ffqqqq}{rgb}{1,0,0}
\definecolor{qqwuqq}{rgb}{0,0,1}
\begin{tikzpicture} [scale=0.8]
\filldraw [black!20] (-2, 2) circle (0.5cm);
\filldraw [black!20] ( 2, 2) circle (0.5cm);
\filldraw [black!20] ( 2,-2) circle (0.5cm);
\filldraw [black!20] (-2,-2) circle (0.5cm);
\begin{scriptsize}
\draw [color=qqwuqq] (-1.8, 1.8) node{$q$}; \draw [color=qqwuqq] ( 1.8,-1.8) node{$p$};
\draw [color=ffqqqq] (-1.8,-1.8) node{$r$}; \draw [color=ffqqqq] ( 1.8, 1.8) node{$s$};
\end{scriptsize}
\draw [line width=1pt] (-2, 2) circle (0.5cm);
\draw [line width=1pt] ( 2, 2) circle (0.5cm);
\draw [line width=1pt] ( 2,-2) circle (0.5cm);
\draw [line width=1pt] (-2,-2) circle (0.5cm);
\draw [dash pattern=on 2pt off 2pt] (-2, 2)-- ( 2, 2);
\draw [dash pattern=on 2pt off 2pt] ( 2, 2)-- ( 2,-2);
\draw [dash pattern=on 2pt off 2pt] ( 2,-2)-- (-2,-2);
\draw [dash pattern=on 2pt off 2pt] (-2, 2)-- (-2,-2);
\draw [line width=1pt,color=ffqqqq] (-2.35,1.65)-- (-1.65,-1.65);
\draw [line width=1pt,color=ffqqqq] (-1.65,2.35)-- (1.65,1.65);
\draw [line width=1pt,color=ffqqqq] (1.65,1.65)-- (2.35,-1.65);
\draw [line width=1pt,color=ffqqqq] (1.65,-2.35)-- (-1.65,-1.65);
\draw [line width=1pt,color=ffqqqq] (-1.65,-1.65)-- (1.65,1.65);
\begin{scriptsize}
\fill [color=qqwuqq] (-1.65, 1.65) circle (2.5pt);
\fill [color=qqwuqq] ( 1.65,-1.65) circle (2.5pt);
\filldraw [color=white, draw=ffqqqq, line width=0.5pt] (-1.65, 2.35) circle (2.5pt); \draw [color=ffqqqq] (-1.65, 2.35) node[above]{$r$};
\fill [color=qqwuqq] (-2.35, 2.35) circle (2.5pt); \draw [color=qqwuqq] (-2.35, 2.35) node[left]{$p$};
\filldraw [color=white, draw=ffqqqq, line width=0.5pt] (-1.65,-1.65) circle (2.5pt);
\filldraw [color=white, draw=ffqqqq, line width=0.5pt] (-2.35,-2.35) circle (2.5pt); \draw [color=ffqqqq] (-2.35,-2.35) node[left]{$s$};
\fill [color=qqwuqq] (-2.35,-1.65) circle (2.5pt); \draw [color=qqwuqq] (-2.35,-1.65) node[left]{$p$};
\fill [color=qqwuqq] (-1.65,-2.35) circle (2.5pt); \draw [color=qqwuqq] (-1.65,-2.35) node[right]{$q$};
\filldraw [color=white, draw=ffqqqq, line width=0.5pt] (-2.35, 1.65) circle (2.5pt); \draw [color=ffqqqq] (-2.35, 1.65) node[left]{$s$};
\fill [color=qqwuqq] ( 1.65, 2.35) circle (2.5pt); \draw [color=qqwuqq] ( 1.65, 2.35) node[left]{$p$};
\filldraw [color=white, draw=ffqqqq, line width=0.5pt] ( 1.65, 1.65) circle (2.5pt);
\filldraw [color=white, draw=ffqqqq, line width=0.5pt] ( 2.35, 2.35) circle (2.5pt); \draw [color=ffqqqq] ( 2.35, 2.35) node[right]{$r$};
\fill [color=qqwuqq] ( 2.35, 1.65) circle (2.5pt); \draw [color=qqwuqq] ( 2.35, 1.65) node[right]{$q$};
\filldraw [color=white, draw=ffqqqq, line width=0.5pt] ( 2.35,-1.65) circle (2.5pt); \draw [color=ffqqqq] ( 2.35,-1.65) node[right]{$r$};
\fill [color=qqwuqq] ( 2.35,-2.35) circle (2.5pt); \draw [color=qqwuqq] ( 2.35,-2.35) node[right]{$q$};
\filldraw [color=white, draw=ffqqqq, line width=0.5pt] ( 1.65,-2.35) circle (2.5pt); \draw [color=ffqqqq] ( 1.65,-2.35) node[below]{$s$};
\end{scriptsize}
\draw (2.1,0) node[right]{\color{red}$a_1$};
\draw (-2.1,0) node[left]{\color{red}$a_1$};
\draw (0,0) node[right]{\color{red}$a_2$};
\draw (0,2.1) node[right]{\color{red}$a_3$};
\draw (0,-2.1) node[left]{\color{red}$a_3$};
\end{tikzpicture}
\ \
\begin{tikzpicture}
\draw [pink] (-2.5, 0.0)-- (2.5, 0.0);
\draw [pink] (-2.5, 0.4)-- (2.5, 0.4);
\draw [pink] (-2.5, 0.8)-- (2.5, 0.8);
\draw [pink] (-2.5, 1.2)-- (2.5, 1.2);
\draw [pink] (-2.5, 1.6)-- (2.5, 1.6);
\draw [pink] (-2.5, 2.0)-- (2.5, 2.0);
\filldraw [black!20] (0,0) circle (1cm);
\draw [line width=1.2pt] (0,0) circle (1cm);
\draw [color=ffqqqq][line width=1.2pt] (0.71,0.71)-- (  0,2.5); \draw (  0,2.5) node[right,red]{$a_1$};
\draw [color=ffqqqq][line width=1.2pt] (0.71,0.71)-- (2.5,  0); \draw (2.5,  0) node[right,red]{$a_3$};
\draw [color=ffqqqq][line width=1.2pt] (0.71,0.71)-- (2.5,2.5); \draw (2.5,2.5) node[right,red]{$a_2$};
\begin{scriptsize}
\filldraw [color=white, draw=ffqqqq, line width=0.5pt] (0.71,0.71) circle (2.5pt); \draw [color=ffqqqq] (0.71,0.71) node[left]{$r$};
\fill [color=qqwuqq] (-0.71,0.71) circle (2.5pt); \draw [color=qqwuqq] (-0.71,0.71) node[left]{$p$};
\filldraw [color=white, draw=ffqqqq, line width=0.5pt] (-0.71,-0.71) circle (2.5pt); \draw [color=ffqqqq] (-0.71,-0.71) node[left]{$s$};
\fill [color=qqwuqq] (0.71,-0.71) circle (2.5pt); \draw [color=qqwuqq] (0.71,-0.71) node[right]{$q$};
\end{scriptsize}
\draw [red][->] (0.36-0.25,1.60) arc (180:115:0.25);
\draw [red] (0.11, 1.7) node[left]{$\ta_1$};
\draw [red][->] (1.59+0.25,1.59) arc (  0: 45:0.25);
\draw [red] (1.84, 1.7) node[right]{$\ta_2$};
\draw [red][->] (1.50+0.25,0.40) arc (  0:155:0.25);
\draw [red] (1.5, 0.65) node[above]{$\ta_3$};
\end{tikzpicture}
\caption{A collection $\dac=\{\wa_1,\wa_2,\wa_3\}$ of graded arcs}
\label{exp. grading}
\end{center}
\end{figure}

Let $\tgamma$ be a graded arc whose endpoints are in $\Marked$ as shown in \Pic\ref{fig:tgamma}.

\begin{figure}[htbp]
\begin{center}
\definecolor{ffqqqq}{rgb}{1,0,0}
\definecolor{qqwuqq}{rgb}{0,0,1}
\begin{tikzpicture} [scale=0.8]
\filldraw [black!20] (-2, 2) circle (0.5cm);
\filldraw [black!20] ( 2, 2) circle (0.5cm);
\filldraw [black!20] ( 2,-2) circle (0.5cm);
\filldraw [black!20] (-2,-2) circle (0.5cm);
\begin{scriptsize}
\draw [color=qqwuqq] (-1.8, 1.8) node{$q$}; \draw [color=qqwuqq] ( 1.8,-1.8) node{$p$};
\draw [color=ffqqqq] (-1.8,-1.8) node{$r$}; \draw [color=ffqqqq] ( 1.8, 1.8) node{$s$};
\end{scriptsize}
\draw [line width=1pt] (-2, 2) circle (0.5cm);
\draw [line width=1pt] ( 2, 2) circle (0.5cm);
\draw [line width=1pt] ( 2,-2) circle (0.5cm);
\draw [line width=1pt] (-2,-2) circle (0.5cm);
\draw [dash pattern=on 2pt off 2pt] (-2, 2)-- ( 2, 2);
\draw [dash pattern=on 2pt off 2pt] ( 2, 2)-- ( 2,-2);
\draw [dash pattern=on 2pt off 2pt] ( 2,-2)-- (-2,-2);
\draw [dash pattern=on 2pt off 2pt] (-2, 2)-- (-2,-2);
\draw [line width=1pt,color=ffqqqq] (-2.35,1.65)-- (-1.65,-1.65);
\draw [line width=1pt,color=ffqqqq] (-1.65,2.35)-- (1.65,1.65);
\draw [line width=1pt,color=ffqqqq] (1.65,1.65)-- (2.35,-1.65);
\draw [line width=1pt,color=ffqqqq] (1.65,-2.35)-- (-1.65,-1.65);
\draw [line width=1pt,color=ffqqqq] (-1.65,-1.65)-- (1.65,1.65);
\begin{scriptsize}
\fill [color=qqwuqq] (-1.65, 1.65) circle (2.5pt);
\fill [color=qqwuqq] ( 1.65,-1.65) circle (2.5pt);
\filldraw [color=white, draw=ffqqqq, line width=0.5pt] (-1.65, 2.35) circle (2.5pt); \draw [color=ffqqqq] (-1.65, 2.35) node[above]{$r$};
\fill [color=qqwuqq] (-2.35, 2.35) circle (2.5pt); \draw [color=qqwuqq] (-2.35, 2.35) node[left]{$p$};
\filldraw [color=white, draw=ffqqqq, line width=0.5pt] (-1.65,-1.65) circle (2.5pt);
\filldraw [color=white, draw=ffqqqq, line width=0.5pt] (-2.35,-2.35) circle (2.5pt); \draw [color=ffqqqq] (-2.35,-2.35) node[left]{$s$};
\fill [color=qqwuqq] (-2.35,-1.65) circle (2.5pt); \draw [color=qqwuqq] (-2.35,-1.65) node[left]{$p$};
\fill [color=qqwuqq] (-1.65,-2.35) circle (2.5pt); \draw [color=qqwuqq] (-1.65,-2.35) node[right]{$q$};
\filldraw [color=white, draw=ffqqqq, line width=0.5pt] (-2.35, 1.65) circle (2.5pt); \draw [color=ffqqqq] (-2.35, 1.65) node[left]{$s$};
\fill [color=qqwuqq] ( 1.65, 2.35) circle (2.5pt); \draw [color=qqwuqq] ( 1.65, 2.35) node[left]{$p$};
\filldraw [color=white, draw=ffqqqq, line width=0.5pt] ( 1.65, 1.65) circle (2.5pt);
\filldraw [color=white, draw=ffqqqq, line width=0.5pt] ( 2.35, 2.35) circle (2.5pt); \draw [color=ffqqqq] ( 2.35, 2.35) node[right]{$r$};
\fill [color=qqwuqq] ( 2.35, 1.65) circle (2.5pt); \draw [color=qqwuqq] ( 2.35, 1.65) node[right]{$q$};
\filldraw [color=white, draw=ffqqqq, line width=0.5pt] ( 2.35,-1.65) circle (2.5pt); \draw [color=ffqqqq] ( 2.35,-1.65) node[right]{$r$};
\fill [color=qqwuqq] ( 2.35,-2.35) circle (2.5pt); \draw [color=qqwuqq] ( 2.35,-2.35) node[right]{$q$};
\filldraw [color=white, draw=ffqqqq, line width=0.5pt] ( 1.65,-2.35) circle (2.5pt); \draw [color=ffqqqq] ( 1.65,-2.35) node[below]{$s$};
\end{scriptsize}
\draw [blue] ( 1.2, 1) node[right]{$\gamma$};
\draw [blue] (-2.8, 1) node[right]{$\gamma$};
\draw [blue] ( 1.2,-3) node[right]{$\gamma$};
\draw [blue] (-2.8,-3) node[right]{$\gamma$};
\draw (2.1,0) node[right]{\color{red}$a_1$};
\draw (-2.1,0) node[left]{\color{red}$a_1$};
\draw (0,0) node[right]{\color{red}$a_2$};
\draw (0,2.1) node[right]{\color{red}$a_3$};
\draw (0,-2.1) node[left]{\color{red}$a_3$};
\draw [line width=1pt,color=blue] (1.65,2.35) to[out=-135,in=135] (1.35, 1.35) to[out=-45,in=-135] (2.35,1.65);
\draw [line width=1pt,color=blue] (1.65,2.35-3.98) to[out=-135,in=135] (1.35, 1.35-3.98) to[out=-45,in=-135] (2.35,1.65-3.98);
\draw [line width=1pt,color=blue] (1.65-3.98,2.35-3.98) to[out=-135,in=135] (1.35-3.98, 1.35-3.98) to[out=-45,in=-135] (2.35-3.98,1.65-3.98);
\draw [line width=1pt,color=blue] (1.65-3.98,2.35) to[out=-135,in=135] (1.35-3.98, 1.35) to[out=-45,in=-135] (2.35-3.98,1.65);
\end{tikzpicture}
\ \ \ \ \ \ \ \
\begin{tikzpicture}
\draw [pink] (-2.5, 0.0)-- (2.5, 0.0);
\draw [pink] (-2.5,-0.4)-- (2.5,-0.4);
\draw [pink] (-2.5,-0.8)-- (2.5,-0.8);
\draw [pink] (-2.5,-1.2)-- (2.5,-1.2);
\draw [pink] (-2.5,-1.6)-- (2.5,-1.6);
\draw [pink] (-2.5,-2.0)-- (2.5,-2.0);
\filldraw [black!20] (0,0) circle (1cm);
\draw [line width=1.2pt] (0,0) circle (1cm);
\draw [color=ffqqqq][line width=1.2pt] (-0.71,-0.71)-- (-2.5,0) node[below]{$a_3$};
\draw [color=ffqqqq][line width=1.2pt] (-0.71,-0.71)-- (0,-2.5) node[right]{$a_1$};
\draw [color=ffqqqq][line width=1.2pt] (-0.71,-0.71)-- (-2.5,-2.5) node[right]{$a_2$};
\begin{scriptsize}
\filldraw [color=white, draw=ffqqqq, line width=0.5pt] (0.71,0.71) circle (2.5pt); \draw [color=ffqqqq] (0.71,0.71) node[right]{$r$};
\fill [color=qqwuqq] (-0.71,0.71) circle (2.5pt); \draw [color=qqwuqq] (-0.71,0.71) node[left]{$p$};
\filldraw [color=white, draw=ffqqqq, line width=0.5pt] (-0.71,-0.71) circle (2.5pt); \draw [color=ffqqqq] (-0.71,-0.71) node[right]{$s$};
\fill [color=qqwuqq] (0.71,-0.71) circle (2.5pt); \draw [color=qqwuqq] (0.71,-0.71) node[right]{$q$};
\end{scriptsize}
\draw [line width=1pt,color=blue] (-0.71, 0.71) to[out=-135,in=135] (-1.50,-1.50) to[out=-45,in=-135] ( 0.71,-0.71);
\draw [blue][->] (0.18+0.25,-1.2) arc (0:  42:0.25); \fill (0.18,-1.2) circle (1pt);
\draw [blue] (0.18+0.25,-1.1) node[right]{$\tgamma$};
\draw [blue][->] (-1.7+0.25,-1.2) arc (0: -55:0.25); \fill (-1.7,-1.2) circle (1pt);
\draw [blue] (-1.7+0.25,-1.3) node[above]{$\tgamma$};
\draw [blue][->] (-1.35+0.25,0.0) arc (0:-123:0.25); \fill (-1.35,0.0) circle (1pt);
\draw [blue] (-1.35+0.25,-0.1) node[below]{$\tgamma$};
\end{tikzpicture}
\caption{A graded arc $\tgamma$}
\label{fig:tgamma}
\end{center}
\end{figure}

\begin{example} \label{exp:gamma}

Let $\ip{i}{}$ be the oriented intersection from $\tgamma$ to $\wa_i$, $1\le i\le 3$, respectively.
Then we have $\ii_{\ip{1}{}}(\tilde{\gamma}, \tilde{a}_1)=0$, $\ii_{\ip{2}{}}(\tilde{\gamma}, \tilde{a}_2) = 1$ and $\ii_{\ip{3}{}}(\tilde{\gamma}, \tilde{a}_3) = 2$, see Figure~\ref{fig:intind}.

\begin{figure}[htpb]
\centering
\definecolor{uuuuuu}{rgb}{0.27,0.27,0.27}
\definecolor{qqqqff}{rgb}{0,0,1}
\definecolor{ffqqqq}{rgb}{1,0,0}
\begin{tikzpicture}
\draw [color=pink] (-2,0)-- (2,0);
\draw [line width=1.2pt][color=red] (-0.5,1.0)-- (0.75,-1.5);
\draw [line width=1.2pt][color=blue] (-1.5,-0.82)-- (1.5,0.82);
\fill (0,0) circle (1.5pt);
\draw (0.1,0.3) node {$\ip{1}{}$};
\draw [blue][->] (0.+1.0,0.) arc (0:28:1.0); \draw [blue] (0.+1.0,0.) node[below]{$\tgamma$};
\draw [black][->] (1.10, 0.61) arc (30:-61:1.3); \draw [black] (0.+1.2,0.2) node[right]{$\kappa$};
\draw [red][->] (0.72,-1.43) arc (-61:0:1.6); \draw [red] (1.12,-1.12) node[right] {$\tilde{a}_1^{-1}$};
\end{tikzpicture}
\ \
\begin{tikzpicture}
\draw [color=pink] (-2,0)-- (2,0);
\draw [line width=1.2pt][color=red]  (-1.45,-1.45)-- (1.0, 1.0);
\draw [line width=1.2pt][color=blue] (-1.0, 1.0)-- (1.45,-1.45);
\fill (0,0) circle (1.5pt);
\draw (0,0.3) node {$\ip{2}{}$};
\draw [blue][->]  (0.+1.0,0.) arc (0:-45:1.0); \draw [blue] (0.+1.0,-0.25) node[right]{$\tgamma$};
\draw [black][->](0.71,-0.71) arc (-45:-135:1.0); \draw [black] (0.,-1.) node[below]{$\kappa$};
\draw [red][->] (-0.71,-0.71) arc (-135:-180:1.0); \draw [red] (0.-1.0,-0.25) node[left] {$\tilde{a}_2^{-1}$};
\end{tikzpicture}
\ \
\begin{tikzpicture}
\draw [color=pink] (-2,0)-- (2,0);
\draw [line width=1.2pt][color=red]  (-1.0, 0.5)-- (1.0,-0.5);
\draw [line width=1.2pt][color=blue] (1.0, 1.0)-- (-1.45,-1.45);
\fill (0,0) circle (1.5pt);
\draw (0,0.3) node {$\ip{3}{}$};
\draw [blue][->]  (0.+1.0,0.) arc (0:-135:1.0); \draw [blue] (0.+1.0,-0.2) node[right]{$\tgamma$};
\draw [black][->](-0.71,-0.71) arc (-135:-205:1.0); \draw [black] (0.-1.0,-0.2) node[left]{$\kappa$};
\draw [red][->] (-0.89,0.41) arc (-206:-360:1.0); \draw [red] (0,1.0) node[above]{$\tilde{a}_3^{-1}$};
\end{tikzpicture}
    \caption{Intersection indices}
    \label{fig:intind}
\end{figure}
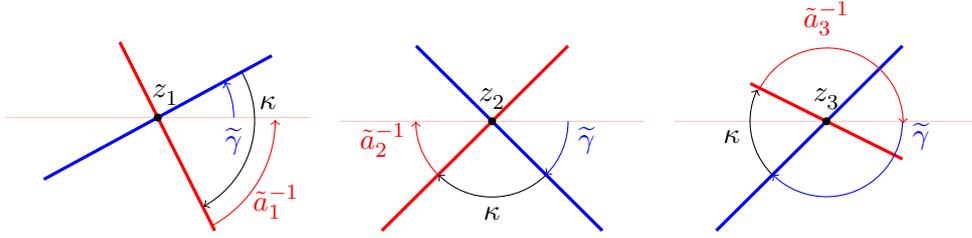

\end{example}

By \cite[Theorems~2.12 and 3.3]{OPS2018}, the marked surface $\SURF$ can be used to describe the category $D^b(\modcat \alg)$ in the following sense.

\begin{theorem}\label{thm:ops}
There is an injective map $X$ from the set $\GOCM(\SURF)$ of graded arcs on $\SURF$ whose endpoints are in $\Marked$ to the set of isoclasses of indecomposable complexes in $D^b(\modcat \alg)$, such that the following hold.
\begin{enumerate}
    \item The isoclasses of any indecomposable presilting complex in $D^b(\modcat \alg)$ belongs to the image of $X$.
    \item For any $\wc\in\GOCM(\SURF)$ and $d\in\mathbb{Z}$, we have $X(\wc[d])=X(\wc)[d]$.
    \item For any two $\wc_1,\wc_2\in\GOCM(\SURF)$ and $d\in\mathbb{Z}$, we have
    \[\dim_\kk \Hom_{D^b(\modcat \alg)} (X(\wc_1), X(\wc_2)[d]) = \Int^d (\wc_1, \wc_2).\]
\end{enumerate}
\end{theorem}

More generally, in \cite{OPS2018}, the above theorem holds for any gentle algebra, and the map $X$ is indeed a bijection if more curves with local system are added into the domain.

\begin{example}\label{ex:presil}
    Since $\Int^d (\tgamma, \tgamma)=0$ for any $d\neq 0$, by Theorem~\ref{thm:ops}, $X(\tgamma)$ is presilting.
\end{example}

By \cite[Corollary~2.28]{AI2012}, any silting complex in $D^b(\modcat \alg)$ has three non-isomorphic indecomposable summands. Hence the following proposition, which is proved at the end of the paper, implies Theorem~\ref{main}.

\begin{proposition}\label{prop:main}
Let $X(\tgamma)$ be the indecomposable presilting object in Example~\ref{ex:presil}. Then there is no indecomposable complex $S^\bullet$ in $D^b(\modcat \alg)$ such that $S^\bullet\ncong X(\tgamma)$ and $X(\tgamma)\oplus S^\bullet$ is presilting.
\end{proposition}

The proof of the above proposition consists of two steps: firstly we show that for any graded arc with self-intersection, it has a self-intersection with positive intersection index (see Proposition~\ref{prop:presilt.ind.obj.}); secondly we show that for any graded arc $\tsigma$ without self-intersection, either $\tsigma$ is homotopic to $\tgamma$, or there is an intersection between $\tsigma$ and $\tgamma$ with positive intersection index. These, together with the formula in Theorem~\ref{thm:ops}~(3), show the proposition.


\begin{remark}\label{rmk:app}
    Denote by $P(i)$ the indecomposable projective $\alg$-module corresponding to vertex $i$, $i=1,2,3$, respectively. By the construction of the map $X$ in \cite{OPS2018}, the complex $X(\tgamma)$ associated to $\tgamma$ is
    \[\cdots \to 0 \to P(3) \xrightarrow{f_{y_2}} P(2) \xrightarrow{f_{y_1}} P(1)\to 0\to \cdots,\]
    where $P(1)$ is on the degree 0 position, and $f_\alpha$ denotes the morphism between indecomposable projectives induced by arrow $\alpha$. This complex is isomorphic to the stalk complex formed by the representation $M = \xymatrix{\kk \ar@/^0.5pc/[r]^{1} \ar@/_0.5pc/[r]_{0} & \kk \ar@/^0.5pc/[r]^{0} \ar@/_0.5pc/[r]_{1} & \kk }$ in $\modcat \alg$. So $M$ satisfies conditions (T1) and (T2) in the introduction. Hence by Proposition~\ref{prop:main}, there is neither a $\alg$-module nor a complex $N\in D^b(\modcat \alg)$ such that $M\oplus N$ is tilting. Thus, our example is also a counter-example for Complement Question for tilting modules/complexes,
    which was first discovered by Kalck, see \cite[Proposition 1.4]{Kal2016}.
\end{remark}

\begin{remark}
    Take the dual graded arcs $\tgamma_1$ and $\tgamma_3$ in $\GOCM(\SURF)$ of $\walpha_1$ and $\walpha_3$ respectively with respect to $\dac$, that is, $\gamma_1$ (resp. $\gamma_3$) crosses only $a_1$ (resp. $a_3$) in $\dac$ and the intersection index is $0$, see Figure~\ref{fig:red}. Then $\tgamma,\tgamma_1,\tgamma_3$ form a full formal arc system in the sense of \cite{HKK2017}. So $D^b(\alg)$ is triangle equivalent to the perfect derived category $\operatorname{per}(\Gamma)$ of $\Gamma=\kk Q'/I'$, where
    \[Q' = \xymatrix{2' \ar@/^0.5pc/[rr]^{\alpha_1} \ar@/_0.5pc/[rr]_{\alpha_2} && 3'\ar[dl]^{\alpha_3}\\ & 1'\ar[ul]^{\alpha_4}},\ |\alpha_1|=|\alpha_2|=|\alpha_4|=0,\ |\alpha_3|=2,\ \text{and}\
    I' = \langle \alpha_3\alpha_4,\alpha_1\alpha_3,\alpha_4\alpha_2 \rangle.\]
    Here, the vertices $1',2',3'$ correspond to $\tgamma,\tgamma_1,\tgamma_3$, respectively. So under the triangle equivalence, $X(\tgamma)$ becomes $P(1')$. By \cite[Theorem~F]{CJS2022}, the silting reduction $\operatorname{per}(\Gamma)/\operatorname{thick}(P(1'))$ of $\operatorname{per}(\Gamma)$ with respect to the presilting complex $X(\tgamma)$ is triangle equivalent to the perfect derived category $\operatorname{per}(\Gamma_e)$, where $e$ stands for the idempotent of $\Gamma$ corresponding to vertex $1'$ and $\Gamma_e$ is the graded algebra $\kk Q''/I''$, where
    $$Q''=\xymatrix{3' \ar@/^1.5pc/[rr]^{\alpha_1} \ar@/_1pc/[rr]_{\alpha_2} && 2'\ar[ll]_{\beta}},\ |\alpha_1|=|\alpha_2|=0,\ |\beta|=1\ \text{and}\ I''=\langle\alpha_1\beta,\beta\alpha_2\rangle.$$
    Then by \cite[Theorem~2.37]{AI2012}, our Proposition~\ref{prop:main} is equivalent to that there is no indecomposable presilting object in $\operatorname{per}(\Gamma_e)$, which was claimed to be true at the end of \cite{CJS2022}, and after we completed and submitted this paper, was proven in \cite{JSW2023}.

\begin{figure}
\centering
\definecolor{ffqqqq}{rgb}{1,0,0}
\definecolor{qqwuqq}{rgb}{0,0,1}
\begin{tikzpicture} [scale=0.8]
\filldraw [black!20] (-2, 2) circle (0.5cm);
\filldraw [black!20] ( 2, 2) circle (0.5cm);
\filldraw [black!20] ( 2,-2) circle (0.5cm);
\filldraw [black!20] (-2,-2) circle (0.5cm);
\begin{scriptsize}
\draw [color=qqwuqq] (-1.8, 1.8) node{$q$}; \draw [color=qqwuqq] ( 1.8,-1.8) node{$p$};
\draw [color=ffqqqq] (-1.8,-1.8) node{$r$}; \draw [color=ffqqqq] ( 1.8, 1.8) node{$s$};
\end{scriptsize}
\draw [line width=1pt] (-2, 2) circle (0.5cm);
\draw [line width=1pt] ( 2, 2) circle (0.5cm);
\draw [line width=1pt] ( 2,-2) circle (0.5cm);
\draw [line width=1pt] (-2,-2) circle (0.5cm);
\draw [dash pattern=on 2pt off 2pt] (-2, 2)-- ( 2, 2);
\draw [dash pattern=on 2pt off 2pt] ( 2, 2)-- ( 2,-2);
\draw [dash pattern=on 2pt off 2pt] ( 2,-2)-- (-2,-2);
\draw [dash pattern=on 2pt off 2pt] (-2, 2)-- (-2,-2);
\draw [line width=1pt,color=ffqqqq] (-2.35,1.65)-- (-1.65,-1.65);
\draw [line width=1pt,color=ffqqqq] (-1.65,2.35)-- (1.65,1.65);
\draw [line width=1pt,color=ffqqqq] (1.65,1.65)-- (2.35,-1.65);
\draw [line width=1pt,color=ffqqqq] (1.65,-2.35)-- (-1.65,-1.65);
\draw [line width=1pt,color=ffqqqq] (-1.65,-1.65)-- (1.65,1.65);
\begin{scriptsize}
\fill [color=qqwuqq] (-1.65, 1.65) circle (2.5pt);
\fill [color=qqwuqq] ( 1.65,-1.65) circle (2.5pt);
\filldraw [color=white, draw=ffqqqq, line width=0.5pt] (-1.65, 2.35) circle (2.5pt); \draw [color=ffqqqq] (-1.65, 2.35) node[above]{$r$};
\fill [color=qqwuqq] (-2.35, 2.35) circle (2.5pt); \draw [color=qqwuqq] (-2.35, 2.35) node[left]{$p$};
\filldraw [color=white, draw=ffqqqq, line width=0.5pt] (-1.65,-1.65) circle (2.5pt);
\filldraw [color=white, draw=ffqqqq, line width=0.5pt] (-2.35,-2.35) circle (2.5pt); \draw [color=ffqqqq] (-2.35,-2.35) node[left]{$s$};
\fill [color=qqwuqq] (-2.35,-1.65) circle (2.5pt); \draw [color=qqwuqq] (-2.35,-1.65) node[left]{$p$};
\fill [color=qqwuqq] (-1.65,-2.35) circle (2.5pt); \draw [color=qqwuqq] (-1.65,-2.45) node[right]{$q$};
\filldraw [color=white, draw=ffqqqq, line width=0.5pt] (-2.35, 1.65) circle (2.5pt); \draw [color=ffqqqq] (-2.35, 1.65) node[left]{$s$};
\fill [color=qqwuqq] ( 1.65, 2.35) circle (2.5pt); \draw [color=qqwuqq] ( 1.65, 2.45) node[left]{$p$};
\filldraw [color=white, draw=ffqqqq, line width=0.5pt] ( 1.65, 1.65) circle (2.5pt);
\filldraw [color=white, draw=ffqqqq, line width=0.5pt] ( 2.35, 2.35) circle (2.5pt); \draw [color=ffqqqq] ( 2.35, 2.35) node[right]{$r$};
\fill [color=qqwuqq] ( 2.35, 1.65) circle (2.5pt); \draw [color=qqwuqq] ( 2.35, 1.65) node[right]{$q$};
\filldraw [color=white, draw=ffqqqq, line width=0.5pt] ( 2.35,-1.65) circle (2.5pt); \draw [color=ffqqqq] ( 2.35,-1.65) node[right]{$r$};
\fill [color=qqwuqq] ( 2.35,-2.35) circle (2.5pt); \draw [color=qqwuqq] ( 2.35,-2.35) node[right]{$q$};
\filldraw [color=white, draw=ffqqqq, line width=0.5pt] ( 1.65,-2.35) circle (2.5pt); \draw [color=ffqqqq] ( 1.65,-2.35) node[below]{$s$};
\end{scriptsize}
\draw [blue] ( 1.2, 1) node[right]{$\gamma$};
\draw [blue] (-2.8, 1) node[right]{$\gamma$};
\draw [blue] ( 1.2,-3) node[right]{$\gamma$};
\draw [blue] (-2.8,-3) node[right]{$\gamma$};
\draw ( 2.0,-0.6) node[right]{\color{red}$a_1$};
\draw (-2.0, 0.6) node[left]{\color{red}$a_1$};
\draw ( 0.0, 0.0) node[right]{\color{red}$a_2$};
\draw (-0.6, 2.0) node[above]{\color{red}$a_3$};
\draw ( 0.6,-2.0) node[below]{\color{red}$a_3$};
\draw [line width=1pt,color=blue] (1.65,2.35) to[out=-135,in=135] (1.35, 1.35) to[out=-45,in=-135] (2.35,1.65);
\draw [line width=1pt,color=blue] (1.65,2.35-3.98) to[out=-135,in=135] (1.35, 1.35-3.98) to[out=-45,in=-135] (2.35,1.65-3.98);
\draw [line width=1pt,color=blue] (1.65-3.98,2.35-3.98) to[out=-135,in=135] (1.35-3.98, 1.35-3.98) to[out=-45,in=-135] (2.35-3.98,1.65-3.98);
\draw [line width=1pt,color=blue] (1.65-3.98,2.35) to[out=-135,in=135] (1.35-3.98, 1.35) to[out=-45,in=-135] (2.35-3.98,1.65);
\draw [line width=1pt,color=blue] (-1.65, 1.65) -- (-2.35,-1.65);
\draw (-2.05,-0.6) node[left]{\color{blue}$\gamma_1$};
\draw [line width=1pt,color=blue] ( 2.35, 1.65) -- ( 1.65,-1.65);
\draw (2.05, 0.6) node[right]{\color{blue}$\gamma_1$};
\draw [line width=1pt,color=blue] (-1.65, 1.65) -- ( 1.65, 2.35);
\draw (0.6, 2.05) node[above]{\color{blue}$\gamma_3$};
\draw [line width=1pt,color=blue] (-1.65,-2.35) -- ( 1.65,-1.65);
\draw (-0.6,-2.05) node[below]{\color{blue}$\gamma_3$};
\end{tikzpicture}
\caption{A full formal arc system containing $\tgamma$}
\label{fig:red}
\end{figure}
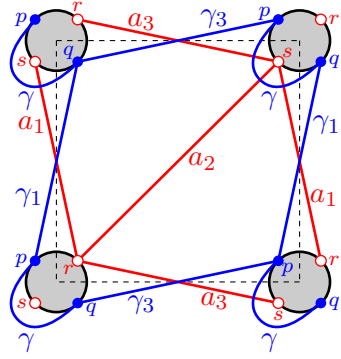
\end{remark}

\section{Simplifications}

An arc on $\SURF$ is called simple if it has no self-intersections except for its endpoints. Let $\SGOCM(\SURF)$ be the subset of $\GOCM(\SURF)$ consisting of the simple graded arcs on $\SURF$ whose endpoints are in $\Marked$.

\begin{notation}\label{not:1}
For any $\ws\in\SGOCM(\SURF)$, throughout the paper, we use the following notations. See \Pic\ref{fig:arc sgments}.
\begin{itemize}
    \item Denote by $\num{i}{\tsigma},1\leq i\leq n(\tsigma)$, the sequence of numbers in $\{1,2,3\}$ such that $\wa_{\num{i}{\tsigma}},1\leq i\leq n(\tsigma)$, are the graded arcs in $\dac=\{\wa_1,\wa_2,\wa_3\}$ that $\tsigma$ crosses in order, at the points $\ip{i}{\ws},1\leq i\leq n(\tsigma)$, with intersection indices $\dii{i}{\ws},1\leq i\leq n(\tsigma)$, respectively.
    \item For $i=0$ or $n(\ws)+1$, denote $\ip{0}{\ws}=\sigma(0)$ and $\ip{n(\ws)+1}{\ws}=\sigma(1)$.
    \item For $0\leq i<j\leq n(\ws)+1$, denote $\ws_{i, j}$ the segment of $\ws$ between $\ip{i}{\ws}$ and $\ip{j}{\ws}$.
\end{itemize}
When there is no confusion arising, we shall drop the subscript $\ws$ from the notations $\num{i}{\tsigma}$, $\ip{i}{\ws}$ and $\dii{i}{\ws}$.


\end{notation}

\begin{figure}[htpb]
\begin{center}
\definecolor{ffqqqq}{rgb}{1,0,0}
\definecolor{qqwuqq}{rgb}{0,0,1}
\definecolor{gradingcolor}{rgb}{0,0.5,0}
\begin{tikzpicture}[scale=1.3]
\draw[red][line width=1pt] (-3,-1) -- (-3, 1);
\draw[red] (-3,-0.8) node[right]{$\ta_{\num{1}{\ws}}$};
\filldraw[red] (-3,-1) circle (0.1cm);
\filldraw[white] (-3,-1) circle (0.07cm);
\filldraw[red] (-3, 1) circle (0.1cm);
\filldraw[white] (-3, 1) circle (0.07cm);
\draw[red][line width=1pt] (-1,-1) -- (-1, 1);
\draw[red] (-1,-0.8) node[right]{$\ta_{\num{2}{\ws}}$};
\filldraw[red] (-1,-1) circle (0.1cm);
\filldraw[white] (-1,-1) circle (0.07cm);
\filldraw[red] (-1, 1) circle (0.1cm);
\filldraw[white] (-1, 1) circle (0.07cm);
\draw[red][line width=1pt] ( 3,-1) -- ( 3, 1);
\draw[red] ( 3,-0.8) node[right]{$\ta_{\num{n(\ws)}{\ws}}$};
\filldraw[red] ( 3,-1) circle (0.1cm);
\filldraw[white] ( 3,-1) circle (0.07cm);
\filldraw[red] ( 3, 1) circle (0.1cm);
\filldraw[white] ( 3, 1) circle (0.07cm);
\draw[blue][line width=1pt] (-5,0) -- (5,0);
\filldraw[qqwuqq] (-5,0) circle (0.1cm);
\filldraw[qqwuqq] ( 5,0) circle (0.1cm);
\draw[blue][->][line width=1pt] (-1,0) -- (0,0);
\draw[blue] (-4,0) node[below]{$\tsigma_{0,1}$};
\draw[blue] (-2,0) node[below]{$\tsigma_{1,2}$};
\draw[blue] (1,0) node[below]{$\cdots$};
\draw[blue] (4,-0.3) node[right]{$\tsigma_{n(\tsigma),n(\tsigma)+1}$};
\draw[blue] (-5,0.2) node[right]{$\ip{0}{\ws}$};
\draw[blue] (-3,0) node{$\bm{\circ}$}; \draw[blue] (-3,0.2) node[right]{$\ip{1}{\ws}$};
\draw[blue] (-1,0) node{$\bm{\circ}$}; \draw[blue] (-1,0.2) node[right]{$\ip{2}{\ws}$};
\draw[blue] (1,0.2) node{$\cdots$};
\draw[blue] (3,0) node{$\bm{\circ}$}; \draw[blue] (3,0.2) node[right]{$\ip{n(\tsigma)}{\ws}$};
\draw[blue] (5,0.2) node[right]{$\ip{n(\ws)+1}{\ws}$};
\draw[gradingcolor][postaction={on each segment={mid arrow=gradingcolor}}] (-2.5,0) arc(0:-90:0.5);
\draw[gradingcolor] (-2.5,-0.4) node{$\dii{1}{\ws}$};
\draw[gradingcolor][postaction={on each segment={mid arrow=gradingcolor}}] (-0.5,0) arc(0:-90:0.5);
\draw[gradingcolor] (-0.5,-0.4) node{$\dii{2}{\ws}$};
\draw[gradingcolor][postaction={on each segment={mid arrow=gradingcolor}}] (3.5,0) arc(0:-90:0.5);
\draw[gradingcolor] (3.7,-0.4) node{$\dii{n(\ws)}{\ws}$};
\end{tikzpicture}
\caption{Notations for a graded arc $\ws$ in $\SGOCM(\SURF)$}
\label{fig:arc sgments}
\end{center}
\end{figure}

Easy calculations shows the following formulas for intersection indices, which are used frequently.

\begin{lemma}\label{lem:int ind}
    Let $\ws\in\SGOCM(\SURF)$. Then for any $1\leq i<n(\ws)$, we have
    $$\sdii{i+1}{\ws}=\begin{cases}
    \sdii{i}{\ws}+1 & \text{if $\snum{i}{\ws}<\snum{i+1}{\ws}$,}\\
    \sdii{i}{\ws}-1 & \text{if $\snum{i}{\ws}>\snum{i+1}{\ws}$.}
    \end{cases}
    $$
    Let $\ws,\ws'\in\SGOCM(\SURF)$ with $\ws(0)=\ws'(0)=z$. If $z$ is an oriented intersection from $\ws$ to $\ws'$, then $$\ii_z(\ws,\ws')=\dii{1}{\ws}-\dii{1}{\ws'}.$$
\end{lemma}



For any $\ws\in\SGOCM(\SURF)$ with $\sigma(0)=\sigma(1)$, there are two possible relative positions of the starting and the ending of $\ws$:
\begin{description}
    \item[(S.L.E.)] facing away $\sigma(0)=\sigma(1)$, the starting segment $\ws_{0,1}$ of $\ws$ is to the left of the ending segment $\ws_{n(\ws),n(\ws)+1}$ of $\ws$,
    \item[(S.R.E.)] facing away $\sigma(0)=\sigma(1)$, the starting segment $\ws_{0,1}$ of $\ws$ is to the right of the ending segment $\ws_{n(\ws),n(\ws)+1}$ of $\ws$.
\end{description}

For any $\ws\in\SGOCM(\SURF)$, we call $\ws$ \emph{contains a circle} provided that there is $1\leq i\leq n(\ws)-6$ such that $\snum{i}{\tsigma}=\snum{i+6}{\tsigma}$ and when gluing $\sip{i}{\ws}$ and $\sip{i+6}{\ws}$ together along $\wa_{\snum{i}{\tsigma}}$, $\tsigma_{i,i+6}$ becomes a circle around the boundary $\partial\Surf$.

\begin{lemma}\label{lem:circle}
For any $\ws\in\SGOCM(\SURF)$, there is a $\wt\in\SGOCM(\SURF)$ which does not contain a circle and such that the following hold.
\begin{enumerate}
    \item[(1)] $\sigma(0)=\tau(0)$ and $\sigma(1)=\tau(1)$.
    \item[(2)] $\dii{1}{\ws}=\dii{1}{\wt}$ and $\dii{n(\ws)}{\ws}=\dii{n(\wt)}{\wt}$.
    \item[(3)] If $\sigma(0)=\sigma(1)$ (and then $\tau(0)=\tau(1)$ by (1)), then the relative position of the starting and the ending of $\ws$ is the same as that of $\wt$.
\end{enumerate}
\end{lemma}

\begin{proof}
Let $\walpha_p$ (resp. $\walpha_q$) be the segment of $\partial\Surf$ between $p$ and $q$ and passing through $r$ (resp. $s$), with an arbitrary grading. Denote $\dii{0}{\ws}=\ii_{\ip{0}{\ws}}(\ws,\walpha_{\sigma(0)})$ and $\dii{n(\ws)+1}{\ws}=\ii_{\ip{n(\ws)+1}{\ws}}(\ws,\walpha_{\sigma(1)})$. Applying the second formula in Lemma~\ref{lem:int ind} to the pairs $(\ws,\walpha_{\sigma(0)})$ and $(\ws,\walpha_{\sigma(1)})$ of graded arcs, we get that condition (2) can be replaced with $\dii{0}{\ws}=\dii{0}{\wt}$ and $\dii{n(\ws)+1}{\ws}=\dii{n(\wt)+1}{\wt}$.

We define the circle number $\CCC^{\circlearrowright}(\ws)$ of $\ws$ from the starting point to be a rational number $\frac{a}{6}$, where
\begin{itemize}
    \item if $\snum{1}{\tsigma}=1$, then $a>0$ and $|a|$ is the maximal integer such that $\snum{i}{\tsigma}\equiv i\mod 3$ for any $1\leq i\leq |a|$,
    \item if $\snum{1}{\tsigma}=2$, then $a=0$, and
    \item if $\snum{1}{\tsigma}=3$, then $a<0$ and $|a|$ is the maximal integer such that $\snum{i}{\tsigma}\equiv 4-i\mod 3$ for any $1\leq i\leq |a|$.
\end{itemize}
Similarly, we define the circle number of $\ws$ from the ending point to be $\CCC^{\circlearrowright}(\underleftarrow{\ws})$, where $\underleftarrow{\ws}$ is the opposite direction of $\ws$. Then $\ws$ contains a circle if and only if at least one of $|\CCC^{\circlearrowright}(\ws)|$ and $|\CCC^{\circlearrowright}(\underleftarrow{\ws})|$ is bigger than $1$. Without loss of generality, we assume $|\CCC^{\circlearrowright}(\ws)|\geq|\CCC^{\circlearrowright}(\underleftarrow{\ws})|$ and $\CCC^{\circlearrowright}(\ws)>1$. So $\CCC^{\circlearrowright}(\ws)\geq\CCC^{\circlearrowright}(\underleftarrow{\ws})$.

By moving both endpoints of $\ws$ anti-clockwise around the boundary $\partial\Surf$ a lap and shifting the grading by $-2$, we obtain a graded arc $\ws'\in\SGOCM(\SURF)$ satisfying the following.
\begin{itemize}
    \item $\sigma(0)=\sigma'(0)$ and $\sigma(1)=\sigma'(1)$.
    \item $\dii{0}{\ws}=\dii{0}{\ws'}$ and $\dii{n(\ws)+1}{\ws}=\dii{n(\ws')+1}{\ws'}$.
    \item If $\sigma(0)=\sigma(1)$, then the relative position of the starting and the ending of $\ws$ is the same as that of $\ws'$.
    \item $\CCC^{\circlearrowright}(\ws')=\CCC^{\circlearrowright}(\ws)-1$ and $\CCC^{\circlearrowright}(\underleftarrow{\ws'})=\CCC^{\circlearrowright}(\underleftarrow{\ws})-1$.
\end{itemize}
Thus, using the induction, we can get a graded arc $\wt$ satisfying conditions (1)-(3), and $0<\CCC^{\circlearrowright}(\wt)\leq 1$ and $\CCC^{\circlearrowright}(\wt)\geq\CCC^{\circlearrowright}(\underleftarrow{\wt})$. Note that in this case, we have $\CCC^{\circlearrowright}(\wt)-\CCC^{\circlearrowright}(\underleftarrow{\wt})\leq 1$, see Figure~\ref{fig in rmk-encircling}.
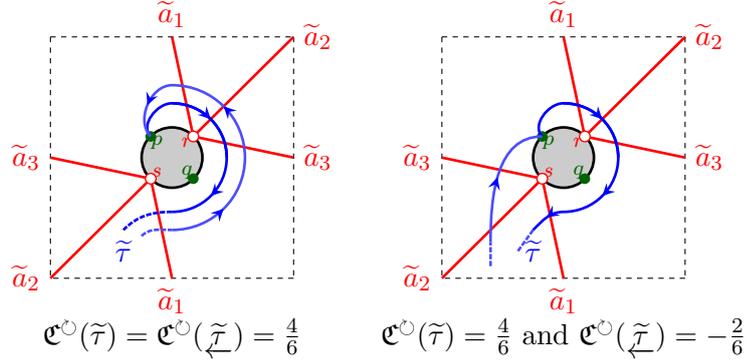
\begin{figure}[htpb]
\begin{center}
\definecolor{ffqqqq}{rgb}{1,0,0}
\definecolor{qqwuqq}{rgb}{0,0.39,0}
\begin{tikzpicture}[scale=0.8]
\filldraw [black!20] (0,0) circle (0.5cm); \draw [line width=1pt] (0, 0) circle (0.5cm);
\draw [dash pattern=on 2pt off 2pt] (-2, 2)-- ( 2, 2);
\draw [dash pattern=on 2pt off 2pt] ( 2, 2)-- ( 2,-2);
\draw [dash pattern=on 2pt off 2pt] ( 2,-2)-- (-2,-2);
\draw [dash pattern=on 2pt off 2pt] (-2, 2)-- (-2,-2);
\begin{scriptsize}
\draw [color=qqwuqq] (-0.25, 0.25) node{$p$}; \draw [color=qqwuqq] ( 0.25,-0.25) node{$q$};
\draw [color=ffqqqq] (-0.25,-0.25) node{$s$}; \draw [color=ffqqqq] ( 0.25, 0.25) node{$r$};
\end{scriptsize}
\draw [color=ffqqqq][line width=1pt] ( 0.35, 0.35)-- ( 0, 2); 
\draw [color=ffqqqq][line width=1pt] (-0.35,-0.35)-- ( 0,-2); 
\draw [color=ffqqqq][line width=1pt] ( 0.35, 0.35)-- ( 2, 2); 
\draw [color=ffqqqq][line width=1pt] (-0.35,-0.35)-- (-2,-2); 
\draw [color=ffqqqq][line width=1pt] ( 0.35, 0.35)-- ( 2, 0); 
\draw [color=ffqqqq][line width=1pt] (-0.35,-0.35)-- (-2, 0); 
\draw [ffqqqq] ( 0, 2) node[above]{$\ta_1$};
\draw [ffqqqq] ( 0,-2) node[below]{$\ta_1$};
\draw [ffqqqq] ( 2, 2) node[right]{$\ta_2$};
\draw [ffqqqq] (-2,-2) node[left]{$\ta_2$};
\draw [ffqqqq] ( 2, 0) node[right]{$\ta_3$};
\draw [ffqqqq] (-2, 0) node[left]{$\ta_3$};
\begin{scriptsize}
\fill [color=qqwuqq] (-0.35, 0.35) circle (2.5pt);
\fill [color=qqwuqq] ( 0.35,-0.35) circle (2.5pt);
\filldraw [color=white, draw=ffqqqq, line width=0.5pt] ( 0.35, 0.35) circle (2.5pt);
\filldraw [color=white, draw=ffqqqq, line width=0.5pt] (-0.35,-0.35) circle (2.5pt);
\end{scriptsize}
\draw [line width=1pt,color=blue]
  (-0.35, 0.35) to[out= 135, in= 180] ( 0.00, 0.90) to[out=   0, in=  90] ( 0.90, 0.00);
\draw [line width=1pt,color=blue] [postaction={on each segment={mid arrow=blue}}]
  ( 0.00, 0.90) to[out=   0, in=  90] ( 0.90, 0.00) to[out= -90, in=   0] ( 0.00,-0.90);
\draw [line width=1pt,color=blue] [dash pattern=on 2pt off 0.6pt]
  ( 0.00,-0.90) to[out= 180, in=  45] (-0.80,-1.20) node[below]{$\wt$};
\draw [line width=1pt,color=blue!70] [dash pattern=on 2pt off 0.6pt]
  (-0.50,-1.30) to[out=  45, in= 180] ( 0.00,-1.20);
\draw [line width=1pt,color=blue!70] [postaction={on each segment={mid arrow=blue}}]
  ( 0.00,-1.20) to[out=   0, in= -90] ( 1.20, 0.00)
                to[out=  90, in=   0] ( 0.00, 1.20) to[out= 180, in= 135] (-0.35, 0.35);
\draw (0,-3) node{$\CCC^{\circlearrowright}(\wt) =
\CCC^{\circlearrowright}(\underleftarrow{\wt}) = \frac{4}{6}$};
\end{tikzpicture}
\ \
\begin{tikzpicture}[scale=0.8]
\filldraw [black!20] (0,0) circle (0.5cm); \draw [line width=1pt] (0, 0) circle (0.5cm);
\draw [dash pattern=on 2pt off 2pt] (-2, 2)-- ( 2, 2);
\draw [dash pattern=on 2pt off 2pt] ( 2, 2)-- ( 2,-2);
\draw [dash pattern=on 2pt off 2pt] ( 2,-2)-- (-2,-2);
\draw [dash pattern=on 2pt off 2pt] (-2, 2)-- (-2,-2);
\begin{scriptsize}
\draw [color=qqwuqq] (-0.25, 0.25) node{$p$}; \draw [color=qqwuqq] ( 0.25,-0.25) node{$q$};
\draw [color=ffqqqq] (-0.25,-0.25) node{$s$}; \draw [color=ffqqqq] ( 0.25, 0.25) node{$r$};
\end{scriptsize}
\draw [color=ffqqqq][line width=1pt] ( 0.35, 0.35) -- ( 0, 2); 
\draw [color=ffqqqq][line width=1pt] (-0.35,-0.35) -- ( 0,-2); 
\draw [color=ffqqqq][line width=1pt] ( 0.35, 0.35)-- ( 2, 2); 
\draw [color=ffqqqq][line width=1pt] (-0.35,-0.35)-- (-2,-2); 
\draw [color=ffqqqq][line width=1pt] ( 0.35, 0.35)-- ( 2, 0); 
\draw [color=ffqqqq][line width=1pt] (-0.35,-0.35)-- (-2, 0); 
\draw [ffqqqq] ( 0, 2) node[above]{$\ta_1$};
\draw [ffqqqq] ( 0,-2) node[below]{$\ta_1$};
\draw [ffqqqq] ( 2, 2) node[right]{$\ta_2$};
\draw [ffqqqq] (-2,-2) node[left]{$\ta_2$};
\draw [ffqqqq] ( 2, 0) node[right]{$\ta_3$};
\draw [ffqqqq] (-2, 0) node[left]{$\ta_3$};
\begin{scriptsize}
\fill [color=qqwuqq] (-0.35, 0.35) circle (2.5pt);
\fill [color=qqwuqq] ( 0.35,-0.35) circle (2.5pt);
\filldraw [color=white, draw=ffqqqq, line width=0.5pt] ( 0.35, 0.35) circle (2.5pt);
\filldraw [color=white, draw=ffqqqq, line width=0.5pt] (-0.35,-0.35) circle (2.5pt);
\end{scriptsize}
\draw [line width=1pt,color=blue]
      (-0.35, 0.35) to[out= 135, in= 180] ( 0.00, 0.90) to[out=   0, in=  90] ( 0.90, 0.00);
\draw [line width=1pt,color=blue] [postaction={on each segment={mid arrow=blue}}]
      ( 0.00, 0.90) to[out=   0, in=  90] ( 0.90, 0.00) to[out= -90, in=   0] ( 0.00,-0.90)
                    to[out= 180, in=  45] (-0.50,-1.20) node[below]{$\wt$};
\draw [line width=1pt,color=blue!70] [dash pattern=on 2pt off 0.6pt]
      (-0.50,-1.20) -- (-0.75,-1.55);
\draw [line width=1pt,color=blue!70] 
      [postaction={on each segment={mid arrow=blue}}]
      (-1.20,-1.40) to[out=  90, in= 180] (-0.35, 0.35);
\draw [line width=1pt,color=blue!70] [dash pattern=on 2pt off 0.6pt]
      (-1.20,-1.80) -- (-1.20,-1.40); 
\draw (0,-3) node{$\CCC^{\circlearrowright}(\wt) = \frac{4}{6}$ and $
\CCC^{\circlearrowright}(\underleftarrow{\wt}) = -\frac{2}{6}$};
\end{tikzpicture}
\caption{Circle numbers}
\label{fig in rmk-encircling}
\end{center}
\end{figure}
So $-1<\CCC^{\circlearrowright}(\underleftarrow{\wt})\leq 1$. Hence $\wt$ does not contain a circle as required.
\end{proof}

A simple graded arc without a circle has the following property.

\begin{lemma}\label{lem:four}
For any $\ws\in\SGOCM(\SURF)$ which does not contain a circle, at least one type in the left/right picture of \Pic\ref{fig:four} does not appear as a segment of $\ws$.
\end{lemma}

\begin{figure}[htpb]
\begin{center}
\definecolor{ffqqqq}{rgb}{1,0,0}
\definecolor{qqwuqq}{rgb}{0,0,1}
\begin{tikzpicture} [scale=0.8]
\filldraw [black!20] (-2, 2) circle (0.5cm);
\filldraw [black!20] ( 2, 2) circle (0.5cm);
\filldraw [black!20] ( 2,-2) circle (0.5cm);
\filldraw [black!20] (-2,-2) circle (0.5cm);
\draw [color=qqwuqq] (-1.8, 1.8) node{$q$}; \draw [color=qqwuqq] ( 1.8,-1.8) node{$p$};
\draw [color=ffqqqq] (-1.8,-1.8) node{$r$}; \draw [color=ffqqqq] ( 1.8, 1.8) node{$s$};
\draw [line width=1pt] (-2, 2) circle (0.5cm);
\draw [line width=1pt] ( 2, 2) circle (0.5cm);
\draw [line width=1pt] ( 2,-2) circle (0.5cm);
\draw [line width=1pt] (-2,-2) circle (0.5cm);
\draw [dash pattern=on 2pt off 2pt] (-2, 2)-- ( 2, 2);
\draw [dash pattern=on 2pt off 2pt] ( 2, 2)-- ( 2,-2);
\draw [dash pattern=on 2pt off 2pt] ( 2,-2)-- (-2,-2);
\draw [dash pattern=on 2pt off 2pt] (-2, 2)-- (-2,-2);
\draw [line width=1pt,color=ffqqqq] (-2.35,1.65)-- (-1.65,-1.65);
\draw [line width=1pt,color=ffqqqq] (-1.65,2.35)-- (1.65,1.65);
\draw [line width=1pt,color=ffqqqq] (1.65,1.65)-- (2.35,-1.65);
\draw [line width=1pt,color=ffqqqq] (1.65,-2.35)-- (-1.65,-1.65);
\draw [line width=1pt,color=ffqqqq] (-1.65,-1.65)-- (1.65,1.65);
\fill [color=qqwuqq] (-1.65, 1.65) circle (2.5pt);
\fill [color=qqwuqq] ( 1.65,-1.65) circle (2.5pt);
\filldraw [color=white, draw=ffqqqq, line width=0.5pt] (-1.65, 2.35) circle (2.5pt); \draw [color=ffqqqq] (-1.65, 2.35) node[above]{$r$};
\fill [color=qqwuqq] (-2.35, 2.35) circle (2.5pt); \draw [color=qqwuqq] (-2.35, 2.35) node[left]{$p$};
\filldraw [color=white, draw=ffqqqq, line width=0.5pt] (-1.65,-1.65) circle (2.5pt);
\filldraw [color=white, draw=ffqqqq, line width=0.5pt] (-2.35,-2.35) circle (2.5pt); \draw [color=ffqqqq] (-2.35,-2.35) node[left]{$s$};
\fill [color=qqwuqq] (-2.35,-1.65) circle (2.5pt); \draw [color=qqwuqq] (-2.35,-1.65) node[left]{$p$};
\fill [color=qqwuqq] (-1.65,-2.35) circle (2.5pt); \draw [color=qqwuqq] (-1.65,-2.35) node[right]{$q$};
\filldraw [color=white, draw=ffqqqq, line width=0.5pt] (-2.35, 1.65) circle (2.5pt); \draw [color=ffqqqq] (-2.35, 1.65) node[left]{$s$};
\fill [color=qqwuqq] ( 1.65, 2.35) circle (2.5pt); \draw [color=qqwuqq] ( 1.65, 2.35) node[left]{$p$};
\filldraw [color=white, draw=ffqqqq, line width=0.5pt] ( 1.65, 1.65) circle (2.5pt);
\filldraw [color=white, draw=ffqqqq, line width=0.5pt] ( 2.35, 2.35) circle (2.5pt); \draw [color=ffqqqq] ( 2.35, 2.35) node[right]{$r$};
\fill [color=qqwuqq] ( 2.35, 1.65) circle (2.5pt); \draw [color=qqwuqq] ( 2.35, 1.65) node[right]{$q$};
\filldraw [color=white, draw=ffqqqq, line width=0.5pt] ( 2.35,-1.65) circle (2.5pt); \draw [color=ffqqqq] ( 2.35,-1.65) node[right]{$r$};
\fill [color=qqwuqq] ( 2.35,-2.35) circle (2.5pt); \draw [color=qqwuqq] ( 2.35,-2.35) node[right]{$q$};
\filldraw [color=white, draw=ffqqqq, line width=0.5pt] ( 1.65,-2.35) circle (2.5pt); \draw [color=ffqqqq] ( 1.65,-2.35) node[below]{$s$};
%
\draw [line width=1pt,color=blue] [postaction={on each segment={mid arrow=blue}}]
      (2.1,-0.4) to[out=180,in=90] (0.4,-2.1);
\draw [blue] (0.87,-0.87) node[left]{\small (i)};
\draw [line width=1pt,color=blue] [postaction={on each segment={mid arrow=blue}}]
      (0.5,1.9) to[out=-90,in=135] (0.98,0.98);
\draw [blue] (0.63,1.38) node[left]{\small (ii)};
\draw [line width=1pt,color=blue] [postaction={on each segment={mid arrow=blue}}]
      (1.21,1.21) to[out=-45,in=180] (1.8,0.98);
\draw [blue] (1.5,0.99) node[below]{\small (iii)};
\draw [line width=1pt,color=blue] [postaction={on each segment={mid arrow=blue}}]
      (-2.1,0.4) to[out=0,in=-90] (-0.4,2.1);
\draw [blue] (-0.87,0.87) node[right]{\small (iv)};
\draw [line width=1pt,color=blue] [postaction={on each segment={mid arrow=blue}}]
      (-0.5,-1.9) to[out=90,in=-45] (-0.98,-0.98);
\draw [blue] (-0.63,-1.38) node[right]{\small (v)};
\draw [line width=1pt,color=blue] [postaction={on each segment={mid arrow=blue}}]
      (-1.21,-1.21) to[out=135,in=0] (-1.8,-0.98);
\draw [blue] (-1.49,-0.99) node[above]{\small (vi)};
\draw (0,-3) node{(i) -- (vi)};
\end{tikzpicture}
\ \ \ \
\begin{tikzpicture} [scale=0.8]
\filldraw [black!20] (-2, 2) circle (0.5cm);
\filldraw [black!20] ( 2, 2) circle (0.5cm);
\filldraw [black!20] ( 2,-2) circle (0.5cm);
\filldraw [black!20] (-2,-2) circle (0.5cm);
\draw [color=qqwuqq] (-1.8, 1.8) node{$q$}; \draw [color=qqwuqq] ( 1.8,-1.8) node{$p$};
\draw [color=ffqqqq] (-1.8,-1.8) node{$r$}; \draw [color=ffqqqq] ( 1.8, 1.8) node{$s$};
\draw [line width=1pt] (-2, 2) circle (0.5cm);
\draw [line width=1pt] ( 2, 2) circle (0.5cm);
\draw [line width=1pt] ( 2,-2) circle (0.5cm);
\draw [line width=1pt] (-2,-2) circle (0.5cm);
\draw [dash pattern=on 2pt off 2pt] (-2, 2)-- ( 2, 2);
\draw [dash pattern=on 2pt off 2pt] ( 2, 2)-- ( 2,-2);
\draw [dash pattern=on 2pt off 2pt] ( 2,-2)-- (-2,-2);
\draw [dash pattern=on 2pt off 2pt] (-2, 2)-- (-2,-2);
\draw [line width=1pt,color=ffqqqq] (-2.35,1.65)-- (-1.65,-1.65);
\draw [line width=1pt,color=ffqqqq] (-1.65,2.35)-- (1.65,1.65);
\draw [line width=1pt,color=ffqqqq] (1.65,1.65)-- (2.35,-1.65);
\draw [line width=1pt,color=ffqqqq] (1.65,-2.35)-- (-1.65,-1.65);
\draw [line width=1pt,color=ffqqqq] (-1.65,-1.65)-- (1.65,1.65);
\fill [color=qqwuqq] (-1.65, 1.65) circle (2.5pt);
\fill [color=qqwuqq] ( 1.65,-1.65) circle (2.5pt);
\filldraw [color=white, draw=ffqqqq, line width=0.5pt] (-1.65, 2.35) circle (2.5pt); \draw [color=ffqqqq] (-1.65, 2.35) node[above]{$r$};
\fill [color=qqwuqq] (-2.35, 2.35) circle (2.5pt); \draw [color=qqwuqq] (-2.35, 2.35) node[left]{$p$};
\filldraw [color=white, draw=ffqqqq, line width=0.5pt] (-1.65,-1.65) circle (2.5pt);
\filldraw [color=white, draw=ffqqqq, line width=0.5pt] (-2.35,-2.35) circle (2.5pt); \draw [color=ffqqqq] (-2.35,-2.35) node[left]{$s$};
\fill [color=qqwuqq] (-2.35,-1.65) circle (2.5pt); \draw [color=qqwuqq] (-2.35,-1.65) node[left]{$p$};
\fill [color=qqwuqq] (-1.65,-2.35) circle (2.5pt); \draw [color=qqwuqq] (-1.65,-2.35) node[right]{$q$};
\filldraw [color=white, draw=ffqqqq, line width=0.5pt] (-2.35, 1.65) circle (2.5pt); \draw [color=ffqqqq] (-2.35, 1.65) node[left]{$s$};
\fill [color=qqwuqq] ( 1.65, 2.35) circle (2.5pt); \draw [color=qqwuqq] ( 1.65, 2.35) node[left]{$p$};
\filldraw [color=white, draw=ffqqqq, line width=0.5pt] ( 1.65, 1.65) circle (2.5pt);
\filldraw [color=white, draw=ffqqqq, line width=0.5pt] ( 2.35, 2.35) circle (2.5pt); \draw [color=ffqqqq] ( 2.35, 2.35) node[right]{$r$};
\fill [color=qqwuqq] ( 2.35, 1.65) circle (2.5pt); \draw [color=qqwuqq] ( 2.35, 1.65) node[right]{$q$};
\filldraw [color=white, draw=ffqqqq, line width=0.5pt] ( 2.35,-1.65) circle (2.5pt); \draw [color=ffqqqq] ( 2.35,-1.65) node[right]{$r$};
\fill [color=qqwuqq] ( 2.35,-2.35) circle (2.5pt); \draw [color=qqwuqq] ( 2.35,-2.35) node[right]{$q$};
\filldraw [color=white, draw=ffqqqq, line width=0.5pt] ( 1.65,-2.35) circle (2.5pt); \draw [color=ffqqqq] ( 1.65,-2.35) node[below]{$s$};
\draw [line width=1pt,color=blue] [postaction={on each segment={mid arrow=blue}}]
      (0.4,-2.1) to[out=90,in=180] (2.1,-0.4);
\draw [blue] (0.87,-0.87) node[left]{\small (i')};
\draw [line width=1pt,color=blue] [postaction={on each segment={mid arrow=blue}}]
      (0.98,0.98) to[out=135,in=-90] (0.5,1.9);
\draw [blue] (0.63,1.38) node[left]{\small (ii')};
\draw [line width=1pt,color=blue] [postaction={on each segment={mid arrow=blue}}]
      (1.8,0.98) to[out=180,in=-45] (1.21,1.21);
\draw [blue] (1.42,0.99) node[below]{\small (iii')};
\draw [line width=1pt,color=blue] [postaction={on each segment={mid arrow=blue}}]
      (-0.4,2.1) to[out=-90,in=  0] (-2.1,0.4);
\draw [blue] (-0.87,0.87) node[right]{\small (iv')};
\draw [line width=1pt,color=blue] [postaction={on each segment={mid arrow=blue}}]
      (-0.98,-0.98) to[out=-45,in=90] (-0.5,-1.9);
\draw [blue] (-0.63,-1.38) node[right]{\small (v')};
\draw [line width=1pt,color=blue] [postaction={on each segment={mid arrow=blue}}]
      (-1.8,-0.98) to[out=0,in=135] (-1.21,-1.21);
\draw [blue] (-1.42,-0.99) node[above]{\small (vi')};
\draw (0,-3) node{(i') -- (vi')};
\end{tikzpicture}
\caption{Types of segments}
\label{fig:four}
\end{center}
\end{figure}
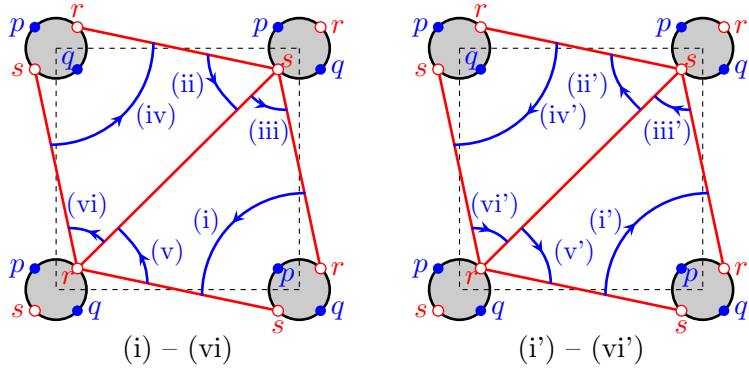

\begin{proof}
By symmetry, we only need to show the assertion for the left picture. Suppose conversely that there are $1\leq i_k\leq n(\ws)$, $1\leq k\leq 6$ such that $\ws_{i_1,i_1+1}$, $\ws_{i_2,i_2+1}$, $\ws_{i_3,i_3+1}$, $\ws_{i_4,i_4+1}$, $\ws_{i_5,i_5+1}$, $\ws_{i_6,i_6+1}$, are of types (i)-(vi), respectively.

By symmetry, we may assume $\sigma(0)=p$. We divide the proof into the following cases.

\begin{enumerate}
    \item[(1)] $\snum{1}{\tsigma}=1$, and either $\sigma(1)\neq\sigma(0)$, or $\sigma(1)=\sigma(0)$ with (S.R.E.). If $\sip{1}{\ws}$ is between $s$ and $\sip{i_6+1}{\ws}$, see the first picture of \Pic\ref{fig:four1}, then $i_6+1$ has to be $n(\ws)$, and hence $\sigma(1)=p$ and (S.L.E.), a contradiction. Hence $\sip{1}{\ws}$ is between $r$ and $\sip{i_6+1}{\ws}$, see the second picture of \Pic\ref{fig:four1}. Since $\ws_{1,2}$ does not cross $\ws_{i_6,i_6+1}$, we have $\snum{2}{\ws}=2$, i.e. $\ws_{1,2}$ is of type (vi'). Similarly, we have that $\ws_{2,3}$, $\ws_{3,4}$, $\ws_{4,5}$, $\ws_{5,6}$ and $\ws_{6,7}$ are of types (v')-(i') respectively, see the third picture of \Pic\ref{fig:four1} (where the relative position of $\sip{i_5+1}{\ws}$ and $\sip{i_6}{\ws}$ and that of $\sip{i_2+1}{\ws}$ and $\sip{i_3}{\ws}$ may change). But this implies that $\ws$ contains a circle $\ws_{1,7}$, a contradiction.

    \item[(2)] $\snum{1}{\tsigma}=1$, $\sigma(1)=\sigma(0)$ and (S.L.E.). Then $l_{n(\ws)}=1$. A similar argument shows that $\ws_{n(\ws)-1,n(\ws)}$, $\ws_{n(\ws)-2,n(\ws)-1}$, $\ws_{n(\ws)-3,n(\ws)-2}$, $\ws_{n(\ws)-4,n(\ws)-3}$, $\ws_{n(\ws)-5,n(\ws)-4}$ and $\ws_{n(\ws)-6,n(\ws)-5}$ are of types (vi)-(i) respectively (and may coincide $\ws_{i_6,i_6+1}$, $\ws_{i_5,i_5+1}$, $\ws_{i_4,i_4+1}$, $\ws_{i_3,i_3+1}$, $\ws_{i_2,i_2+1}$, $\ws_{i_1,i_1+1}$,  respectively), which form a circle $\ws_{n(\ws)-6,n(\ws)}$, see the fourth picture of \Pic\ref{fig:four1}, a contradiction.
    \item[(3)] $\snum{1}{\tsigma}=2$. Then type (i) does not appear as a segment of $\ws$. So we are done.
    \item[(4)] $\snum{1}{\tsigma}=3$. This case is similar to the case $\snum{1}{\tsigma}=1$ (i.e. (1) and (2) above), so we omit the proof.
\end{enumerate}

\begin{figure}[htpb]
\begin{center}
\definecolor{ffqqqq}{rgb}{1,0,0}
\definecolor{qqwuqq}{rgb}{0,0,1}
\begin{tikzpicture}[scale = 1.2]
\filldraw [black!20] (-2, 2) circle (0.5cm);
\filldraw [black!20] ( 2, 2) circle (0.5cm);
\filldraw [black!20] ( 2,-2) circle (0.5cm);
\filldraw [black!20] (-2,-2) circle (0.5cm);
\draw [color=qqwuqq] (-1.8, 1.8) node{$q$}; \draw [color=qqwuqq] ( 1.8,-1.8) node{$p$};
\draw [color=ffqqqq] (-1.8,-1.8) node{$r$}; \draw [color=ffqqqq] ( 1.8, 1.8) node{$s$};
\draw [line width=1pt] (-2, 2) circle (0.5cm);
\draw [line width=1pt] ( 2, 2) circle (0.5cm);
\draw [line width=1pt] ( 2,-2) circle (0.5cm);
\draw [line width=1pt] (-2,-2) circle (0.5cm);
\draw [dash pattern=on 2pt off 2pt] (-2, 2)-- ( 2, 2);
\draw [dash pattern=on 2pt off 2pt] ( 2, 2)-- ( 2,-2);
\draw [dash pattern=on 2pt off 2pt] ( 2,-2)-- (-2,-2);
\draw [dash pattern=on 2pt off 2pt] (-2, 2)-- (-2,-2);
\draw [line width=1pt,color=ffqqqq] (-2.35,1.65)-- (-1.65,-1.65);
\draw [line width=1pt,color=ffqqqq] (-1.65,2.35)-- (1.65,1.65);
\draw [line width=1pt,color=ffqqqq] (1.65,1.65)-- (2.35,-1.65);
\draw [line width=1pt,color=ffqqqq] (1.65,-2.35)-- (-1.65,-1.65);
\draw [line width=1pt,color=ffqqqq] (-1.65,-1.65)-- (1.65,1.65);
\fill [color=qqwuqq] (-1.65, 1.65) circle (2.5pt);
\fill [color=qqwuqq] ( 1.65,-1.65) circle (2.5pt);
\filldraw [color=white, draw=ffqqqq, line width=0.5pt] (-1.65, 2.35) circle (2.5pt);
\fill [color=qqwuqq] (-2.35, 2.35) circle (2.5pt);
\filldraw [color=white, draw=ffqqqq, line width=0.5pt] (-1.65,-1.65) circle (2.5pt);
\filldraw [color=white, draw=ffqqqq, line width=0.5pt] (-2.35,-2.35) circle (2.5pt);
\fill [color=qqwuqq] (-2.35,-1.65) circle (2.5pt);
\fill [color=qqwuqq] (-1.65,-2.35) circle (2.5pt);
\filldraw [color=white, draw=ffqqqq, line width=0.5pt] (-2.35, 1.65) circle (2.5pt);
\fill [color=qqwuqq] ( 1.65, 2.35) circle (2.5pt);
\filldraw [color=white, draw=ffqqqq, line width=0.5pt] ( 1.65, 1.65) circle (2.5pt);
\filldraw [color=white, draw=ffqqqq, line width=0.5pt] ( 2.35, 2.35) circle (2.5pt);
\fill [color=qqwuqq] ( 2.35, 1.65) circle (2.5pt);
\filldraw [color=white, draw=ffqqqq, line width=0.5pt] ( 2.35,-1.65) circle (2.5pt);
\fill [color=qqwuqq] ( 2.35,-2.35) circle (2.5pt);
\filldraw [color=white, draw=ffqqqq, line width=0.5pt] ( 1.65,-2.35) circle (2.5pt);
\draw [color=qqwuqq] (-1.8, 1.8) node{$q$};
\draw [color=ffqqqq] (-1.8,-1.8) node{$r$};
\draw [color=ffqqqq] (-1.65, 2.35) node[above]{$r$};
\draw [color=ffqqqq] ( 2.35,-1.65) node[right]{$r$};
\draw [color=ffqqqq] ( 1.8, 1.8) node{$s$};
\draw [color=ffqqqq] ( 1.65,-2.35) node[below]{$s$};
\draw [color=ffqqqq] (-2.35, 1.65) node[left]{$s$};
\draw [black!50][dash pattern=on 2pt off 2pt] (-2, 0) -- (2, 0);
\draw [line width=1pt,color=blue][postaction={on each segment={mid arrow=blue}}]
      (1.65,-1.65) to[out=90,in=180] (2, 0) node[right]{$\sip{1}{\ws}$};
\draw [line width=1pt,color=blue][->] (-2, 0) node[left]{$\sip{1}{\ws}$} to[out=0,in=200] (-0.2,0.2);
\draw [line width=1pt,color=cyan] [postaction={on each segment={mid arrow=cyan}}]
      (-1.21,-1.21) node[right]{$\sip{i_6}{\ws}$} to[out=135,in=0]
      (-2,-0.98) node[left]{$\sip{i_6+1}{\ws}$};
\draw [cyan][->] (-0.5,-0.79) node[right]{$\ws_{i_6, i_6\!+\!1}$}
      -- (-1.09,-0.79) -- (-1.29,-0.99);
\draw (1.8,0.4) node[right]{\small\color{red}$a_1$};
\draw (-1.8,-0.4) node[left]{\small\color{red}$a_1$};
\draw (0,0) node[right]{\small\color{red}$a_2$};
\draw (0,2.1) node[above]{\small\color{red}$a_3$};
\draw (0,-2.1) node[below]{\small\color{red}$a_3$};
\end{tikzpicture}
\ \ \ \ \
\begin{tikzpicture}[scale = 1.2]
\filldraw [black!20] (-2, 2) circle (0.5cm);
\filldraw [black!20] ( 2, 2) circle (0.5cm);
\filldraw [black!20] ( 2,-2) circle (0.5cm);
\filldraw [black!20] (-2,-2) circle (0.5cm);
\begin{scriptsize}
\draw [color=qqwuqq] (-1.8, 1.8) node{$q$}; \draw [color=qqwuqq] ( 1.8,-1.8) node{$p$};
\draw [color=ffqqqq] (-1.8,-1.8) node{$r$}; \draw [color=ffqqqq] ( 1.8, 1.8) node{$s$};
\end{scriptsize}
\draw [line width=1pt] (-2, 2) circle (0.5cm);
\draw [line width=1pt] ( 2, 2) circle (0.5cm);
\draw [line width=1pt] ( 2,-2) circle (0.5cm);
\draw [line width=1pt] (-2,-2) circle (0.5cm);
\draw [dash pattern=on 2pt off 2pt] (-2, 2)-- ( 2, 2);
\draw [dash pattern=on 2pt off 2pt] ( 2, 2)-- ( 2,-2);
\draw [dash pattern=on 2pt off 2pt] ( 2,-2)-- (-2,-2);
\draw [dash pattern=on 2pt off 2pt] (-2, 2)-- (-2,-2);
\draw [line width=1pt,color=ffqqqq] (-2.35,1.65)-- (-1.65,-1.65);
\draw [line width=1pt,color=ffqqqq] (-1.65,2.35)-- (1.65,1.65);
\draw [line width=1pt,color=ffqqqq] (1.65,1.65)-- (2.35,-1.65);
\draw [line width=1pt,color=ffqqqq] (1.65,-2.35)-- (-1.65,-1.65);
\draw [line width=1pt,color=ffqqqq] (-1.65,-1.65)-- (1.65,1.65);
\fill [color=qqwuqq] (-1.65, 1.65) circle (2.5pt);
\fill [color=qqwuqq] ( 1.65,-1.65) circle (2.5pt);
\filldraw [color=white, draw=ffqqqq, line width=0.5pt] (-1.65, 2.35) circle (2.5pt);
\fill [color=qqwuqq] (-2.35, 2.35) circle (2.5pt);
\filldraw [color=white, draw=ffqqqq, line width=0.5pt] (-1.65,-1.65) circle (2.5pt);
\filldraw [color=white, draw=ffqqqq, line width=0.5pt] (-2.35,-2.35) circle (2.5pt);
\fill [color=qqwuqq] (-2.35,-1.65) circle (2.5pt);
\fill [color=qqwuqq] (-1.65,-2.35) circle (2.5pt);
\filldraw [color=white, draw=ffqqqq, line width=0.5pt] (-2.35, 1.65) circle (2.5pt);
\fill [color=qqwuqq] ( 1.65, 2.35) circle (2.5pt);
\filldraw [color=white, draw=ffqqqq, line width=0.5pt] ( 1.65, 1.65) circle (2.5pt);
\filldraw [color=white, draw=ffqqqq, line width=0.5pt] ( 2.35, 2.35) circle (2.5pt);
\fill [color=qqwuqq] ( 2.35, 1.65) circle (2.5pt);
\filldraw [color=white, draw=ffqqqq, line width=0.5pt] ( 2.35,-1.65) circle (2.5pt);
\fill [color=qqwuqq] ( 2.35,-2.35) circle (2.5pt);
\filldraw [color=white, draw=ffqqqq, line width=0.5pt] ( 1.65,-2.35) circle (2.5pt);
\draw [color=qqwuqq] (-1.8, 1.8) node{$q$};
\draw [color=ffqqqq] (-1.8,-1.8) node{$r$};
\draw [color=ffqqqq] (-1.65, 2.35) node[above]{$r$};
\draw [color=ffqqqq] ( 2.35,-1.65) node[right]{$r$};
\draw [color=ffqqqq] ( 1.8, 1.8) node{$s$};
\draw [color=ffqqqq] ( 1.65,-2.35) node[below]{$s$};
\draw [color=ffqqqq] (-2.35, 1.65) node[left]{$s$};
\draw [black!50][dash pattern=on 2pt off 2pt] (-2,-0.4) -- (2,-0.4);
\draw [line width=1pt,color=blue][postaction={on each segment={mid arrow=blue}}]
      (1.65,-1.65) to[out=90,in=180] (2,-0.4) node[right]{$\sip{1}{\ws}$};
\draw [line width=1pt,color=blue][postaction={on each segment={mid arrow=blue}}]
      (-2,-0.4) node[left]{$\sip{1}{\ws}$} to[out=0,in=134] (-0.9,-0.9);
\draw[blue] (-0.9+0.25,-0.9-0.25) node{$\sip{2}{\ws}$};
\draw [line width=1pt,color=cyan][postaction={on each segment={mid arrow=cyan}}]
      (-0.65,-0.65) to[out=135,in=0] (-2,-0) node[left]{$\sip{i_6+1}{\ws}$};
\draw [cyan] (-0.65+0.25,-0.65-0.25) node{$\sip{i_6}{\ws}$};
\draw (1.8,0.4) node[right]{\small\color{red}$a_1$};
\draw (-2.0,0.4) node[left]{\small\color{red}$a_1$};
\draw (0,0) node[right]{\small\color{red}$a_2$};
\draw (0,2.1) node[above]{\small\color{red}$a_3$};
\draw (0,-2.1) node[below]{\small\color{red}$a_3$};
\end{tikzpicture}
\\
\begin{tikzpicture}[scale = 1.2]
\filldraw [black!20] (-2, 2) circle (0.5cm);
\filldraw [black!20] ( 2, 2) circle (0.5cm);
\filldraw [black!20] ( 2,-2) circle (0.5cm);
\filldraw [black!20] (-2,-2) circle (0.5cm);
\begin{scriptsize}
\draw [color=qqwuqq] (-1.8, 1.8) node{$q$}; \draw [color=qqwuqq] ( 1.8,-1.8) node{$p$};
\draw [color=ffqqqq] (-1.8,-1.8) node{$r$}; \draw [color=ffqqqq] ( 1.8, 1.8) node{$s$};
\end{scriptsize}
\draw [line width=1pt] (-2, 2) circle (0.5cm);
\draw [line width=1pt] ( 2, 2) circle (0.5cm);
\draw [line width=1pt] ( 2,-2) circle (0.5cm);
\draw [line width=1pt] (-2,-2) circle (0.5cm);
\draw [dash pattern=on 2pt off 2pt] (-2, 2)-- ( 2, 2);
\draw [dash pattern=on 2pt off 2pt] ( 2, 2)-- ( 2,-2);
\draw [dash pattern=on 2pt off 2pt] ( 2,-2)-- (-2,-2);
\draw [dash pattern=on 2pt off 2pt] (-2, 2)-- (-2,-2);
\draw [line width=1pt,color=ffqqqq] (-2.35,1.65)-- (-1.65,-1.65);
\draw [line width=1pt,color=ffqqqq] (-1.65,2.35)-- (1.65,1.65);
\draw [line width=1pt,color=ffqqqq] (1.65,1.65)-- (2.35,-1.65);
\draw [line width=1pt,color=ffqqqq] (1.65,-2.35)-- (-1.65,-1.65);
\draw [line width=1pt,color=ffqqqq] (-1.65,-1.65)-- (1.65,1.65);
\fill [color=qqwuqq] (-1.65, 1.65) circle (2.5pt);
\fill [color=qqwuqq] ( 1.65,-1.65) circle (2.5pt);
\filldraw [color=white, draw=ffqqqq, line width=0.5pt] (-1.65, 2.35) circle (2.5pt);
\fill [color=qqwuqq] (-2.35, 2.35) circle (2.5pt);
\filldraw [color=white, draw=ffqqqq, line width=0.5pt] (-1.65,-1.65) circle (2.5pt);
\filldraw [color=white, draw=ffqqqq, line width=0.5pt] (-2.35,-2.35) circle (2.5pt);
\fill [color=qqwuqq] (-2.35,-1.65) circle (2.5pt);
\fill [color=qqwuqq] (-1.65,-2.35) circle (2.5pt);
\filldraw [color=white, draw=ffqqqq, line width=0.5pt] (-2.35, 1.65) circle (2.5pt);
\fill [color=qqwuqq] ( 1.65, 2.35) circle (2.5pt);
\filldraw [color=white, draw=ffqqqq, line width=0.5pt] ( 1.65, 1.65) circle (2.5pt);
\filldraw [color=white, draw=ffqqqq, line width=0.5pt] ( 2.35, 2.35) circle (2.5pt);
\fill [color=qqwuqq] ( 2.35, 1.65) circle (2.5pt);
\filldraw [color=white, draw=ffqqqq, line width=0.5pt] ( 2.35,-1.65) circle (2.5pt);
\fill [color=qqwuqq] ( 2.35,-2.35) circle (2.5pt);
\filldraw [color=white, draw=ffqqqq, line width=0.5pt] ( 1.65,-2.35) circle (2.5pt);
\draw [color=qqwuqq] (-1.8, 1.8) node{$q$};
\draw [color=ffqqqq] (-1.8,-1.8) node{$r$};
\draw [color=ffqqqq] (-1.65, 2.35) node[above]{$r$};
\draw [color=ffqqqq] ( 2.35,-1.65) node[right]{$r$};
\draw [color=ffqqqq] ( 1.8, 1.8) node{$s$};
\draw [color=ffqqqq] ( 1.65,-2.35) node[below]{$s$};
\draw [color=ffqqqq] (-2.35, 1.65) node[left]{$s$};
\draw [line width=1pt,color=blue][postaction={on each segment={mid arrow=blue}}]
      (1.65,-1.65) to[out=90,in=180] (2.2,-1.0) node[right]{$\sip{1}{\ws}$};
\draw [black!50][dash pattern=on 2pt off 2pt] (2.2,-1.0) -- (-2,-1.0);
\draw [line width=1pt,color=blue][postaction={on each segment={mid arrow=blue}}]
      (-1.8,-1.0) node[left]{$\sip{1}{\ws}$} to[out=0, in=135] (-1.2, -1.2);
\draw [blue] (-1.25,-1.25) node[above]{$\sip{2}{\ws}$};
\draw [line width=1pt,color=blue][postaction={on each segment={mid arrow=blue}}]
      (-1.2, -1.2) to[out=-45, in=90] (-1.0, -1.8) node[below]{$\sip{3}{\ws}$};
\draw [black!50][dash pattern=on 2pt off 2pt] (-1.0, -2) -- (-1.0, 2);
\draw [line width=1pt,color=blue][postaction={on each segment={mid arrow=blue}}]
      (-1.0, 2.2) node[above]{$\sip{3}{\ws}$} to[out=-90,in=0]
      (-2.2, 1.0) node[left]{$\sip{4}{\ws}$};
\draw [black!50][dash pattern=on 2pt off 2pt] (-2, 1.0) -- ( 2, 1.0);
\draw [line width=1pt,color=blue][postaction={on each segment={mid arrow=blue}}]
      (1.8, 1.0) node[right]{$\sip{4}{\ws}$} to[out=180, in=-45] (1.2, 1.2);
\draw [blue] ( 1.25, 1.25) node[below]{$\sip{5}{\ws}$};
\draw [line width=1pt,color=blue][postaction={on each segment={mid arrow=blue}}]
      (1.2, 1.2) to[out=135, in=-90] (1.0, 1.8) node[above]{$\sip{6}{\ws}$};
\draw [black!50][dash pattern=on 2pt off 2pt] (1.0, 2.0) -- (1.0,-2.0);
\draw [line width=1pt,color=blue][postaction={on each segment={mid arrow=blue}}]
      (1.0,-2.2) node[below]{$\sip{6}{\ws}$} to[out=90, in=180]
      (2.1, -0.5)  node[right]{$\sip{7}{\ws}$};
\draw [cyan][line width=1pt] [postaction={on each segment={mid arrow=cyan}}]
      (2., -0.2) node[right]{$\sip{i_1}{\ws}$} to[out=180, in=90] (0.2,-2.1);
\draw [cyan][->] (0.2,-2.15) -- (0.2,-2.4); \draw[cyan] (0.2,-2.25) node[below]{$\sip{i_1\!+1}{\ws}$};
\draw [cyan][line width=1pt] [postaction={on each segment={mid arrow=cyan}}]
      (0, 2)  node[above]{$\sip{i_2}{\ws}$} to[out=-90,in=135]
      (0.4, 0.4); \draw[cyan] (0.4+0.4, 0.4-0.2) node{$\sip{i_2\!+1}{\ws}$};
\draw [cyan][line width=1pt] [postaction={on each segment={mid arrow=cyan}}]
      (0.67, 0.67) node[above]{$\sip{i_3}{\ws}$} to[out=-45,in=180]
      (1.95, 0.38) node[right]{$\sip{i_3\!+1}{\ws}$};
\draw [cyan][line width=1pt] [postaction={on each segment={mid arrow=cyan}}]
      (-2.1, 0.4) node[left]{$\sip{i_4}{\ws}$} to[out=0, in=-90] (-0.4, 2.1);
\draw [cyan][->] (-0.4, 2.15) -- (-0.4, 2.5); \draw[cyan] (-0.4, 2.4) node[above]{$\sip{i_4\!+1}{\ws}$};
\draw [cyan][line width=1pt] [postaction={on each segment={mid arrow=cyan}}]
      (-0.67,-0.67) to[out=135,in=0] (-1.95,-0.38) node[left]{$\sip{i_6\!+1}{\ws}$};
\draw [cyan] (-0.67+0.25,-0.67-0.25) node{$\sip{i_6}{\ws}$};
\draw [cyan][line width=1pt] [postaction={on each segment={mid arrow=cyan}}]
      (0,-2) to[out= 90,in=-45] (-0.4,-0.4);
\draw[cyan] (0-0.3,-2+0.3) node{$\sip{i_5}{\ws}$};
\draw[cyan] (-0.4-0.2, -0.4) node[above]{$\sip{i_5\!+1}{\ws}$};
\draw ( 2.0, 0.1) node[right]{\small\color{red}$a_1$};
\draw (-2.0, 0.0) node[left]{\small\color{red}$a_1$};
\draw (0,0) node[right]{\small\color{red}$a_2$};
\draw (0.5, 1.9) node[above]{\small\color{red}$a_3$};
\draw (-0.35,-1.9) node[below]{\small\color{red}$a_3$};
\end{tikzpicture}
\ \
\begin{tikzpicture}[scale = 1.2]
\filldraw [black!20] (-2, 2) circle (0.5cm);
\filldraw [black!20] ( 2, 2) circle (0.5cm);
\filldraw [black!20] ( 2,-2) circle (0.5cm);
\filldraw [black!20] (-2,-2) circle (0.5cm);
\begin{scriptsize}
\draw [color=qqwuqq] (-1.8, 1.8) node{$q$}; \draw [color=qqwuqq] ( 1.8,-1.8) node{$p$};
\draw [color=ffqqqq] (-1.8,-1.8) node{$r$}; \draw [color=ffqqqq] ( 1.8, 1.8) node{$s$};
\end{scriptsize}
\draw [line width=1pt] (-2, 2) circle (0.5cm);
\draw [line width=1pt] ( 2, 2) circle (0.5cm);
\draw [line width=1pt] ( 2,-2) circle (0.5cm);
\draw [line width=1pt] (-2,-2) circle (0.5cm);
\draw [dash pattern=on 2pt off 2pt] (-2, 2)-- ( 2, 2);
\draw [dash pattern=on 2pt off 2pt] ( 2, 2)-- ( 2,-2);
\draw [dash pattern=on 2pt off 2pt] ( 2,-2)-- (-2,-2);
\draw [dash pattern=on 2pt off 2pt] (-2, 2)-- (-2,-2);
\draw [line width=1pt,color=ffqqqq] (-2.35,1.65)-- (-1.65,-1.65);
\draw [line width=1pt,color=ffqqqq] (-1.65,2.35)-- (1.65,1.65);
\draw [line width=1pt,color=ffqqqq] (1.65,1.65)-- (2.35,-1.65);
\draw [line width=1pt,color=ffqqqq] (1.65,-2.35)-- (-1.65,-1.65);
\draw [line width=1pt,color=ffqqqq] (-1.65,-1.65)-- (1.65,1.65);
\fill [color=qqwuqq] (-1.65, 1.65) circle (2.5pt);
\fill [color=qqwuqq] ( 1.65,-1.65) circle (2.5pt);
\filldraw [color=white, draw=ffqqqq, line width=0.5pt] (-1.65, 2.35) circle (2.5pt);
\fill [color=qqwuqq] (-2.35, 2.35) circle (2.5pt);
\filldraw [color=white, draw=ffqqqq, line width=0.5pt] (-1.65,-1.65) circle (2.5pt);
\filldraw [color=white, draw=ffqqqq, line width=0.5pt] (-2.35,-2.35) circle (2.5pt);
\fill [color=qqwuqq] (-2.35,-1.65) circle (2.5pt);
\fill [color=qqwuqq] (-1.65,-2.35) circle (2.5pt);
\filldraw [color=white, draw=ffqqqq, line width=0.5pt] (-2.35, 1.65) circle (2.5pt);
\fill [color=qqwuqq] ( 1.65, 2.35) circle (2.5pt);
\filldraw [color=white, draw=ffqqqq, line width=0.5pt] ( 1.65, 1.65) circle (2.5pt);
\filldraw [color=white, draw=ffqqqq, line width=0.5pt] ( 2.35, 2.35) circle (2.5pt);
\fill [color=qqwuqq] ( 2.35, 1.65) circle (2.5pt);
\filldraw [color=white, draw=ffqqqq, line width=0.5pt] ( 2.35,-1.65) circle (2.5pt);
\fill [color=qqwuqq] ( 2.35,-2.35) circle (2.5pt);
\filldraw [color=white, draw=ffqqqq, line width=0.5pt] ( 1.65,-2.35) circle (2.5pt);
\draw [color=qqwuqq] (-1.8, 1.8) node{$q$};
\draw [color=ffqqqq] (-1.8,-1.8) node{$r$};
\draw [color=ffqqqq] (-1.65, 2.35) node[above]{$r$};
\draw [color=ffqqqq] ( 2.35,-1.65) node[right]{$r$};
\draw [color=ffqqqq] ( 1.8, 1.8) node{$s$};
\draw [color=ffqqqq] ( 1.65,-2.35) node[below]{$s$};
\draw [color=ffqqqq] (-2.35, 1.65) node[left]{$s$};
\draw [line width=1pt,color=blue][postaction={on each segment={mid arrow=blue}}]
      (2.2,-1.0) node[right]{$\sip{n(\ws)}{\ws}$} to[out=180,in=90] (1.65,-1.65);
\draw [black!50][dash pattern=on 2pt off 2pt] (2.2,-1.0) -- (-2,-1.0);
\draw [line width=1pt,color=blue][postaction={on each segment={mid arrow=blue}}]
      (-1.2, -1.2) to[out=135, in=0] (-1.8,-1.0) node[left]{$\sip{n(\ws)}{\ws}$};
\draw [blue] (-1.3,-1.3) node[right]{$\sip{n\!(\ws\!)\!-\!1}{\ws}$};
\draw [line width=1pt,color=blue][postaction={on each segment={mid arrow=blue}}]
      (-1.0, -1.8) node[below]{$\sip{n\!(\ws\!)\!-\!2}{\ws}$} to[out=90, in=-45] (-1.2, -1.2);
\draw [black!50][dash pattern=on 2pt off 2pt] (-1.0, -2) -- (-1.0, 2);
\draw [line width=1pt,color=blue][postaction={on each segment={mid arrow=blue}}]
      (-2.2, 1.0) node[left]{$\sip{n\!(\!\ws\!)\!-\!3}{\ws}$} to[out=0,in=-90]
      (-1.0, 2.2) node[above]{$\sip{n\!(\!\ws\!)\!-\!2}{\ws}$};
\draw [black!50][dash pattern=on 2pt off 2pt] (-2, 1.0) -- ( 2, 1.0);
\draw [line width=1pt,color=blue][postaction={on each segment={mid arrow=blue}}]
      (1.2, 1.2) to[out=-45, in=180] (1.8, 1.0) node[right]{$\sip{n\!(\!\ws\!)\!-\!3}{\ws}$};
\draw [blue] ( 1.25, 1.25) node[left]{$\sip{n\!(\!\ws\!)\!-\!4}{\ws}$};
\draw [line width=1pt,color=blue][postaction={on each segment={mid arrow=blue}}]
      (1.0, 1.8) node[above]{$\sip{n\!(\!\ws\!)\!-\!5}{\ws}$} to[out=-90, in=135] (1.2, 1.2);
\draw [black!50][dash pattern=on 2pt off 2pt] (1.0, 2.0) -- (1.0,-2.0);
\draw [line width=1pt,color=blue][postaction={on each segment={mid arrow=blue}}]
      (2.1, -0.5)  node[right]{$\sip{n\!(\!\ws\!)\!-\!6}{\ws}$} to[out=180, in=90]
      (1.0,-2.2) node[below]{$\sip{n\!(\!\ws\!)\!-\!5}{\ws}$};
\draw [cyan][line width=1pt] [postaction={on each segment={mid arrow=cyan}}]
      (2., -0.2) node[right]{$\sip{i_1}{\ws}$} to[out=180, in=90] (0.2,-2.1);
\draw [cyan][->] (0.2,-2.15) -- (0.2,-2.4); \draw[cyan] (0.2,-2.25) node[below]{$\sip{i_1\!+1}{\ws}$};
\draw [cyan][line width=1pt] [postaction={on each segment={mid arrow=cyan}}]
      (0, 2)  node[above]{$\sip{i_2}{\ws}$} to[out=-90,in=135]
      (0.4, 0.4); \draw[cyan] (0.4+0.4, 0.4-0.1) node{$\sip{i_2\!+1}{\ws}$};
\draw [cyan][line width=1pt] [postaction={on each segment={mid arrow=cyan}}]
      (0.67, 0.67) to[out=-45,in=180]
      (1.95, 0.38) node[right]{$\sip{i_3\!+1}{\ws}$}; \draw[cyan] (0.6,0.9) node{$\sip{i_3}{\ws}$};
\draw [cyan][line width=1pt] [postaction={on each segment={mid arrow=cyan}}]
      (-2.1, 0.4) node[left]{$\sip{i_4}{\ws}$} to[out=0, in=-90] (-0.4, 2.1);
\draw [cyan][->] (-0.4, 2.15) -- (-0.4, 2.5); \draw[cyan] (-0.4, 2.4) node[above]{$\sip{i_4\!+1}{\ws}$};
\draw [cyan][line width=1pt] [postaction={on each segment={mid arrow=cyan}}]
      (-0.67,-0.67) to[out=135,in=0] (-1.95,-0.38) node[left]{$\sip{i_6\!+1}{\ws}$};
\draw [cyan] (-0.67+0.25,-0.67-0.25) node{$\sip{i_6}{\ws}$};
\draw [cyan][line width=1pt] [postaction={on each segment={mid arrow=cyan}}]
      (0,-2) to[out= 90,in=-45] (-0.4,-0.4);
\draw[cyan] (0-0.3,-2+0.3) node{$\sip{i_5}{\ws}$};
\draw[cyan] (-0.4-0.2, -0.4) node[above]{$\sip{i_5\!+1}{\ws}$};
\draw[blue][line width=1pt] [postaction={on each segment={mid arrow=blue}}]
     (1.65,-1.65) to[out=100, in=180] (2.15,-0.7) node[right]{$\sip{1}{\ws}$};
\draw ( 2.0, 0.1) node[right]{\small\color{red}$a_1$};
\draw (-2.0, 0.0) node[left]{\small\color{red}$a_1$};
\draw (0,0) node[right]{\small\color{red}$a_2$};
\draw (0.4, 1.9) node[above]{\small\color{red}$a_3$};
\draw (-0.35,-1.9) node[below]{\small\color{red}$a_3$};
\end{tikzpicture}
\end{center}
\caption{Cases in the proof of Lemma~\ref{lem:four}}
\label{fig:four1}
\end{figure}

\end{proof}

For any $\ws\in\SGOCM(\SURF)$, an \emph{admissible sequence} of $\ws$ is a sequence $0=i_1<i_2<\cdots<i_t=n(\ws)$ such that for any $1\leq k\leq t-2$,
\begin{itemize}
    \item[(A1)] $\ws_{i_k,i_k+1}$ and $\ws_{i_{k+1},i_{k+1}+1}$ are in different quadrilaterals divided by $\dac$,
    \item[(A2)] $\snum{i_k+1}{\tsigma}=\snum{i_{k+1}}{\tsigma}$ (i.e. the endpoint of $\ws_{i_k,i_k+1}$ is in the same arc in $\dac$ as the starting point of $\ws_{i_{k+1},i_{k+1}+1}$), and
    \item[(A3)] $\sdii{i_k+1}{\ws}=\sdii{i_{k+1}}{\ws}$.
\end{itemize}
A \emph{simplest sequence} of $\ws$ is an admissible sequence whose any nontrivial subsequence is not admissible.

\begin{lemma}\label{lem:cancel}
Let $\ws\in\SGOCM(\SURF)$ and $0=i_1<i_2<\cdots<i_t=n(\ws)$ a simplest sequence of $\ws$. Then for any $1< k<t-1$, the three numbers $\snum{i_k}{\ws},\snum{i_{k+1}}{\ws},\snum{i_{k+2}}{\ws}$ are different from each other.
\end{lemma}

\begin{proof}
Conversely assume (at least) two of $\snum{i_k}{\ws},\snum{i_{k+1}}{\ws},\snum{i_{k+2}}{\ws}$ are the same. By (\textrm{A2}), we have $\snum{i_{k+1}}{\ws}=\snum{i_k+1}{\tsigma}\neq\snum{i_k}{\ws}$ and $\snum{i_{k+2}}{\ws}=\snum{i_{k+1}+1}{\tsigma}\neq\snum{i_{k+1}}{\ws}$. So $\snum{i_k}{\ws}=\snum{i_{k+2}}{\ws}\neq \snum{i_{k+1}}{\ws}$. Then by the first formula in Lemma~\ref{lem:int ind} and (\textrm{A3}), we have $\sdii{i_k}{\ws}=\sdii{i_{k+2}}{\ws}$. Thus, one can remove $i_k$ and $i_{k+1}$ from the simplest sequence to get an admissible subsequence, a contradiction.
\end{proof}

The following lemma rules out some possibilities of the form of a simplest sequence.

\begin{lemma}\label{lemm:self-int Type}
Let $\tsigma\in\SGOCM(\SURF)$. If $\tsigma$ does not contain a circle, then any simplest sequence of $\tsigma$ is not of any form shown in \Pic \ref{fig:self-int Type}, where the dotted lines express the order of  segments.

\begin{figure}[htpb]
\centering
\definecolor{ffqqqq}{rgb}{1,0,0}
\definecolor{qqwuqq}{rgb}{0,0,1}
\begin{tikzpicture}[scale=0.55]
\filldraw [black!20] (-2, 2) circle (0.5cm);
\filldraw [black!20] ( 2, 2) circle (0.5cm);
\filldraw [black!20] ( 2,-2) circle (0.5cm);
\filldraw [black!20] (-2,-2) circle (0.5cm);
\draw [line width=1pt] (-2, 2) circle (0.5cm);
\draw [line width=1pt] ( 2, 2) circle (0.5cm);
\draw [line width=1pt] ( 2,-2) circle (0.5cm);
\draw [line width=1pt] (-2,-2) circle (0.5cm);
\draw [dash pattern=on 2pt off 2pt] (-2, 2)-- ( 2, 2);
\draw [dash pattern=on 2pt off 2pt] ( 2, 2)-- ( 2,-2);
\draw [dash pattern=on 2pt off 2pt] ( 2,-2)-- (-2,-2);
\draw [dash pattern=on 2pt off 2pt] (-2, 2)-- (-2,-2);
\draw [line width=1pt,color=ffqqqq] (-2.35,1.65)-- (-1.65,-1.65);
\draw [line width=1pt,color=ffqqqq] (-1.65,2.35)-- (1.65,1.65);
\draw [line width=1pt,color=ffqqqq] (1.65,1.65)-- (2.35,-1.65);
\draw [line width=1pt,color=ffqqqq] (1.65,-2.35)-- (-1.65,-1.65);
\draw [line width=1pt,color=ffqqqq] (-1.65,-1.65)-- (1.65,1.65);
\fill [color=qqwuqq] (-1.65, 1.65) circle (2.5pt);
\fill [color=qqwuqq] ( 1.65,-1.65) circle (2.5pt);
\filldraw [color=white, draw=ffqqqq, line width=0.5pt] (-1.65, 2.35) circle (2.5pt);
\fill [color=qqwuqq] (-2.35, 2.35) circle (2.5pt);
\filldraw [color=white, draw=ffqqqq, line width=0.5pt] (-1.65,-1.65) circle (2.5pt);
\filldraw [color=white, draw=ffqqqq, line width=0.5pt] (-2.35,-2.35) circle (2.5pt);
\fill [color=qqwuqq] (-2.35,-1.65) circle (2.5pt);
\fill [color=qqwuqq] (-1.65,-2.35) circle (2.5pt);
\filldraw [color=white, draw=ffqqqq, line width=0.5pt] (-2.35, 1.65) circle (2.5pt);
\fill [color=qqwuqq] ( 1.65, 2.35) circle (2.5pt);
\filldraw [color=white, draw=ffqqqq, line width=0.5pt] ( 1.65, 1.65) circle (2.5pt);
\filldraw [color=white, draw=ffqqqq, line width=0.5pt] ( 2.35, 2.35) circle (2.5pt);
\fill [color=qqwuqq] ( 2.35, 1.65) circle (2.5pt);
\filldraw [color=white, draw=ffqqqq, line width=0.5pt] ( 2.35,-1.65) circle (2.5pt);
\fill [color=qqwuqq] ( 2.35,-2.35) circle (2.5pt);
\filldraw [color=white, draw=ffqqqq, line width=0.5pt] ( 1.65,-2.35) circle (2.5pt);
\draw [color=qqwuqq] ( 1.8,-1.8) node{$p$};
\draw [color=qqwuqq] (-1.8, 1.8) node{$q$};
\draw [color=ffqqqq] (-1.8,-1.8) node{$r$};
\draw [color=ffqqqq] (-1.65, 2.35) node[above]{$r$};
\draw [color=ffqqqq] ( 2.35,-1.65) node[right]{$r$};
\draw [color=ffqqqq] ( 1.8, 1.8) node{$s$};
\draw [color=ffqqqq] ( 1.65,-2.35) node[below]{$s$};
\draw [color=ffqqqq] (-2.35, 1.65) node[left]{$s$};
\draw [line width=1pt, color=cyan] (1.65,-1.65) to[out=180, in=180] (2,1);
\draw [color=cyan, line width=0.7pt, dotted] (2,1) to[out=180, in=0] (-2,-0.2);
\draw [line width=1pt, color=cyan] (-2,-0.2) to[out=0, in=135] (-0.75,-0.75);
\draw [line width=1pt, color=cyan] (-0.65,-0.65) to[out=-45, in=90] (0,-2);
\draw [color=cyan, line width=0.7pt, dotted] (0,-2) to[out=90, in=-90] (-0.2,2);
\draw [line width=1pt,color=cyan] (-0.2,2) to[out=-90, in=0] (-2,0.2);
\draw [color=cyan, line width=0.7pt, dotted] (-2,0.2) to[out=0, in=180] (2,-0.6);
\draw [line width=1pt,color=cyan] plot[smooth, tension=0.1] (2,-0.6) to[out=180, in=90] (1.65,-1.65);
\draw ( 2.0, 0.0) node[right]{\small\color{red}$a_1$};
\draw (-2.0, 0.0) node[left]{\small\color{red}$a_1$};
\draw (0,0) node[right]{\small\color{red}$a_2$};
\draw ( 0.0, 2.0) node[above]{\small\color{red}$a_3$};
\draw (-0.0,-2.0) node[below]{\small\color{red}$a_3$};
\end{tikzpicture}
\begin{tikzpicture}[scale=0.55]
\filldraw [black!20] (-2, 2) circle (0.5cm);
\filldraw [black!20] ( 2, 2) circle (0.5cm);
\filldraw [black!20] ( 2,-2) circle (0.5cm);
\filldraw [black!20] (-2,-2) circle (0.5cm);
\draw [line width=1pt] (-2, 2) circle (0.5cm);
\draw [line width=1pt] ( 2, 2) circle (0.5cm);
\draw [line width=1pt] ( 2,-2) circle (0.5cm);
\draw [line width=1pt] (-2,-2) circle (0.5cm);
\draw [dash pattern=on 2pt off 2pt] (-2, 2)-- ( 2, 2);
\draw [dash pattern=on 2pt off 2pt] ( 2, 2)-- ( 2,-2);
\draw [dash pattern=on 2pt off 2pt] ( 2,-2)-- (-2,-2);
\draw [dash pattern=on 2pt off 2pt] (-2, 2)-- (-2,-2);
\draw [line width=1pt,color=ffqqqq] (-2.35,1.65)-- (-1.65,-1.65);
\draw [line width=1pt,color=ffqqqq] (-1.65,2.35)-- (1.65,1.65);
\draw [line width=1pt,color=ffqqqq] (1.65,1.65)-- (2.35,-1.65);
\draw [line width=1pt,color=ffqqqq] (1.65,-2.35)-- (-1.65,-1.65);
\draw [line width=1pt,color=ffqqqq] (-1.65,-1.65)-- (1.65,1.65);
\fill [color=qqwuqq] (-1.65, 1.65) circle (2.5pt);
\fill [color=qqwuqq] ( 1.65,-1.65) circle (2.5pt);
\filldraw [color=white, draw=ffqqqq, line width=0.5pt] (-1.65, 2.35) circle (2.5pt);
\fill [color=qqwuqq] (-2.35, 2.35) circle (2.5pt);
\filldraw [color=white, draw=ffqqqq, line width=0.5pt] (-1.65,-1.65) circle (2.5pt);
\filldraw [color=white, draw=ffqqqq, line width=0.5pt] (-2.35,-2.35) circle (2.5pt);
\fill [color=qqwuqq] (-2.35,-1.65) circle (2.5pt);
\fill [color=qqwuqq] (-1.65,-2.35) circle (2.5pt);
\filldraw [color=white, draw=ffqqqq, line width=0.5pt] (-2.35, 1.65) circle (2.5pt);
\fill [color=qqwuqq] ( 1.65, 2.35) circle (2.5pt);
\filldraw [color=white, draw=ffqqqq, line width=0.5pt] ( 1.65, 1.65) circle (2.5pt);
\filldraw [color=white, draw=ffqqqq, line width=0.5pt] ( 2.35, 2.35) circle (2.5pt);
\fill [color=qqwuqq] ( 2.35, 1.65) circle (2.5pt);
\filldraw [color=white, draw=ffqqqq, line width=0.5pt] ( 2.35,-1.65) circle (2.5pt);
\fill [color=qqwuqq] ( 2.35,-2.35) circle (2.5pt);
\filldraw [color=white, draw=ffqqqq, line width=0.5pt] ( 1.65,-2.35) circle (2.5pt);
\draw [color=qqwuqq] ( 1.8,-1.8) node{$p$};
\draw [color=qqwuqq] (-1.8, 1.8) node{$q$};
\draw [color=ffqqqq] (-1.8,-1.8) node{$r$};
\draw [color=ffqqqq] (-1.65, 2.35) node[above]{$r$};
\draw [color=ffqqqq] ( 2.35,-1.65) node[right]{$r$};
\draw [color=ffqqqq] ( 1.8, 1.8) node{$s$};
\draw [color=ffqqqq] ( 1.65,-2.35) node[below]{$s$};
\draw [color=ffqqqq] (-2.35, 1.65) node[left]{$s$};
\draw [line width=1pt,color=cyan] (1.65,-1.65) to[out=90, in=90] (-1,-2);
\draw [color=cyan, line width=0.7pt,dotted] (-1,-2) to[out=90, in=-90] (0.2,2);
\draw [line width=1pt,color=cyan] (0.2,2) to[out=-90, in=135] (0.65,0.65);
\draw [line width=1pt,color=cyan] (0.75,0.75) to[out=-45, in=180] (2,0);
\draw [color=cyan, line width=0.7pt,dotted] (2,0) to[out=180, in=0] (-2,0.2);
\draw [line width=1pt,color=cyan] (-2,0.2) to[out=0, in=-90] (-0.2,2);
\draw [color=cyan, line width=0.7pt,dotted] (-0.2,2) to[out=-90, in=90] (0.6,-2);
\draw [line width=1pt,color=cyan] plot[smooth, tension=0.1] (0.6,-2) to[out=90, in=180] (1.65,-1.65);
\draw ( 2.0, 0.0) node[right]{\small\color{red}$a_1$};
\draw (-2.0, 0.0) node[left]{\small\color{red}$a_1$};
\draw (0,0) node[right]{\small\color{red}$a_2$};
\draw ( 0.0, 2.0) node[above]{\small\color{red}$a_3$};
\draw (-0.0,-2.0) node[below]{\small\color{red}$a_3$};
\end{tikzpicture}
\begin{tikzpicture}[scale=0.55]
\filldraw [black!20] (-2, 2) circle (0.5cm);
\filldraw [black!20] ( 2, 2) circle (0.5cm);
\filldraw [black!20] ( 2,-2) circle (0.5cm);
\filldraw [black!20] (-2,-2) circle (0.5cm);
\draw [line width=1pt] (-2, 2) circle (0.5cm);
\draw [line width=1pt] ( 2, 2) circle (0.5cm);
\draw [line width=1pt] ( 2,-2) circle (0.5cm);
\draw [line width=1pt] (-2,-2) circle (0.5cm);
\draw [dash pattern=on 2pt off 2pt] (-2, 2)-- ( 2, 2);
\draw [dash pattern=on 2pt off 2pt] ( 2, 2)-- ( 2,-2);
\draw [dash pattern=on 2pt off 2pt] ( 2,-2)-- (-2,-2);
\draw [dash pattern=on 2pt off 2pt] (-2, 2)-- (-2,-2);
\draw [line width=1pt,color=ffqqqq] (-2.35,1.65)-- (-1.65,-1.65);
\draw [line width=1pt,color=ffqqqq] (-1.65,2.35)-- (1.65,1.65);
\draw [line width=1pt,color=ffqqqq] (1.65,1.65)-- (2.35,-1.65);
\draw [line width=1pt,color=ffqqqq] (1.65,-2.35)-- (-1.65,-1.65);
\draw [line width=1pt,color=ffqqqq] (-1.65,-1.65)-- (1.65,1.65);
\fill [color=qqwuqq] (-1.65, 1.65) circle (2.5pt);
\fill [color=qqwuqq] ( 1.65,-1.65) circle (2.5pt);
\filldraw [color=white, draw=ffqqqq, line width=0.5pt] (-1.65, 2.35) circle (2.5pt);
\fill [color=qqwuqq] (-2.35, 2.35) circle (2.5pt);
\filldraw [color=white, draw=ffqqqq, line width=0.5pt] (-1.65,-1.65) circle (2.5pt);
\filldraw [color=white, draw=ffqqqq, line width=0.5pt] (-2.35,-2.35) circle (2.5pt);
\fill [color=qqwuqq] (-2.35,-1.65) circle (2.5pt);
\fill [color=qqwuqq] (-1.65,-2.35) circle (2.5pt);
\filldraw [color=white, draw=ffqqqq, line width=0.5pt] (-2.35, 1.65) circle (2.5pt);
\fill [color=qqwuqq] ( 1.65, 2.35) circle (2.5pt);
\filldraw [color=white, draw=ffqqqq, line width=0.5pt] ( 1.65, 1.65) circle (2.5pt);
\filldraw [color=white, draw=ffqqqq, line width=0.5pt] ( 2.35, 2.35) circle (2.5pt);
\fill [color=qqwuqq] ( 2.35, 1.65) circle (2.5pt);
\filldraw [color=white, draw=ffqqqq, line width=0.5pt] ( 2.35,-1.65) circle (2.5pt);
\fill [color=qqwuqq] ( 2.35,-2.35) circle (2.5pt);
\filldraw [color=white, draw=ffqqqq, line width=0.5pt] ( 1.65,-2.35) circle (2.5pt);
\draw [color=qqwuqq] ( 1.8,-1.8) node{$p$};
\draw [color=qqwuqq] (-1.8, 1.8) node{$q$};
\draw [color=ffqqqq] (-1.8,-1.8) node{$r$};
\draw [color=ffqqqq] (-1.65, 2.35) node[above]{$r$};
\draw [color=ffqqqq] ( 2.35,-1.65) node[right]{$r$};
\draw [color=ffqqqq] ( 1.8, 1.8) node{$s$};
\draw [color=ffqqqq] ( 1.65,-2.35) node[below]{$s$};
\draw [color=ffqqqq] (-2.35, 1.65) node[left]{$s$};
\draw [rotate around={180:(0,0)}][line width=1pt, color=cyan]
      (1.65,-1.65) to[out=180, in=180] (2,1);
\draw [rotate around={180:(0,0)}][color=cyan, line width=0.7pt, dotted]
      (2,1) to[out=180, in=0] (-2,-0.2);
\draw [rotate around={180:(0,0)}][line width=1pt, color=cyan]
      (-2,-0.2) to[out=0, in=135] (-0.75,-0.75);
\draw [rotate around={180:(0,0)}][line width=1pt, color=cyan]
      (-0.65,-0.65) to[out=-45, in=90] (0,-2);
\draw [rotate around={180:(0,0)}][color=cyan, line width=0.7pt, dotted]
      (0,-2) to[out=90, in=-90] (-0.2,2);
\draw [rotate around={180:(0,0)}][line width=1pt,color=cyan]
      (-0.2,2) to[out=-90, in=0] (-2,0.2);
\draw [rotate around={180:(0,0)}][color=cyan, line width=0.7pt, dotted]
      (-2,0.2) to[out=0, in=180] (2,-0.6);
\draw [rotate around={180:(0,0)}][line width=1pt,color=cyan] plot[smooth, tension=0.1]
      (2,-0.6) to[out=180, in=90] (1.65,-1.65);
\draw ( 2.0, 0.0) node[right]{\small\color{red}$a_1$};
\draw (-2.0, 0.0) node[left]{\small\color{red}$a_1$};
\draw (0,0) node[right]{\small\color{red}$a_2$};
\draw ( 0.0, 2.0) node[above]{\small\color{red}$a_3$};
\draw (-0.0,-2.0) node[below]{\small\color{red}$a_3$};
\end{tikzpicture}
\begin{tikzpicture}[scale=0.55]
\filldraw [black!20] (-2, 2) circle (0.5cm);
\filldraw [black!20] ( 2, 2) circle (0.5cm);
\filldraw [black!20] ( 2,-2) circle (0.5cm);
\filldraw [black!20] (-2,-2) circle (0.5cm);
\draw [line width=1pt] (-2, 2) circle (0.5cm);
\draw [line width=1pt] ( 2, 2) circle (0.5cm);
\draw [line width=1pt] ( 2,-2) circle (0.5cm);
\draw [line width=1pt] (-2,-2) circle (0.5cm);
\draw [dash pattern=on 2pt off 2pt] (-2, 2)-- ( 2, 2);
\draw [dash pattern=on 2pt off 2pt] ( 2, 2)-- ( 2,-2);
\draw [dash pattern=on 2pt off 2pt] ( 2,-2)-- (-2,-2);
\draw [dash pattern=on 2pt off 2pt] (-2, 2)-- (-2,-2);
\draw [line width=1pt,color=ffqqqq] (-2.35,1.65)-- (-1.65,-1.65);
\draw [line width=1pt,color=ffqqqq] (-1.65,2.35)-- (1.65,1.65);
\draw [line width=1pt,color=ffqqqq] (1.65,1.65)-- (2.35,-1.65);
\draw [line width=1pt,color=ffqqqq] (1.65,-2.35)-- (-1.65,-1.65);
\draw [line width=1pt,color=ffqqqq] (-1.65,-1.65)-- (1.65,1.65);
\fill [color=qqwuqq] (-1.65, 1.65) circle (2.5pt);
\fill [color=qqwuqq] ( 1.65,-1.65) circle (2.5pt);
\filldraw [color=white, draw=ffqqqq, line width=0.5pt] (-1.65, 2.35) circle (2.5pt);
\fill [color=qqwuqq] (-2.35, 2.35) circle (2.5pt);
\filldraw [color=white, draw=ffqqqq, line width=0.5pt] (-1.65,-1.65) circle (2.5pt);
\filldraw [color=white, draw=ffqqqq, line width=0.5pt] (-2.35,-2.35) circle (2.5pt);
\fill [color=qqwuqq] (-2.35,-1.65) circle (2.5pt);
\fill [color=qqwuqq] (-1.65,-2.35) circle (2.5pt);
\filldraw [color=white, draw=ffqqqq, line width=0.5pt] (-2.35, 1.65) circle (2.5pt);
\fill [color=qqwuqq] ( 1.65, 2.35) circle (2.5pt);
\filldraw [color=white, draw=ffqqqq, line width=0.5pt] ( 1.65, 1.65) circle (2.5pt);
\filldraw [color=white, draw=ffqqqq, line width=0.5pt] ( 2.35, 2.35) circle (2.5pt);
\fill [color=qqwuqq] ( 2.35, 1.65) circle (2.5pt);
\filldraw [color=white, draw=ffqqqq, line width=0.5pt] ( 2.35,-1.65) circle (2.5pt);
\fill [color=qqwuqq] ( 2.35,-2.35) circle (2.5pt);
\filldraw [color=white, draw=ffqqqq, line width=0.5pt] ( 1.65,-2.35) circle (2.5pt);
\draw [color=qqwuqq] ( 1.8,-1.8) node{$p$};
\draw [color=qqwuqq] (-1.8, 1.8) node{$q$};
\draw [color=ffqqqq] (-1.8,-1.8) node{$r$};
\draw [color=ffqqqq] (-1.65, 2.35) node[above]{$r$};
\draw [color=ffqqqq] ( 2.35,-1.65) node[right]{$r$};
\draw [color=ffqqqq] ( 1.8, 1.8) node{$s$};
\draw [color=ffqqqq] ( 1.65,-2.35) node[below]{$s$};
\draw [color=ffqqqq] (-2.35, 1.65) node[left]{$s$};
\draw [rotate around={180:(0,0)}][line width=1pt,color=cyan]
      (1.65,-1.65) to[out=90, in=90] (-1,-2);
\draw [rotate around={180:(0,0)}][color=cyan, line width=0.7pt,dotted]
      (-1,-2) to[out=90, in=-90] (0.2,2);
\draw [rotate around={180:(0,0)}][line width=1pt,color=cyan]
      (0.2,2) to[out=-90, in=135] (0.65,0.65);
\draw [rotate around={180:(0,0)}][line width=1pt,color=cyan]
      (0.75,0.75) to[out=-45, in=180] (2,0);
\draw [rotate around={180:(0,0)}][color=cyan, line width=0.7pt,dotted]
      (2,0) to[out=180, in=0] (-2,0.2);
\draw [rotate around={180:(0,0)}][line width=1pt,color=cyan]
      (-2,0.2) to[out=0, in=-90] (-0.2,2);
\draw [rotate around={180:(0,0)}][color=cyan, line width=0.7pt,dotted]
      (-0.2,2) to[out=-90, in=90] (0.6,-2);
\draw [rotate around={180:(0,0)}][line width=1pt,color=cyan] plot[smooth, tension=0.1]
      (0.6,-2) to[out=90, in=180] (1.65,-1.65);
\draw ( 2.0, 0.0) node[right]{\small\color{red}$a_1$};
\draw (-2.0, 0.0) node[left]{\small\color{red}$a_1$};
\draw (0,0) node[right]{\small\color{red}$a_2$};
\draw ( 0.0, 2.0) node[above]{\small\color{red}$a_3$};
\draw (-0.0,-2.0) node[below]{\small\color{red}$a_3$};
\end{tikzpicture}
\caption{Impossible forms of a simplest sequence}
\label{fig:self-int Type}
\end{figure}
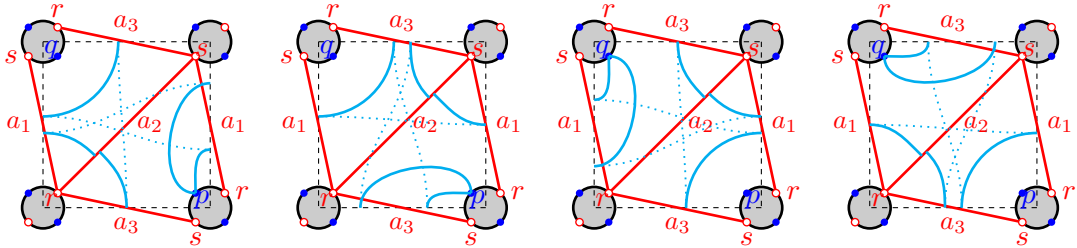

\end{lemma}

\begin{proof}
By symmetry, we only need to show that any simplest sequence of $\tsigma$ is not of the first form in \Pic\ref{fig:self-int Type}. Reversing the direction of $\ws$ if necessary, we may assume (S.R.E.). Assume conversely that there is a simplest sequence $0=i_1<i_2<\cdots<i_5=n(\ws)$ of $\tsigma$ such that $\snum{1}{\ws}=\snum{i_1+1}{\ws}=1$, and the segments $\ws_{i_2,i_2+1}$, $\ws_{i_3,i_3+1}$ and $\ws_{i_4,i_4+1}$ are of types (iv), (v) and (vi) in Figure~\ref{fig:four}, respectively.  
Then a similar argument as in the proof of Lemma~\ref{lem:four} shows that $\snum{2}{\tsigma}=2$, $\snum{3}{\tsigma}=3$, $\snum{4}{\tsigma}=1$, and $\ws$ does not cross the interiors of the following segments (see the shadow ones in the first picture of \Pic\ref{fig:cases}):
\begin{itemize}
    \item the segment of $\wa_1$ between $r$ and $\sip{1}{\ws}$,
    \item the segment of $\wa_2$ between $r$ and $\sip{2}{\ws}$,
    \item the segment of $\wa_3$ between $r$ and $\sip{3}{\ws}$,
    \item the segment of $\wa_1$ between $s$ and $\sip{4}{\ws}$.
\end{itemize}
We have the following three cases.
\begin{itemize}
    \item[(a)] $\snum{5}{\tsigma}=3$, see \Pic\ref{fig:cases}~(a). Then  
    $\sigma_{5, n(\tsigma)}$ is the concatenation of segments of types (i), (ii), (iv') and (v') in \Pic\ref{fig:four}.
    \item[(b)] $\snum{5}{\tsigma}=2$ and $\snum{6}{\tsigma}=1$, see \Pic\ref{fig:cases}~(b). Then $\sigma_{6, n(\tsigma)}$ is the concatenation of segments of types (i), (vi), (iii') and (iv') in \Pic\ref{fig:four}.
    \item[(c)] $\snum{5}{\tsigma}=2$ and $\snum{6}{\tsigma}=3$, see \Pic\ref{fig:cases}~(c). In this case, $\ws$ does not cross the interior of the segment of $\wa_3$ between $s$ and $\sip{6}{\ws}$. Since $\ws$ does not contain a circle, we have $\snum{7}{\ws}=2$. Then $\sigma_{7,n(\tsigma)}$ is the concatenation of segments of types (v), (vi), (ii') and (iii') in \Pic\ref{fig:four}.
\end{itemize}
In each case, $\ws_{i_2,i_2+1}$ (which is of type (iv)) is not a segment of $\ws$, a contradiction. Thus, we finish the proof.
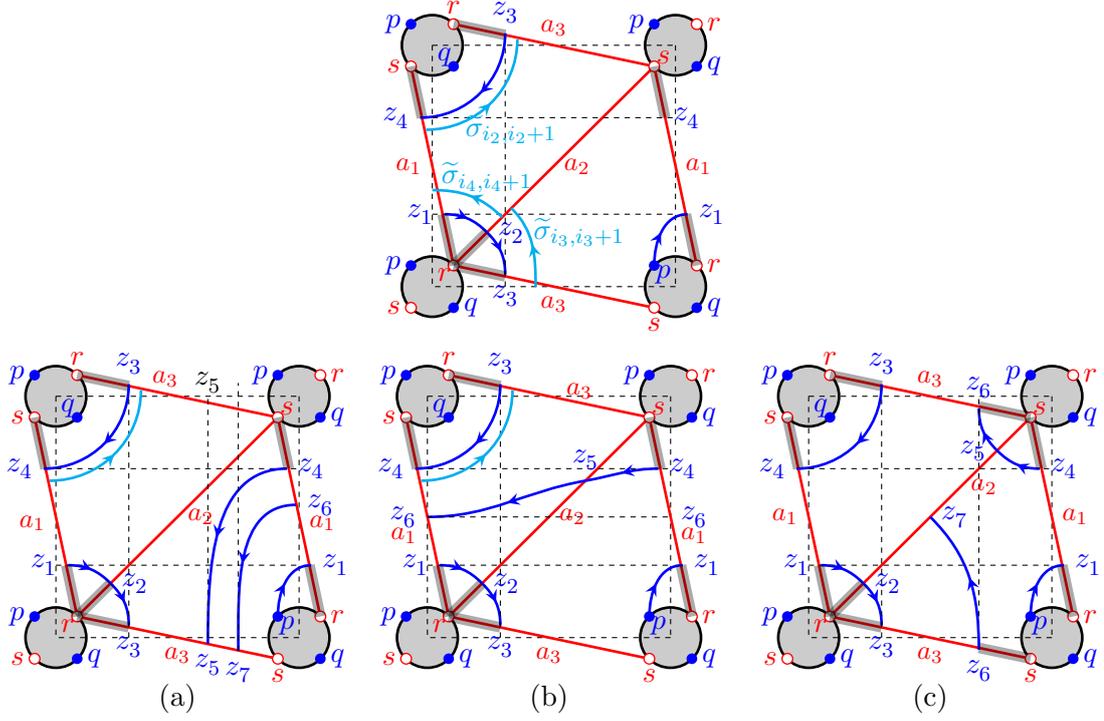
\begin{figure}[htbp]
\centering
\definecolor{ffqqqq}{rgb}{1,0,0}
\definecolor{qqwuqq}{rgb}{0,0,1}
\begin{tikzpicture}[scale=.8]
\filldraw [black!20] (-2, 2) circle (0.5cm);
\filldraw [black!20] ( 2, 2) circle (0.5cm);
\filldraw [black!20] ( 2,-2) circle (0.5cm);
\filldraw [black!20] (-2,-2) circle (0.5cm);
\draw [line width=1pt] (-2, 2) circle (0.5cm);
\draw [line width=1pt] ( 2, 2) circle (0.5cm);
\draw [line width=1pt] ( 2,-2) circle (0.5cm);
\draw [line width=1pt] (-2,-2) circle (0.5cm);
\draw [dash pattern=on 2pt off 2pt] (-2, 2)-- ( 2, 2);
\draw [dash pattern=on 2pt off 2pt] ( 2, 2)-- ( 2,-2);
\draw [dash pattern=on 2pt off 2pt] ( 2,-2)-- (-2,-2);
\draw [dash pattern=on 2pt off 2pt] (-2, 2)-- (-2,-2);
\draw [line width=1pt,color=ffqqqq] (-2.35,1.65)-- (-1.65,-1.65);
\draw [line width=1pt,color=ffqqqq] (-1.65,2.35)-- (1.65,1.65);
\draw [line width=1pt,color=ffqqqq] (1.65,1.65)-- (2.35,-1.65);
\draw [line width=1pt,color=ffqqqq] (1.65,-2.35)-- (-1.65,-1.65);
\draw [line width=1pt,color=ffqqqq] (-1.65,-1.65)-- (1.65,1.65);
\fill [color=qqwuqq] (-1.65, 1.65) circle (2.5pt);
\fill [color=qqwuqq] ( 1.65,-1.65) circle (2.5pt);
\filldraw [color=white, draw=ffqqqq, line width=0.5pt] (-1.65, 2.35) circle (2.5pt);
\fill [color=qqwuqq] (-2.35, 2.35) circle (2.5pt);
\filldraw [color=white, draw=ffqqqq, line width=0.5pt] (-1.65,-1.65) circle (2.5pt);
\filldraw [color=white, draw=ffqqqq, line width=0.5pt] (-2.35,-2.35) circle (2.5pt);
\fill [color=qqwuqq] (-2.35,-1.65) circle (2.5pt);
\fill [color=qqwuqq] (-1.65,-2.35) circle (2.5pt);
\filldraw [color=white, draw=ffqqqq, line width=0.5pt] (-2.35, 1.65) circle (2.5pt);
\fill [color=qqwuqq] ( 1.65, 2.35) circle (2.5pt);
\filldraw [color=white, draw=ffqqqq, line width=0.5pt] ( 1.65, 1.65) circle (2.5pt);
\filldraw [color=white, draw=ffqqqq, line width=0.5pt] ( 2.35, 2.35) circle (2.5pt);
\fill [color=qqwuqq] ( 2.35, 1.65) circle (2.5pt);
\filldraw [color=white, draw=ffqqqq, line width=0.5pt] ( 2.35,-1.65) circle (2.5pt);
\fill [color=qqwuqq] ( 2.35,-2.35) circle (2.5pt);
\filldraw [color=white, draw=ffqqqq, line width=0.5pt] ( 1.65,-2.35) circle (2.5pt);
\draw [color=qqwuqq] ( 1.8,-1.8) node{$p$};
\draw [color=qqwuqq] (-2.35, 2.35) node[left]{$p$};
\draw [color=qqwuqq] (-2.35,-1.65) node[left]{$p$};
\draw [color=qqwuqq] ( 1.65, 2.35) node[left]{$p$};
\draw [color=qqwuqq] (-1.8, 1.8) node{$q$};
\draw [color=qqwuqq] ( 2.35, 1.65) node[right]{$q$};
\draw [color=qqwuqq] ( 2.35,-2.35) node[right]{$q$};
\draw [color=qqwuqq] (-1.65,-2.35) node[right]{$q$};
\draw [color=ffqqqq] (-1.8,-1.8) node{$r$};
\draw [color=ffqqqq] (-1.65, 2.35) node[above]{$r$};
\draw [color=ffqqqq] ( 2.35, 2.35) node[right]{$r$};
\draw [color=ffqqqq] ( 2.35,-1.65) node[right]{$r$};
\draw [color=ffqqqq] ( 1.8, 1.8) node{$s$};
\draw [color=ffqqqq] ( 1.65,-2.35) node[below]{$s$};
\draw [color=ffqqqq] (-2.35,-2.35) node[left]{$s$};
\draw [color=ffqqqq] (-2.35, 1.65) node[left]{$s$};
\draw [blue][line width = 1pt][postaction={on each segment={mid arrow=blue}}]
      (1.65, -1.65) to[out=90, in =180] (2.2, -0.8) node[right]{$\sip{1}{\ws}$};
\draw [dash pattern=on 2pt off 2pt] (-2,-0.8)-- (2,-0.8);
\draw [blue][line width = 1pt][postaction={on each segment={mid arrow=blue}}]
      (-1.8,-0.8) node[left]{$\sip{1}{\ws}$} to[out=0, in=135]
      (-1.1,-1.1) node[right]{$\sip{2}{\ws}$} to[out=-45, in=90]
      (-0.8,-1.8) node[below]{$\sip{3}{\ws}$};
\draw [dash pattern=on 2pt off 2pt] (-0.8,-2)-- (-0.8,2);
\draw [blue][line width = 1pt][postaction={on each segment={mid arrow=blue}}]
      (-0.8, 2.2) node[above]{$\sip{3}{\ws}$} to[out=-90,in=0]
      (-2.2, 0.8) node[left]{$\sip{4}{\ws}$};
\draw[blue] ( 1.8, 0.8) node[right]{$\sip{4}{\ws}$};
\draw [dash pattern=on 2pt off 2pt] (-2, 0.8) -- ( 2, 0.8);
\draw [black][line width = 4pt][opacity=0.33] ( 1.83, 0.8 ) -- ( 1.65, 1.65);
\draw [black][line width = 4pt][opacity=0.33] ( 2.17,-0.8 ) -- ( 2.35,-1.65);
\draw [black][line width = 4pt][opacity=0.33] (-1.65,-1.65) -- (-0.8 ,-1.83);
\draw [black][line width = 4pt][opacity=0.33] (-1.65,-1.65) -- (-1.1 ,-1.1 );
\draw [black][line width = 4pt][opacity=0.33] (-1.65,-1.65) -- (-1.83,-0.8 );
\draw [black][line width = 4pt][opacity=0.33] (-2.35, 1.65) -- (-2.17, 0.8 );
\draw [black][line width = 4pt][opacity=0.33] (-1.65, 2.35) -- (-0.8 , 2.17);
\draw [cyan] [line width = 1pt] [postaction={on each segment={mid arrow=cyan}}]
      (-2.1,0.6) to[out=0,in=-90] (-0.6,2.1);
\draw [cyan] (-0.7, 0.6) node{$\ws_{i_2,i_2+1}$};
\draw [cyan] [line width = 1pt] [postaction={on each segment={mid arrow=cyan}}]
      (-0.3,-2) to[out=90,in=-45] (-0.7,-0.7);
\draw [cyan] (-0.5,-1.1) node[right]{$\ws_{i_3, i_3+1}$};
\draw [cyan] [line width = 1pt] [postaction={on each segment={mid arrow=cyan}}]
      (-0.85,-0.85) to[out=135,in=0] (-2,-0.4);
\draw [cyan] (-1.1,-0.6) node[above]{$\ws_{i_4, i_4+1}$};
\draw ( 2.0, 0.0) node[right]{\small\color{red}$a_1$};
\draw (-2.0, 0.0) node[left]{\small\color{red}$a_1$};
\draw (0,0) node[right]{\small\color{red}$a_2$};
\draw ( 0.0, 2.0) node[above]{\small\color{red}$a_3$};
\draw (-0.0,-2.0) node[below]{\small\color{red}$a_3$};
\end{tikzpicture}
\\
\begin{tikzpicture}[scale=0.8]
\filldraw [black!20] (-2, 2) circle (0.5cm);
\filldraw [black!20] ( 2, 2) circle (0.5cm);
\filldraw [black!20] ( 2,-2) circle (0.5cm);
\filldraw [black!20] (-2,-2) circle (0.5cm);
\draw [line width=1pt] (-2, 2) circle (0.5cm);
\draw [line width=1pt] ( 2, 2) circle (0.5cm);
\draw [line width=1pt] ( 2,-2) circle (0.5cm);
\draw [line width=1pt] (-2,-2) circle (0.5cm);
\draw [dash pattern=on 2pt off 2pt] (-2, 2)-- ( 2, 2);
\draw [dash pattern=on 2pt off 2pt] ( 2, 2)-- ( 2,-2);
\draw [dash pattern=on 2pt off 2pt] ( 2,-2)-- (-2,-2);
\draw [dash pattern=on 2pt off 2pt] (-2, 2)-- (-2,-2);
\draw [line width=1pt,color=ffqqqq] (-2.35,1.65)-- (-1.65,-1.65);
\draw [line width=1pt,color=ffqqqq] (-1.65,2.35)-- (1.65,1.65);
\draw [line width=1pt,color=ffqqqq] (1.65,1.65)-- (2.35,-1.65);
\draw [line width=1pt,color=ffqqqq] (1.65,-2.35)-- (-1.65,-1.65);
\draw [line width=1pt,color=ffqqqq] (-1.65,-1.65)-- (1.65,1.65);
\fill [color=qqwuqq] (-1.65, 1.65) circle (2.5pt);
\fill [color=qqwuqq] ( 1.65,-1.65) circle (2.5pt);
\filldraw [color=white, draw=ffqqqq, line width=0.5pt] (-1.65, 2.35) circle (2.5pt);
\fill [color=qqwuqq] (-2.35, 2.35) circle (2.5pt);
\filldraw [color=white, draw=ffqqqq, line width=0.5pt] (-1.65,-1.65) circle (2.5pt);
\filldraw [color=white, draw=ffqqqq, line width=0.5pt] (-2.35,-2.35) circle (2.5pt);
\fill [color=qqwuqq] (-2.35,-1.65) circle (2.5pt);
\fill [color=qqwuqq] (-1.65,-2.35) circle (2.5pt);
\filldraw [color=white, draw=ffqqqq, line width=0.5pt] (-2.35, 1.65) circle (2.5pt);
\fill [color=qqwuqq] ( 1.65, 2.35) circle (2.5pt);
\filldraw [color=white, draw=ffqqqq, line width=0.5pt] ( 1.65, 1.65) circle (2.5pt);
\filldraw [color=white, draw=ffqqqq, line width=0.5pt] ( 2.35, 2.35) circle (2.5pt);
\fill [color=qqwuqq] ( 2.35, 1.65) circle (2.5pt);
\filldraw [color=white, draw=ffqqqq, line width=0.5pt] ( 2.35,-1.65) circle (2.5pt);
\fill [color=qqwuqq] ( 2.35,-2.35) circle (2.5pt);
\filldraw [color=white, draw=ffqqqq, line width=0.5pt] ( 1.65,-2.35) circle (2.5pt);
\draw [color=qqwuqq] ( 1.8,-1.8) node{$p$};
\draw [color=qqwuqq] (-2.35, 2.35) node[left]{$p$};
\draw [color=qqwuqq] (-2.35,-1.65) node[left]{$p$};
\draw [color=qqwuqq] ( 1.65, 2.35) node[left]{$p$};
\draw [color=qqwuqq] (-1.8, 1.8) node{$q$};
\draw [color=qqwuqq] ( 2.35, 1.65) node[right]{$q$};
\draw [color=qqwuqq] ( 2.35,-2.35) node[right]{$q$};
\draw [color=qqwuqq] (-1.65,-2.35) node[right]{$q$};
\draw [color=ffqqqq] (-1.8,-1.8) node{$r$};
\draw [color=ffqqqq] (-1.65, 2.35) node[above]{$r$};
\draw [color=ffqqqq] ( 2.35, 2.35) node[right]{$r$};
\draw [color=ffqqqq] ( 2.35,-1.65) node[right]{$r$};
\draw [color=ffqqqq] ( 1.8, 1.8) node{$s$};
\draw [color=ffqqqq] ( 1.65,-2.35) node[below]{$s$};
\draw [color=ffqqqq] (-2.35,-2.35) node[left]{$s$};
\draw [color=ffqqqq] (-2.35, 1.65) node[left]{$s$};
\draw [blue][line width = 1pt][postaction={on each segment={mid arrow=blue}}]
      (1.65, -1.65) to[out=90, in =180] (2.2, -0.8) node[right]{$\sip{1}{\ws}$};
\draw [dash pattern=on 2pt off 2pt] (-2,-0.8)-- (2,-0.8);
\draw [blue][line width = 1pt][postaction={on each segment={mid arrow=blue}}]
      (-1.8,-0.8) node[left]{$\sip{1}{\ws}$} to[out=0, in=135]
      (-1.1,-1.1) node[right]{$\sip{2}{\ws}$} to[out=-45, in=90]
      (-0.8,-1.8) node[below]{$\sip{3}{\ws}$};
\draw [dash pattern=on 2pt off 2pt] (-0.8,-2)-- (-0.8,2);
\draw [blue][line width = 1pt][postaction={on each segment={mid arrow=blue}}]
      (-0.8, 2.2) node[above]{$\sip{3}{\ws}$} to[out=-90,in=0]
      (-2.2, 0.8) node[left]{$\sip{4}{\ws}$};
\draw [dash pattern=on 2pt off 2pt] (-2, 0.8) -- ( 2, 0.8);
\draw [blue][line width = 1pt][postaction={on each segment={mid arrow=blue}}]
      ( 1.8, 0.8) node[right]{$\sip{4}{\ws}$} to[out=180,in=90]
      ( 0.5,-2.1) node[below]{$\sip{5}{\ws}$};
\draw [dash pattern=on 2pt off 2pt] ( 0.5, -2)-- ( 0.5, 2);
\draw ( 0.5, 1.9) node[above]{$\sip{5}{\ws}$};
\draw [blue][line width = 1pt][postaction={on each segment={mid arrow=blue}}]
      (1.95, 0.2) node[right]{$\sip{6}{\ws}$} to[out=180,in=90]
      (1, -2.22) node[below]{$\sip{7}{\ws}$};
\draw [dash pattern=on 2pt off 2pt] ( 1,-2.22) -- ( 1, 2.22);
\draw [cyan] [line width = 1pt] [postaction={on each segment={mid arrow=cyan}}]
      (-2.1,0.6) to[out=0,in=-90] (-0.6,2.1);
\draw [black][line width = 4pt][opacity=0.33] ( 1.83, 0.8 ) -- ( 1.65, 1.65);
\draw [black][line width = 4pt][opacity=0.33] ( 2.17,-0.8 ) -- ( 2.35,-1.65);
\draw [black][line width = 4pt][opacity=0.33] (-1.65,-1.65) -- (-0.8 ,-1.83);
\draw [black][line width = 4pt][opacity=0.33] (-1.65,-1.65) -- (-1.1 ,-1.1 );
\draw [black][line width = 4pt][opacity=0.33] (-1.65,-1.65) -- (-1.83,-0.8 );
\draw [black][line width = 4pt][opacity=0.33] (-2.35, 1.65) -- (-2.17, 0.8 );
\draw [black][line width = 4pt][opacity=0.33] (-1.65, 2.35) -- (-0.8 , 2.17);
\draw ( 2.0,-0.1) node[right]{\small\color{red}$a_1$};
\draw (-2.0,-0.1) node[left]{\small\color{red}$a_1$};
\draw ( 0.0, 0.0) node[right]{\small\color{red}$a_2$};
\draw (-0.2, 2.0) node[above]{\small\color{red}$a_3$};
\draw (-0.0,-2.0) node[below]{\small\color{red}$a_3$};
\draw (0,-3) node{(a)};
\end{tikzpicture}
\begin{tikzpicture}[scale=0.8]
\filldraw [black!20] (-2, 2) circle (0.5cm);
\filldraw [black!20] ( 2, 2) circle (0.5cm);
\filldraw [black!20] ( 2,-2) circle (0.5cm);
\filldraw [black!20] (-2,-2) circle (0.5cm);
\draw [line width=1pt] (-2, 2) circle (0.5cm);
\draw [line width=1pt] ( 2, 2) circle (0.5cm);
\draw [line width=1pt] ( 2,-2) circle (0.5cm);
\draw [line width=1pt] (-2,-2) circle (0.5cm);
\draw [dash pattern=on 2pt off 2pt] (-2, 2)-- ( 2, 2);
\draw [dash pattern=on 2pt off 2pt] ( 2, 2)-- ( 2,-2);
\draw [dash pattern=on 2pt off 2pt] ( 2,-2)-- (-2,-2);
\draw [dash pattern=on 2pt off 2pt] (-2, 2)-- (-2,-2);
\draw [line width=1pt,color=ffqqqq] (-2.35,1.65)-- (-1.65,-1.65);
\draw [line width=1pt,color=ffqqqq] (-1.65,2.35)-- (1.65,1.65);
\draw [line width=1pt,color=ffqqqq] (1.65,1.65)-- (2.35,-1.65);
\draw [line width=1pt,color=ffqqqq] (1.65,-2.35)-- (-1.65,-1.65);
\draw [line width=1pt,color=ffqqqq] (-1.65,-1.65)-- (1.65,1.65);
\fill [color=qqwuqq] (-1.65, 1.65) circle (2.5pt);
\fill [color=qqwuqq] ( 1.65,-1.65) circle (2.5pt);
\filldraw [color=white, draw=ffqqqq, line width=0.5pt] (-1.65, 2.35) circle (2.5pt);
\fill [color=qqwuqq] (-2.35, 2.35) circle (2.5pt);
\filldraw [color=white, draw=ffqqqq, line width=0.5pt] (-1.65,-1.65) circle (2.5pt);
\filldraw [color=white, draw=ffqqqq, line width=0.5pt] (-2.35,-2.35) circle (2.5pt);
\fill [color=qqwuqq] (-2.35,-1.65) circle (2.5pt);
\fill [color=qqwuqq] (-1.65,-2.35) circle (2.5pt);
\filldraw [color=white, draw=ffqqqq, line width=0.5pt] (-2.35, 1.65) circle (2.5pt);
\fill [color=qqwuqq] ( 1.65, 2.35) circle (2.5pt);
\filldraw [color=white, draw=ffqqqq, line width=0.5pt] ( 1.65, 1.65) circle (2.5pt);
\filldraw [color=white, draw=ffqqqq, line width=0.5pt] ( 2.35, 2.35) circle (2.5pt);
\fill [color=qqwuqq] ( 2.35, 1.65) circle (2.5pt);
\filldraw [color=white, draw=ffqqqq, line width=0.5pt] ( 2.35,-1.65) circle (2.5pt);
\fill [color=qqwuqq] ( 2.35,-2.35) circle (2.5pt);
\filldraw [color=white, draw=ffqqqq, line width=0.5pt] ( 1.65,-2.35) circle (2.5pt);
\draw [color=qqwuqq] ( 1.8,-1.8) node{$p$};
\draw [color=qqwuqq] (-2.35, 2.35) node[left]{$p$};
\draw [color=qqwuqq] (-2.35,-1.65) node[left]{$p$};
\draw [color=qqwuqq] ( 1.65, 2.35) node[left]{$p$};
\draw [color=qqwuqq] (-1.8, 1.8) node{$q$};
\draw [color=qqwuqq] ( 2.35, 1.65) node[right]{$q$};
\draw [color=qqwuqq] ( 2.35,-2.35) node[right]{$q$};
\draw [color=qqwuqq] (-1.65,-2.35) node[right]{$q$};
\draw [color=ffqqqq] (-1.8,-1.8) node{$r$};
\draw [color=ffqqqq] (-1.65, 2.35) node[above]{$r$};
\draw [color=ffqqqq] ( 2.35, 2.35) node[right]{$r$};
\draw [color=ffqqqq] ( 2.35,-1.65) node[right]{$r$};
\draw [color=ffqqqq] ( 1.8, 1.8) node{$s$};
\draw [color=ffqqqq] ( 1.65,-2.35) node[below]{$s$};
\draw [color=ffqqqq] (-2.35,-2.35) node[left]{$s$};
\draw [color=ffqqqq] (-2.35, 1.65) node[left]{$s$};
\draw [blue][line width = 1pt][postaction={on each segment={mid arrow=blue}}]
      (1.65, -1.65) to[out=90, in =180] (2.2, -0.8) node[right]{$\sip{1}{\ws}$};
\draw [dash pattern=on 2pt off 2pt] (-2,-0.8)-- (2,-0.8);
\draw [blue][line width = 1pt][postaction={on each segment={mid arrow=blue}}]
      (-1.8,-0.8) node[left]{$\sip{1}{\ws}$} to[out=0, in=135]
      (-1.1,-1.1) node[right]{$\sip{2}{\ws}$} to[out=-45, in=90]
      (-0.8,-1.8) node[below]{$\sip{3}{\ws}$};
\draw [dash pattern=on 2pt off 2pt] (-0.8,-2)-- (-0.8,2);
\draw [blue][line width = 1pt][postaction={on each segment={mid arrow=blue}}]
      (-0.8, 2.2) node[above]{$\sip{3}{\ws}$} to[out=-90,in=0]
      (-2.2, 0.8) node[left]{$\sip{4}{\ws}$};
\draw [dash pattern=on 2pt off 2pt] (-2, 0.8) -- ( 2, 0.8);
\draw [blue][line width = 1pt][postaction={on each segment={mid arrow=blue}}]
      ( 1.8, 0.8) node[right]{$\sip{4}{\ws}$} to[out=180,in=10]
      ( 0.6, 0.6) node[above]{$\sip{5}{\ws}$} to[out=190,in= 0]
      (-2,-0.0) node[left]{$\sip{6}{\ws}$};
\draw [dash pattern=on 2pt off 2pt] (-2,-0.0) -- ( 2, 0.0);
\draw [blue] ( 2, 0.0) node[right]{$\sip{6}{\ws}$};
\draw [cyan] [line width = 1pt] [postaction={on each segment={mid arrow=cyan}}]
      (-2.1,0.6) to[out=0,in=-90] (-0.6,2.1);
\draw [black][line width = 4pt][opacity=0.33] ( 1.83, 0.8 ) -- ( 1.65, 1.65);
\draw [black][line width = 4pt][opacity=0.33] ( 2.17,-0.8 ) -- ( 2.35,-1.65);
\draw [black][line width = 4pt][opacity=0.33] (-1.65,-1.65) -- (-0.8 ,-1.83);
\draw [black][line width = 4pt][opacity=0.33] (-1.65,-1.65) -- (-1.1 ,-1.1 );
\draw [black][line width = 4pt][opacity=0.33] (-1.65,-1.65) -- (-1.83,-0.8 );
\draw [black][line width = 4pt][opacity=0.33] (-2.35, 1.65) -- (-2.17, 0.8 );
\draw [black][line width = 4pt][opacity=0.33] (-1.65, 2.35) -- (-0.8 , 2.17);
\draw ( 2.0,-0.3) node[right]{\small\color{red}$a_1$};
\draw (-2.0,-0.3) node[left]{\small\color{red}$a_1$};
\draw (0,0) node[right]{\small\color{red}$a_2$};
\draw ( 0.5, 1.8) node[above]{\small\color{red}$a_3$};
\draw (-0.0,-2.0) node[below]{\small\color{red}$a_3$};
\draw (0,-3) node{(b)};
\end{tikzpicture}
\
\begin{tikzpicture}[scale=0.8]
\filldraw [black!20] (-2, 2) circle (0.5cm);
\filldraw [black!20] ( 2, 2) circle (0.5cm);
\filldraw [black!20] ( 2,-2) circle (0.5cm);
\filldraw [black!20] (-2,-2) circle (0.5cm);
\draw [line width=1pt] (-2, 2) circle (0.5cm);
\draw [line width=1pt] ( 2, 2) circle (0.5cm);
\draw [line width=1pt] ( 2,-2) circle (0.5cm);
\draw [line width=1pt] (-2,-2) circle (0.5cm);
\draw [dash pattern=on 2pt off 2pt] (-2, 2)-- ( 2, 2);
\draw [dash pattern=on 2pt off 2pt] ( 2, 2)-- ( 2,-2);
\draw [dash pattern=on 2pt off 2pt] ( 2,-2)-- (-2,-2);
\draw [dash pattern=on 2pt off 2pt] (-2, 2)-- (-2,-2);
\draw [line width=1pt,color=ffqqqq] (-2.35,1.65)-- (-1.65,-1.65);
\draw [line width=1pt,color=ffqqqq] (-1.65,2.35)-- (1.65,1.65);
\draw [line width=1pt,color=ffqqqq] (1.65,1.65)-- (2.35,-1.65);
\draw [line width=1pt,color=ffqqqq] (1.65,-2.35)-- (-1.65,-1.65);
\draw [line width=1pt,color=ffqqqq] (-1.65,-1.65)-- (1.65,1.65);
\fill [color=qqwuqq] (-1.65, 1.65) circle (2.5pt);
\fill [color=qqwuqq] ( 1.65,-1.65) circle (2.5pt);
\filldraw [color=white, draw=ffqqqq, line width=0.5pt] (-1.65, 2.35) circle (2.5pt);
\fill [color=qqwuqq] (-2.35, 2.35) circle (2.5pt);
\filldraw [color=white, draw=ffqqqq, line width=0.5pt] (-1.65,-1.65) circle (2.5pt);
\filldraw [color=white, draw=ffqqqq, line width=0.5pt] (-2.35,-2.35) circle (2.5pt);
\fill [color=qqwuqq] (-2.35,-1.65) circle (2.5pt);
\fill [color=qqwuqq] (-1.65,-2.35) circle (2.5pt);
\filldraw [color=white, draw=ffqqqq, line width=0.5pt] (-2.35, 1.65) circle (2.5pt);
\fill [color=qqwuqq] ( 1.65, 2.35) circle (2.5pt);
\filldraw [color=white, draw=ffqqqq, line width=0.5pt] ( 1.65, 1.65) circle (2.5pt);
\filldraw [color=white, draw=ffqqqq, line width=0.5pt] ( 2.35, 2.35) circle (2.5pt);
\fill [color=qqwuqq] ( 2.35, 1.65) circle (2.5pt);
\filldraw [color=white, draw=ffqqqq, line width=0.5pt] ( 2.35,-1.65) circle (2.5pt);
\fill [color=qqwuqq] ( 2.35,-2.35) circle (2.5pt);
\filldraw [color=white, draw=ffqqqq, line width=0.5pt] ( 1.65,-2.35) circle (2.5pt);
\draw [color=qqwuqq] ( 1.8,-1.8) node{$p$};
\draw [color=qqwuqq] (-2.35, 2.35) node[left]{$p$};
\draw [color=qqwuqq] (-2.35,-1.65) node[left]{$p$};
\draw [color=qqwuqq] ( 1.65, 2.35) node[left]{$p$};
\draw [color=qqwuqq] (-1.8, 1.8) node{$q$};
\draw [color=qqwuqq] ( 2.35, 1.65) node[right]{$q$};
\draw [color=qqwuqq] ( 2.35,-2.35) node[right]{$q$};
\draw [color=qqwuqq] (-1.65,-2.35) node[right]{$q$};
\draw [color=ffqqqq] (-1.8,-1.8) node{$r$};
\draw [color=ffqqqq] (-1.65, 2.35) node[above]{$r$};
\draw [color=ffqqqq] ( 2.35, 2.35) node[right]{$r$};
\draw [color=ffqqqq] ( 2.35,-1.65) node[right]{$r$};
\draw [color=ffqqqq] ( 1.8, 1.8) node{$s$};
\draw [color=ffqqqq] ( 1.65,-2.35) node[below]{$s$};
\draw [color=ffqqqq] (-2.35,-2.35) node[left]{$s$};
\draw [color=ffqqqq] (-2.35, 1.65) node[left]{$s$};
\draw [blue][line width = 1pt][postaction={on each segment={mid arrow=blue}}]
      (1.65, -1.65) to[out=90, in =180] (2.2, -0.8) node[right]{$\sip{1}{\ws}$};
\draw [dash pattern=on 2pt off 2pt] (-2,-0.8)-- (2,-0.8);
\draw [blue][line width = 1pt][postaction={on each segment={mid arrow=blue}}]
      (-1.8,-0.8) node[left]{$\sip{1}{\ws}$} to[out=0, in=135]
      (-1.1,-1.1) node[right]{$\sip{2}{\ws}$} to[out=-45, in=90]
      (-0.8,-1.8) node[below]{$\sip{3}{\ws}$};
\draw [dash pattern=on 2pt off 2pt] (-0.8,-2)-- (-0.8,2);
\draw [blue][line width = 1pt][postaction={on each segment={mid arrow=blue}}]
      (-0.8, 2.2) node[above]{$\sip{3}{\ws}$} to[out=-90,in=0]
      (-2.2, 0.8) node[left]{$\sip{4}{\ws}$};
\draw [dash pattern=on 2pt off 2pt] (-2, 0.8) -- ( 2, 0.8);
\draw [blue][line width = 1pt][postaction={on each segment={mid arrow=blue}}]
      ( 1.8, 0.8) node[right]{$\sip{4}{\ws}$} to[out=180,in=-45]
      ( 1.1, 1.1) node[left]{$\sip{5}{\ws}$} to[out=135,in=-90]
      ( 0.8, 1.8) node[above]{$\sip{6}{\ws}$};
\draw [dash pattern=on 2pt off 2pt] ( 0.8,-2) -- ( 0.8, 2);
\draw [blue][line width = 1pt][postaction={on each segment={mid arrow=blue}}]
      ( 0.8,-2.2) node[below]{$\sip{6}{\ws}$} to[out=90, in=-45] (0,0) node[right]{$\sip{7}{\ws}$};
\draw ( 2.0, 0.0) node[right]{\small\color{red}$a_1$};
\draw (-2.0, 0.0) node[left]{\small\color{red}$a_1$};
\draw ( 0.5, 0.5) node[right]{\small\color{red}$a_2$};
\draw ( 0.0, 2.0) node[above]{\small\color{red}$a_3$};
\draw (-0.0,-2.0) node[below]{\small\color{red}$a_3$};
\draw [black][line width = 4pt][opacity=0.33] ( 0.8 , 1.83) -- ( 1.65, 1.65);
\draw [black][line width = 4pt][opacity=0.33] ( 1.83, 0.8 ) -- ( 1.65, 1.65);
\draw [black][line width = 4pt][opacity=0.33] ( 2.17,-0.8 ) -- ( 2.35,-1.65);
\draw [black][line width = 4pt][opacity=0.33] ( 1.65,-2.35) -- ( 0.8 ,-2.17);
\draw [black][line width = 4pt][opacity=0.33] (-1.65,-1.65) -- (-0.8 ,-1.83);
\draw [black][line width = 4pt][opacity=0.33] (-1.65,-1.65) -- (-1.1 ,-1.1 );
\draw [black][line width = 4pt][opacity=0.33] (-1.65,-1.65) -- (-1.83,-0.8 );
\draw [black][line width = 4pt][opacity=0.33] (-2.35, 1.65) -- (-2.17, 0.8 );
\draw [black][line width = 4pt][opacity=0.33] (-1.65, 2.35) -- (-0.8 , 2.17);
\draw (0,-3) node{(c)};
\end{tikzpicture}
\caption{Cases in the proof of Lemma~\ref{lemm:self-int Type}}
\label{fig:cases}
\end{figure}

\end{proof}

In the following case, a simplest sequence can recover the arc.

\begin{lemma}\label{lem:wg}
    Let $\ws\in\SGOCM(\SURF)$ with $\sigma(0)=p$ and $\sigma(1)=q$. If $\ws$ has a simplest sequence $0=i_1<i_2<\cdots<i_4=n(\ws)$ such that $\snum{i_2}{\ws}=3$, $\snum{i_3}{\ws}=2$ and $\snum{i_4}{\ws}=1$ (see the first picture of Figure~\ref{fig-star Type 6}), then $\ws$ is homotopic to $\tgamma$ (see Figure~\ref{fig:tgamma} for the graded arc $\tgamma$).

\begin{figure}[htbp]
\definecolor{ffqqqq}{rgb}{1,0,0}
\definecolor{qqwuqq}{rgb}{0,0,1}
\begin{center}
\begin{tikzpicture} [scale=0.7]
\filldraw [black!20] (-2, 2) circle (0.5cm);
\filldraw [black!20] ( 2, 2) circle (0.5cm);
\filldraw [black!20] ( 2,-2) circle (0.5cm);
\filldraw [black!20] (-2,-2) circle (0.5cm);
\draw [line width=1pt] (-2, 2) circle (0.5cm);
\draw [line width=1pt] ( 2, 2) circle (0.5cm);
\draw [line width=1pt] ( 2,-2) circle (0.5cm);
\draw [line width=1pt] (-2,-2) circle (0.5cm);
\draw [dash pattern=on 2pt off 2pt] (-2, 2)-- ( 2, 2);
\draw [dash pattern=on 2pt off 2pt] ( 2, 2)-- ( 2,-2);
\draw [dash pattern=on 2pt off 2pt] ( 2,-2)-- (-2,-2);
\draw [dash pattern=on 2pt off 2pt] (-2, 2)-- (-2,-2);
\draw [line width=1pt,color=ffqqqq] (-2.35,1.65)-- (-1.65,-1.65);
\draw [line width=1pt,color=ffqqqq] (-1.65,2.35)-- (1.65,1.65);
\draw [line width=1pt,color=ffqqqq] (1.65,1.65)-- (2.35,-1.65);
\draw [line width=1pt,color=ffqqqq] (1.65,-2.35)-- (-1.65,-1.65);
\draw [line width=1pt,color=ffqqqq] (-1.65,-1.65)-- (1.65,1.65);
\fill [color=qqwuqq] (-1.65, 1.65) circle (2.5pt);
\fill [color=qqwuqq] ( 1.65,-1.65) circle (2.5pt);
\filldraw [color=white, draw=ffqqqq, line width=0.5pt] (-1.65, 2.35) circle (2.5pt);
\fill [color=qqwuqq] (-2.35, 2.35) circle (2.5pt);
\filldraw [color=white, draw=ffqqqq, line width=0.5pt] (-1.65,-1.65) circle (2.5pt);
\filldraw [color=white, draw=ffqqqq, line width=0.5pt] (-2.35,-2.35) circle (2.5pt);
\fill [color=qqwuqq] (-2.35,-1.65) circle (2.5pt);
\fill [color=qqwuqq] (-1.65,-2.35) circle (2.5pt);
\filldraw [color=white, draw=ffqqqq, line width=0.5pt] (-2.35, 1.65) circle (2.5pt);
\fill [color=qqwuqq] ( 1.65, 2.35) circle (2.5pt);
\filldraw [color=white, draw=ffqqqq, line width=0.5pt] ( 1.65, 1.65) circle (2.5pt);
\filldraw [color=white, draw=ffqqqq, line width=0.5pt] ( 2.35, 2.35) circle (2.5pt);
\fill [color=qqwuqq] ( 2.35, 1.65) circle (2.5pt);
\filldraw [color=white, draw=ffqqqq, line width=0.5pt] ( 2.35,-1.65) circle (2.5pt);
\fill [color=qqwuqq] ( 2.35,-2.35) circle (2.5pt);
\filldraw [color=white, draw=ffqqqq, line width=0.5pt] ( 1.65,-2.35) circle (2.5pt);
\draw [color=qqwuqq] ( 1.8,-1.8) node{$p$};
\draw [color=qqwuqq] (-2.35, 2.35) node[left]{$p$};
\draw [color=qqwuqq] (-2.35,-1.65) node[left]{$p$};
\draw [color=qqwuqq] ( 1.65, 2.35) node[left]{$p$};
\draw [color=qqwuqq] (-1.8, 1.8) node{$q$};
\draw [color=qqwuqq] ( 2.35, 1.65) node[right]{$q$};
\draw [color=qqwuqq] ( 2.35,-2.35) node[right]{$q$};
\draw [color=qqwuqq] (-1.65,-2.35) node[right]{$q$};
\draw [color=ffqqqq] (-1.8,-1.8) node{$r$};
\draw [color=ffqqqq] (-1.65, 2.35) node[above]{$r$};
\draw [color=ffqqqq] ( 2.35, 2.35) node[right]{$r$};
\draw [color=ffqqqq] ( 2.35,-1.65) node[right]{$r$};
\draw [color=ffqqqq] ( 1.8, 1.8) node{$s$};
\draw [color=ffqqqq] ( 1.65,-2.35) node[below]{$s$};
\draw [color=ffqqqq] (-2.35,-2.35) node[left]{$s$};
\draw [color=ffqqqq] (-2.35, 1.65) node[left]{$s$};
\draw [cyan][line width=1pt][postaction={on each segment={mid arrow=cyan}}]
      (1.65,-1.65) to[out=180, in=90] (0.2, -2.05) node[below]{$\sip{i_1+1}{\ws}$};
\draw [cyan][line width=1pt][postaction={on each segment={mid arrow=cyan}}]
      (0.2, 1.95) node[above]{$\sip{i_2}{\ws}$} to[out=-90, in=135]
      (0.7, 0.7)  node[left]{$\sip{i_2+1}{\ws}$};
\draw [cyan][line width=1pt][postaction={on each segment={mid arrow=cyan}}]
      (0.5,0.5)  node[below]{$\sip{i_3}{\ws}$} to[out=-45,in=180]
      (2,0); \draw [cyan] (1.8,0.0) node[right]{$\sip{i_3\!+\!1}{\ws}$};
\draw [cyan][line width=1pt][postaction={on each segment={mid arrow=cyan}}]
      (-2.2, 1) node[left]{$\sip{i_4}{\ws}$} to[out=0,in=-90] (-1.65,1.65);
\draw ( 1.9, 0.5) node[right]{\small\color{red}$a_1$};
\draw (-2.0, 0.0) node[left]{\small\color{red}$a_1$};
\draw (-0.5,-0.5) node[right]{\small\color{red}$a_2$};
\draw (-0.5, 1.9) node[above]{\small\color{red}$a_3$};
\draw (-0.5,-1.9) node[below]{\small\color{red}$a_3$};
\draw[white] (0,-3) node{(1)};
\end{tikzpicture}
\begin{tikzpicture} [scale=0.7]
\filldraw [black!20] (-2, 2) circle (0.5cm);
\filldraw [black!20] ( 2, 2) circle (0.5cm);
\filldraw [black!20] ( 2,-2) circle (0.5cm);
\filldraw [black!20] (-2,-2) circle (0.5cm);
\draw [line width=1pt] (-2, 2) circle (0.5cm);
\draw [line width=1pt] ( 2, 2) circle (0.5cm);
\draw [line width=1pt] ( 2,-2) circle (0.5cm);
\draw [line width=1pt] (-2,-2) circle (0.5cm);
\draw [dash pattern=on 2pt off 2pt] (-2, 2)-- ( 2, 2);
\draw [dash pattern=on 2pt off 2pt] ( 2, 2)-- ( 2,-2);
\draw [dash pattern=on 2pt off 2pt] ( 2,-2)-- (-2,-2);
\draw [dash pattern=on 2pt off 2pt] (-2, 2)-- (-2,-2);
\draw [line width=1pt,color=ffqqqq] (-2.35,1.65)-- (-1.65,-1.65);
\draw [line width=1pt,color=ffqqqq] (-1.65,2.35)-- (1.65,1.65);
\draw [line width=1pt,color=ffqqqq] (1.65,1.65)-- (2.35,-1.65);
\draw [line width=1pt,color=ffqqqq] (1.65,-2.35)-- (-1.65,-1.65);
\draw [line width=1pt,color=ffqqqq] (-1.65,-1.65)-- (1.65,1.65);
\fill [color=qqwuqq] (-1.65, 1.65) circle (2.5pt);
\fill [color=qqwuqq] ( 1.65,-1.65) circle (2.5pt);
\filldraw [color=white, draw=ffqqqq, line width=0.5pt] (-1.65, 2.35) circle (2.5pt);
\fill [color=qqwuqq] (-2.35, 2.35) circle (2.5pt);
\filldraw [color=white, draw=ffqqqq, line width=0.5pt] (-1.65,-1.65) circle (2.5pt);
\filldraw [color=white, draw=ffqqqq, line width=0.5pt] (-2.35,-2.35) circle (2.5pt);
\fill [color=qqwuqq] (-2.35,-1.65) circle (2.5pt);
\fill [color=qqwuqq] (-1.65,-2.35) circle (2.5pt);
\filldraw [color=white, draw=ffqqqq, line width=0.5pt] (-2.35, 1.65) circle (2.5pt);
\fill [color=qqwuqq] ( 1.65, 2.35) circle (2.5pt);
\filldraw [color=white, draw=ffqqqq, line width=0.5pt] ( 1.65, 1.65) circle (2.5pt);
\filldraw [color=white, draw=ffqqqq, line width=0.5pt] ( 2.35, 2.35) circle (2.5pt);
\fill [color=qqwuqq] ( 2.35, 1.65) circle (2.5pt);
\filldraw [color=white, draw=ffqqqq, line width=0.5pt] ( 2.35,-1.65) circle (2.5pt);
\fill [color=qqwuqq] ( 2.35,-2.35) circle (2.5pt);
\filldraw [color=white, draw=ffqqqq, line width=0.5pt] ( 1.65,-2.35) circle (2.5pt);
\draw [color=qqwuqq] ( 1.8,-1.8) node{$p$};
\draw [color=qqwuqq] (-2.35, 2.35) node[left]{$p$};
\draw [color=qqwuqq] (-2.35,-1.65) node[left]{$p$};
\draw [color=qqwuqq] ( 1.65, 2.35) node[left]{$p$};
\draw [color=qqwuqq] (-1.8, 1.8) node{$q$};
\draw [color=qqwuqq] ( 2.35, 1.65) node[right]{$q$};
\draw [color=qqwuqq] ( 2.35,-2.35) node[right]{$q$};
\draw [color=qqwuqq] (-1.65,-2.35) node[right]{$q$};
\draw [color=ffqqqq] (-1.8,-1.8) node{$r$};
\draw [color=ffqqqq] (-1.65, 2.35) node[above]{$r$};
\draw [color=ffqqqq] ( 2.35, 2.35) node[right]{$r$};
\draw [color=ffqqqq] ( 2.35,-1.65) node[right]{$r$};
\draw [color=ffqqqq] ( 1.8, 1.8) node{$s$};
\draw [color=ffqqqq] ( 1.65,-2.35) node[below]{$s$};
\draw [color=ffqqqq] (-2.35,-2.35) node[left]{$s$};
\draw [color=ffqqqq] (-2.35, 1.65) node[left]{$s$};
\draw [blue][line width=1pt][postaction={on each segment={mid arrow=blue}}]
      (1.65,-1.65) to[out=180, in=90] (0.5, -2.1) node[below]{$\sip{1}{\ws}$};
\draw [dash pattern=on 2pt off 2pt] (0.5,-2)-- (0.5, 2);
\draw [blue][line width=1pt][postaction={on each segment={mid arrow=blue}}]
      (0.5, 1.9) node[above]{$\sip{1}{\ws}$} to[out=-90, in=135]
      (0.9, 0.9) node[above]{$\sip{2}{\ws}$} to[out=-45, in=180]
      (1.9, 0.5) node[right]{$\sip{3}{\ws}$};

\draw [dash pattern=on 2pt off 2pt] (-2, 0.5)-- (2, 0.5);
\draw [blue] (-2.1, 0.5) node[left]{$\sip{3}{\ws}$};
\draw [cyan][line width=1pt][postaction={on each segment={mid arrow=cyan}}]
      (0.05, 1.95) node[above]{$\sip{i_2}{\ws}$} to[out=-90, in=135]
      (0.65, 0.65)  node[left]{$\sip{i_2+1}{\ws}$};
\draw [cyan][line width=1pt][postaction={on each segment={mid arrow=cyan}}]
      (0.5,0.5)  node[below]{$\sip{i_3}{\ws}$} to[out=-45,in=180]
      (2,0); \draw [cyan] (1.8,0.0) node[right]{$\sip{i_3\!+\!1}{\ws}$};
\draw [black][opacity=0.4][line width=4pt] (0.5, 1.9) -- (1.65, 1.65);
\draw [black][opacity=0.4][line width=4pt] (0.5,-2.1) -- (1.65,-2.35);
\draw ( 2.1,-0.5) node[right]{\small\color{red}$a_1$};
\draw (-2.0, 0.0) node[left]{\small\color{red}$a_1$};
\draw (-0.5,-0.5) node[right]{\small\color{red}$a_2$};
\draw (-0.5, 1.9) node[above]{\small\color{red}$a_3$};
\draw (-0.5,-1.9) node[below]{\small\color{red}$a_3$};
\draw[white] (0,-3) node{(1)};
\end{tikzpicture}
\begin{tikzpicture} [scale=0.7]
\filldraw [black!20] (-2, 2) circle (0.5cm);
\filldraw [black!20] ( 2, 2) circle (0.5cm);
\filldraw [black!20] ( 2,-2) circle (0.5cm);
\filldraw [black!20] (-2,-2) circle (0.5cm);
\draw [line width=1pt] (-2, 2) circle (0.5cm);
\draw [line width=1pt] ( 2, 2) circle (0.5cm);
\draw [line width=1pt] ( 2,-2) circle (0.5cm);
\draw [line width=1pt] (-2,-2) circle (0.5cm);
\draw [dash pattern=on 2pt off 2pt] (-2, 2)-- ( 2, 2);
\draw [dash pattern=on 2pt off 2pt] ( 2, 2)-- ( 2,-2);
\draw [dash pattern=on 2pt off 2pt] ( 2,-2)-- (-2,-2);
\draw [dash pattern=on 2pt off 2pt] (-2, 2)-- (-2,-2);
\draw [line width=1pt,color=ffqqqq] (-2.35,1.65)-- (-1.65,-1.65);
\draw [line width=1pt,color=ffqqqq] (-1.65,2.35)-- (1.65,1.65);
\draw [line width=1pt,color=ffqqqq] (1.65,1.65)-- (2.35,-1.65);
\draw [line width=1pt,color=ffqqqq] (1.65,-2.35)-- (-1.65,-1.65);
\draw [line width=1pt,color=ffqqqq] (-1.65,-1.65)-- (1.65,1.65);
\fill [color=qqwuqq] (-1.65, 1.65) circle (2.5pt);
\fill [color=qqwuqq] ( 1.65,-1.65) circle (2.5pt);
\filldraw [color=white, draw=ffqqqq, line width=0.5pt] (-1.65, 2.35) circle (2.5pt);
\fill [color=qqwuqq] (-2.35, 2.35) circle (2.5pt);
\filldraw [color=white, draw=ffqqqq, line width=0.5pt] (-1.65,-1.65) circle (2.5pt);
\filldraw [color=white, draw=ffqqqq, line width=0.5pt] (-2.35,-2.35) circle (2.5pt);
\fill [color=qqwuqq] (-2.35,-1.65) circle (2.5pt);
\fill [color=qqwuqq] (-1.65,-2.35) circle (2.5pt);
\filldraw [color=white, draw=ffqqqq, line width=0.5pt] (-2.35, 1.65) circle (2.5pt);
\fill [color=qqwuqq] ( 1.65, 2.35) circle (2.5pt);
\filldraw [color=white, draw=ffqqqq, line width=0.5pt] ( 1.65, 1.65) circle (2.5pt);
\filldraw [color=white, draw=ffqqqq, line width=0.5pt] ( 2.35, 2.35) circle (2.5pt);
\fill [color=qqwuqq] ( 2.35, 1.65) circle (2.5pt);
\filldraw [color=white, draw=ffqqqq, line width=0.5pt] ( 2.35,-1.65) circle (2.5pt);
\fill [color=qqwuqq] ( 2.35,-2.35) circle (2.5pt);
\filldraw [color=white, draw=ffqqqq, line width=0.5pt] ( 1.65,-2.35) circle (2.5pt);
\draw [color=qqwuqq] ( 1.8,-1.8) node{$p$};
\draw [color=qqwuqq] (-2.35, 2.35) node[left]{$p$};
\draw [color=qqwuqq] (-2.35,-1.65) node[left]{$p$};
\draw [color=qqwuqq] ( 1.63, 2.35) node[right]{\tiny$p$};
\draw [color=qqwuqq] (-1.8, 1.8) node{$q$};
\draw [color=qqwuqq] ( 2.35, 1.65) node[right]{$q$};
\draw [color=qqwuqq] ( 2.35,-2.35) node[right]{$q$};
\draw [color=qqwuqq] (-1.65,-2.35) node[right]{$q$};
\draw [color=ffqqqq] (-1.8,-1.8) node{$r$};
\draw [color=ffqqqq] (-1.65, 2.35) node[above]{$r$};
\draw [color=ffqqqq] ( 2.35, 2.35) node[right]{$r$};
\draw [color=ffqqqq] ( 2.35,-1.65) node[right]{$r$};
\draw [color=ffqqqq] ( 1.8, 1.8) node{$s$};
\draw [color=ffqqqq] ( 1.65,-2.35) node[below]{$s$};
\draw [color=ffqqqq] (-2.35,-2.35) node[left]{$s$};
\draw [color=ffqqqq] (-2.35, 1.65) node[left]{$s$};
\draw [blue][line width=1pt][postaction={on each segment={mid arrow=blue}}]
      (1.65,-1.65) to[out=180, in=90] (0.5, -2.1) node[below]{$\sip{1}{\ws}$};
\draw [dash pattern=on 2pt off 2pt] (0.5,-2)-- (0.5, 2);
\draw [blue][line width=1pt][postaction={on each segment={mid arrow=blue}}]
      (0.5, 1.9) node[above]{$\sip{1}{\ws}$} to[out=-90, in=135]
      (0.9, 0.9) node[above]{$\sip{2}{\ws}$} to[out=-45, in=180]
      (1.9, 0.5) node[right]{$\sip{3}{\ws}$};
\draw [dash pattern=on 2pt off 2pt] (-2, 0.5)-- (2, 0.5);
\draw [blue][line width=1pt][postaction={on each segment={mid arrow=blue}}]
      (-2.1, 0.5) node[left]{$\sip{3}{\ws}$} to[out=0, in=220]
      (-1.3, 0.8);
\draw [blue][line width=1pt][postaction={on each segment={mid arrow=blue}}]
      (1.1, 1.75) to[out=-90, in=135]
      (1.3,1.3) to [out=-45, in=180]
      (1.75, 1.1) node[right]{$\sip{n(\ws)}{\ws}$};
\draw [blue][->] (1.1, 1.85) -- (1.1, 2.5); \draw[blue] (1.1,2.7) node{$\sip{n(\ws)\!-\!2}{\ws}$};
\draw [blue][->] (1.3, 1.25) -- (1.3,-0.5); \draw[blue] (1.3,-0.2) node[below]{$\sip{n(\ws)\!-\!1}{\ws}$};
\draw [blue][line width=1pt][postaction={on each segment={mid arrow=blue}}]
      (-2.2, 1)  to[out=0,in=-90] (-1.65,1.65);
\draw [blue] (-2.2, 1) node[left]{$\sip{n(\ws)}{\ws}$};
\draw [cyan][line width=1pt][postaction={on each segment={mid arrow=cyan}}]
      (0.05, 1.95) to[out=-90, in=135] (0.65, 0.65);
\draw [cyan][line width=1pt][postaction={on each segment={mid arrow=cyan}}]
      (0.5,0.5)  to[out=-45,in=180] (2,0);
\draw [black][opacity=0.4][line width=4pt] (0.5, 1.9) -- (1.65, 1.65);
\draw [black][opacity=0.4][line width=4pt] (0.5,-2.1) -- (1.65,-2.35);
\draw ( 2.1,-0.5) node[right]{\small\color{red}$a_1$};
\draw (-2.0, 0.0) node[left]{\small\color{red}$a_1$};
\draw (-0.5,-0.5) node[right]{\small\color{red}$a_2$};
\draw (-0.5, 1.9) node[above]{\small\color{red}$a_3$};
\draw (-0.5,-1.9) node[below]{\small\color{red}$a_3$};
\end{tikzpicture}
\caption{Cases for Lemma~\ref{lem:wg}}
\label{fig-star Type 6}
\end{center}
\end{figure}
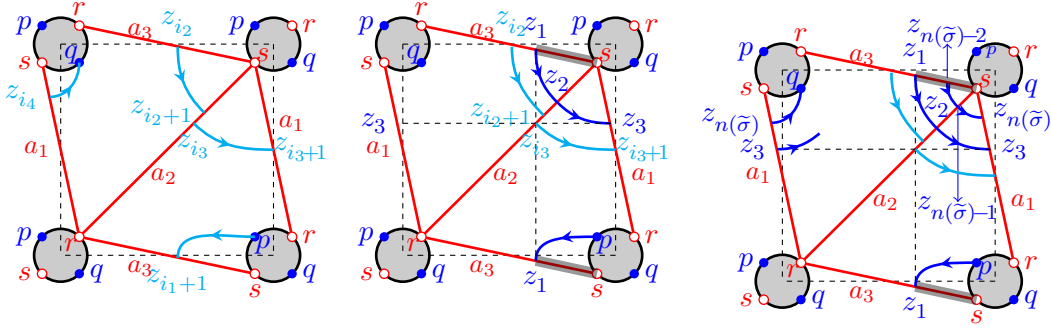

\end{lemma}

\begin{proof}
    By (A2), we have $\snum{1}{\ws}=\snum{i_1+1}{\ws}=\snum{i_2}{\ws}=3$. Note that $\ws$ does not cross the interior of the segment of $\wa_3$ between $\sip{1}{\ws}$ and $s$, see the shadow part in the second picture of Figure~\ref{fig-star Type 6}, because otherwise $\sigma(1)=p$, a contradiction. So $\sip{i_2}{\ws}$ is between $r$ and $\sip{1}{\ws}$. Hence $\snum{2}{\ws}=2$. Similarly, $\sip{i_3}{\ws}$ is between $r$ and $\sip{2}{\ws}$. So we have $\snum{3}{\ws}=1$. If $n(\ws)>3$, then $\sip{n(\ws)}{\ws}=\sip{i_4}{\ws}$ is in the interior of the segment of $\wa_1$ between $\sip{3}{\ws}$ and $s$, see the last picture of Figure~\ref{fig-star Type 6}. This implies that $\sip{n(\ws)-2}{\ws}$ is in the interior of the segment $\wa_3$ between $\sip{1}{\ws}$ and $s$, a contradiction. Thus, $n(\ws)=3$ and hence $\ws$ is homotopic to $\tgamma$.
\end{proof}

\section{Classification of indecomposable silting complexes}\label{sect:cls}

In this section, we first provide a classification of indecomposable presilting complexes in $D^b(\modcat\alg)$. Recall that $\SGOCM(\SURF)$ is the set of simple graded arcs on $\SURF$ whose endpoints are in $\Marked$.

\begin{proposition}\label{prop:presilt.ind.obj.}
There exists a bijection from the set $\{\ws\in\SGOCM(\SURF)\mid\sigma(0)\neq\sigma(1)\}$ to the set of isoclasses of indecomposable presilting complexes in $D^b(\modcat \alg)$.
\end{proposition}

\begin{proof}
By Theorem~\ref{thm:ops}, it suffices to show that for any $\ws\in\GOCM(\SURF)$, if $X(\ws)$ is presilting, then $\ws$ is simple and $\sigma(0)\neq\sigma(1)$. If $\ws$ contains a self-intersection in $\Surf\setminus\partial\Surf$, taking two representatives $\ws'$ and $\ws''$ in the homotopy class of $\ws$ which are in a minimal position, there is an intersection $\ip{}{}$ between $\ws'$ and $\ws''$ in $\Surf\setminus\partial\Surf$. By \eqref{eq:hkk}, either $\ii_{\ip{}{}}(\ws',\ws'')>0$ or $\ii_{\ip{}{}}(\ws'',\ws')>0$. In each case, by Theorem~\ref{thm:ops}~(3), we have $\Hom_{D^b(\modcat\alg)}(X(\ws),X(\ws)[d])\neq 0$ for some $d>0$. Thus, $X(\ws)$ is not presilting. Hence we only need to show that for any $\ws\in\SGOCM(\SURF)$, if $\sigma(0)=\sigma(1)$, then $X(\ws)$ is not presilting. By Theorem~\ref{thm:ops}, this is equivalent to showing $\ii_{\ws(0)}(\ws,\ws)>0$.

By symmetry, we may assume $\sigma(0)=\sigma(1)=p$, and reversing the orientation of $\ws$ if necessary, we may assume (S.R.E.). By Lemma~\ref{lem:circle} and the second formula in Lemma~\ref{lem:int ind}, we may assume that $\ws$ does not contain a circle. Let $0=i_1<i_2<\cdots<i_t=n(\ws)$ be a simplest sequence of $\ws$. Since (S.R.E.), by the second formula in Lemma~\ref{lem:int ind}, $\ii_{\ws(0)}(\ws,\ws)=\sdii{i_t}{\ws}-\sdii{i_1}{\ws}$. Note that $\gamma(0)=\gamma(1)$ implies that $t$ is odd. There are the following cases.

\begin{itemize}
    \item[Case 1] $\snum{i_2}{\ws}=1$.
    \begin{itemize}
        \item[Case 1.1] $\snum{i_3}{\ws}=2$. If $t=3$, see the first picture of \Pic\ref{fig:case 1.1} (where we label $\ws_{i_k,i_k+1}$ by $\circled{k}$, the same below), then using the first formula in Lemma~\ref{lem:int ind} (the same below), we have $\sdii{i_t}{\ws}-\sdii{i_1}{\ws}=1>0$ as required. Now we assume $t>3$. By Lemma~\ref{lem:cancel}, $\snum{i_4}{\ws}=3$ and $\snum{i_5}{\ws}=1$. If $t=5$, see the second picture of \Pic\ref{fig:case 1.1}, then $\sdii{i_t}{\ws}-\sdii{i_1}{\ws}=1>0$ as required. Now we assume $t>5$. By Lemma~\ref{lem:cancel}, $\snum{i_6}{\ws}=2$ and $\snum{i_7}{\ws}=3$. If $t=7$, see the third picture of \Pic\ref{fig:case 1.1}, then $\sdii{i_t}{\ws}-\sdii{i_1}{\ws}=3>0$ as required. If $t> 7$, by Lemma~\ref{lem:cancel}, $\snum{i_8}{\ws}=1$, see the fourth picture of \Pic\ref{fig:case 1.1}, which contradicts Lemma~\ref{lem:four}.

\begin{figure}[htbp]
\begin{center}
\definecolor{ffqqqq}{rgb}{1,0,0}
\definecolor{qqwuqq}{rgb}{0,0,1}
\begin{tikzpicture}[scale=0.55]
\filldraw [black!20] (-2, 2) circle (0.5cm);
\filldraw [black!20] ( 2, 2) circle (0.5cm);
\filldraw [black!20] ( 2,-2) circle (0.5cm);
\filldraw [black!20] (-2,-2) circle (0.5cm);
\draw [line width=1pt] (-2, 2) circle (0.5cm);
\draw [line width=1pt] ( 2, 2) circle (0.5cm);
\draw [line width=1pt] ( 2,-2) circle (0.5cm);
\draw [line width=1pt] (-2,-2) circle (0.5cm);
\draw [dash pattern=on 2pt off 2pt] (-2, 2)-- ( 2, 2);
\draw [dash pattern=on 2pt off 2pt] ( 2, 2)-- ( 2,-2);
\draw [dash pattern=on 2pt off 2pt] ( 2,-2)-- (-2,-2);
\draw [dash pattern=on 2pt off 2pt] (-2, 2)-- (-2,-2);
\draw [line width=1pt,color=ffqqqq] (-2.35,1.65)-- (-1.65,-1.65);
\draw [line width=1pt,color=ffqqqq] (-1.65,2.35)-- (1.65,1.65);
\draw [line width=1pt,color=ffqqqq] (1.65,1.65)-- (2.35,-1.65);
\draw [line width=1pt,color=ffqqqq] (1.65,-2.35)-- (-1.65,-1.65);
\draw [line width=1pt,color=ffqqqq] (-1.65,-1.65)-- (1.65,1.65);
\fill [color=qqwuqq] (-1.65, 1.65) circle (2.5pt);
\fill [color=qqwuqq] ( 1.65,-1.65) circle (2.5pt);
\filldraw [color=white, draw=ffqqqq, line width=0.5pt] (-1.65, 2.35) circle (2.5pt);
\fill [color=qqwuqq] (-2.35, 2.35) circle (2.5pt);
\filldraw [color=white, draw=ffqqqq, line width=0.5pt] (-1.65,-1.65) circle (2.5pt);
\filldraw [color=white, draw=ffqqqq, line width=0.5pt] (-2.35,-2.35) circle (2.5pt);
\fill [color=qqwuqq] (-2.35,-1.65) circle (2.5pt);
\fill [color=qqwuqq] (-1.65,-2.35) circle (2.5pt);
\filldraw [color=white, draw=ffqqqq, line width=0.5pt] (-2.35, 1.65) circle (2.5pt);
\fill [color=qqwuqq] ( 1.65, 2.35) circle (2.5pt);
\filldraw [color=white, draw=ffqqqq, line width=0.5pt] ( 1.65, 1.65) circle (2.5pt);
\filldraw [color=white, draw=ffqqqq, line width=0.5pt] ( 2.35, 2.35) circle (2.5pt);
\fill [color=qqwuqq] ( 2.35, 1.65) circle (2.5pt);
\filldraw [color=white, draw=ffqqqq, line width=0.5pt] ( 2.35,-1.65) circle (2.5pt);
\fill [color=qqwuqq] ( 2.35,-2.35) circle (2.5pt);
\filldraw [color=white, draw=ffqqqq, line width=0.5pt] ( 1.65,-2.35) circle (2.5pt);
\draw [color=qqwuqq] ( 1.8,-1.8) node{$p$};
\draw [color=qqwuqq] (-1.8, 1.8) node{$q$};
\draw [color=ffqqqq] (-1.8,-1.8) node{$r$};
\draw [color=ffqqqq] (-1.65, 2.35) node[above]{$r$};
\draw [color=ffqqqq] ( 2.35,-1.65) node[right]{$r$};
\draw [color=ffqqqq] ( 1.8, 1.8) node{$s$};
\draw [color=ffqqqq] ( 1.65,-2.35) node[below]{$s$};
\draw [color=ffqqqq] (-2.35, 1.65) node[left]{$s$};
\draw [blue][line width=1pt][postaction={on each segment={mid arrow=blue}}]
       (1.65,-1.65) to[out=90, in=180] (2,0);
\draw [blue](1.3, -0.6) node{\tiny$\circled{1}$};
\draw [blue][line width=1pt][postaction={on each segment={mid arrow=blue}}]
      (-2.1,0.5) to[out=0, in=135] (-0.3, -0.3);
\draw [blue](-1.2, 0.7) node{\tiny$\circled{2}$};
\draw [blue][line width=1pt][postaction={on each segment={mid arrow=blue}}]
     (-0.1, -0.1) to[out=-45, in=180] (1.64,-1.64);
\draw [blue](0.55, -1.5) node{\tiny$\circled{3}$};
\draw ( 2.0,-0.0) node[right]{\small\color{red}$a_1$};
\draw (-2.0, 0.0) node[left]{\small\color{red}$a_1$};
\draw (-0.0,-0.0) node[right]{\small\color{red}$a_2$};
\draw (-0.0, 2.0) node[above]{\small\color{red}$a_3$};
\draw (-0.0,-2.0) node[below]{\small\color{red}$a_3$};
\end{tikzpicture}
\begin{tikzpicture}[scale=0.55]
\filldraw [black!20] (-2, 2) circle (0.5cm);
\filldraw [black!20] ( 2, 2) circle (0.5cm);
\filldraw [black!20] ( 2,-2) circle (0.5cm);
\filldraw [black!20] (-2,-2) circle (0.5cm);
\draw [line width=1pt] (-2, 2) circle (0.5cm);
\draw [line width=1pt] ( 2, 2) circle (0.5cm);
\draw [line width=1pt] ( 2,-2) circle (0.5cm);
\draw [line width=1pt] (-2,-2) circle (0.5cm);
\draw [dash pattern=on 2pt off 2pt] (-2, 2)-- ( 2, 2);
\draw [dash pattern=on 2pt off 2pt] ( 2, 2)-- ( 2,-2);
\draw [dash pattern=on 2pt off 2pt] ( 2,-2)-- (-2,-2);
\draw [dash pattern=on 2pt off 2pt] (-2, 2)-- (-2,-2);
\draw [line width=1pt,color=ffqqqq] (-2.35,1.65)-- (-1.65,-1.65);
\draw [line width=1pt,color=ffqqqq] (-1.65,2.35)-- (1.65,1.65);
\draw [line width=1pt,color=ffqqqq] (1.65,1.65)-- (2.35,-1.65);
\draw [line width=1pt,color=ffqqqq] (1.65,-2.35)-- (-1.65,-1.65);
\draw [line width=1pt,color=ffqqqq] (-1.65,-1.65)-- (1.65,1.65);
\fill [color=qqwuqq] (-1.65, 1.65) circle (2.5pt);
\fill [color=qqwuqq] ( 1.65,-1.65) circle (2.5pt);
\filldraw [color=white, draw=ffqqqq, line width=0.5pt] (-1.65, 2.35) circle (2.5pt);
\fill [color=qqwuqq] (-2.35, 2.35) circle (2.5pt);
\filldraw [color=white, draw=ffqqqq, line width=0.5pt] (-1.65,-1.65) circle (2.5pt);
\filldraw [color=white, draw=ffqqqq, line width=0.5pt] (-2.35,-2.35) circle (2.5pt);
\fill [color=qqwuqq] (-2.35,-1.65) circle (2.5pt);
\fill [color=qqwuqq] (-1.65,-2.35) circle (2.5pt);
\filldraw [color=white, draw=ffqqqq, line width=0.5pt] (-2.35, 1.65) circle (2.5pt);
\fill [color=qqwuqq] ( 1.65, 2.35) circle (2.5pt);
\filldraw [color=white, draw=ffqqqq, line width=0.5pt] ( 1.65, 1.65) circle (2.5pt);
\filldraw [color=white, draw=ffqqqq, line width=0.5pt] ( 2.35, 2.35) circle (2.5pt);
\fill [color=qqwuqq] ( 2.35, 1.65) circle (2.5pt);
\filldraw [color=white, draw=ffqqqq, line width=0.5pt] ( 2.35,-1.65) circle (2.5pt);
\fill [color=qqwuqq] ( 2.35,-2.35) circle (2.5pt);
\filldraw [color=white, draw=ffqqqq, line width=0.5pt] ( 1.65,-2.35) circle (2.5pt);
\draw [color=qqwuqq] ( 1.8,-1.8) node{$p$};
\draw [color=qqwuqq] (-1.8, 1.8) node{$q$};
\draw [color=ffqqqq] (-1.8,-1.8) node{$r$};
\draw [color=ffqqqq] (-1.65, 2.35) node[above]{$r$};
\draw [color=ffqqqq] ( 2.35,-1.65) node[right]{$r$};
\draw [color=ffqqqq] ( 1.8, 1.8) node{$s$};
\draw [color=ffqqqq] ( 1.65,-2.35) node[below]{$s$};
\draw [color=ffqqqq] (-2.35, 1.65) node[left]{$s$};
\draw [blue][line width=1pt][postaction={on each segment={mid arrow=blue}}]
      (1.65,-1.65) to[out=90, in=180] (2.1,-0.5);
\draw [blue](1.7, -0.3) node{\tiny$\circled{1}$};
\draw [blue][line width=1pt][postaction={on each segment={mid arrow=blue}}]
      (-1.9,-0.5) to[out=0, in=135] (-0.95,-0.95);
\draw [blue](-1.2, -0.5) node{\tiny$\circled{2}$};
\draw [blue][line width=1pt][postaction={on each segment={mid arrow=blue}}]
      (-0.7, -0.7) to[out=-45, in=90] (-0.2,-2);
\draw [blue](-0,-1.4) node{\tiny$\circled{3}$};
\draw [blue][line width=1pt][postaction={on each segment={mid arrow=blue}}]
      (0,2) to[out=-90, in=0] (-2,0);
\draw [blue](-0.4, 0.4) node{\tiny$\circled{4}$};
\draw [blue][line width=1pt][postaction={on each segment={mid arrow=blue}}]
      (1.8,1) to[out=180, in=135] (1.65,-1.65);
\draw [blue](0.7,-0.4) node{\tiny$\circled{5}$};
\draw ( 2.0,-0.0) node[right]{\small\color{red}$a_1$};
\draw (-2.0, 0.0) node[left]{\small\color{red}$a_1$};
\draw (-0.0,-0.0) node[right]{\small\color{red}$a_2$};
\draw (-0.0, 2.0) node[above]{\small\color{red}$a_3$};
\draw (-0.0,-2.0) node[below]{\small\color{red}$a_3$};
\end{tikzpicture}
\begin{tikzpicture}[scale=0.55]
\filldraw [black!20] (-2, 2) circle (0.5cm);
\filldraw [black!20] ( 2, 2) circle (0.5cm);
\filldraw [black!20] ( 2,-2) circle (0.5cm);
\filldraw [black!20] (-2,-2) circle (0.5cm);
\draw [line width=1pt] (-2, 2) circle (0.5cm);
\draw [line width=1pt] ( 2, 2) circle (0.5cm);
\draw [line width=1pt] ( 2,-2) circle (0.5cm);
\draw [line width=1pt] (-2,-2) circle (0.5cm);
\draw [dash pattern=on 2pt off 2pt] (-2, 2)-- ( 2, 2);
\draw [dash pattern=on 2pt off 2pt] ( 2, 2)-- ( 2,-2);
\draw [dash pattern=on 2pt off 2pt] ( 2,-2)-- (-2,-2);
\draw [dash pattern=on 2pt off 2pt] (-2, 2)-- (-2,-2);
\draw [line width=1pt,color=ffqqqq] (-2.35,1.65)-- (-1.65,-1.65);
\draw [line width=1pt,color=ffqqqq] (-1.65,2.35)-- (1.65,1.65);
\draw [line width=1pt,color=ffqqqq] (1.65,1.65)-- (2.35,-1.65);
\draw [line width=1pt,color=ffqqqq] (1.65,-2.35)-- (-1.65,-1.65);
\draw [line width=1pt,color=ffqqqq] (-1.65,-1.65)-- (1.65,1.65);
\fill [color=qqwuqq] (-1.65, 1.65) circle (2.5pt);
\fill [color=qqwuqq] ( 1.65,-1.65) circle (2.5pt);
\filldraw [color=white, draw=ffqqqq, line width=0.5pt] (-1.65, 2.35) circle (2.5pt);
\fill [color=qqwuqq] (-2.35, 2.35) circle (2.5pt);
\filldraw [color=white, draw=ffqqqq, line width=0.5pt] (-1.65,-1.65) circle (2.5pt);
\filldraw [color=white, draw=ffqqqq, line width=0.5pt] (-2.35,-2.35) circle (2.5pt);
\fill [color=qqwuqq] (-2.35,-1.65) circle (2.5pt);
\fill [color=qqwuqq] (-1.65,-2.35) circle (2.5pt);
\filldraw [color=white, draw=ffqqqq, line width=0.5pt] (-2.35, 1.65) circle (2.5pt);
\fill [color=qqwuqq] ( 1.65, 2.35) circle (2.5pt);
\filldraw [color=white, draw=ffqqqq, line width=0.5pt] ( 1.65, 1.65) circle (2.5pt);
\filldraw [color=white, draw=ffqqqq, line width=0.5pt] ( 2.35, 2.35) circle (2.5pt);
\fill [color=qqwuqq] ( 2.35, 1.65) circle (2.5pt);
\filldraw [color=white, draw=ffqqqq, line width=0.5pt] ( 2.35,-1.65) circle (2.5pt);
\fill [color=qqwuqq] ( 2.35,-2.35) circle (2.5pt);
\filldraw [color=white, draw=ffqqqq, line width=0.5pt] ( 1.65,-2.35) circle (2.5pt);
\draw [color=qqwuqq] ( 1.8,-1.8) node{$p$};
\draw [color=qqwuqq] (-1.8, 1.8) node{$q$};
\draw [color=ffqqqq] (-1.8,-1.8) node{$r$};
\draw [color=ffqqqq] (-1.65, 2.35) node[above]{$r$};
\draw [color=ffqqqq] ( 2.35,-1.65) node[right]{$r$};
\draw [color=ffqqqq] ( 1.8, 1.8) node{$s$};
\draw [color=ffqqqq] ( 1.65,-2.35) node[below]{$s$};
\draw [color=ffqqqq] (-2.35, 1.65) node[left]{$s$};
\draw [blue][line width=1pt][postaction={on each segment={mid arrow=blue}}]
      (1.65,-1.65) to[out=90, in=180] (2.2,-0.8);
\draw [blue](1.8, -0.6) node{\tiny$\circled{1}$};
\draw [blue][line width=1pt][postaction={on each segment={mid arrow=blue}}]
      (-1.9,-0.4) to[out=0, in=135] (-0.9, -0.9);
\draw [blue](-1.2, -0.4) node{\tiny$\circled{2}$};
\draw [blue][line width=1pt][postaction={on each segment={mid arrow=blue}}]
      (-0.7, -0.7) to[out=-45, in=90] (-0.2,-2);
\draw [blue](-0,-1.4) node{\tiny$\circled{3}$};
\draw [blue][line width=1pt][postaction={on each segment={mid arrow=blue}}]
      (-0.6,2.1) to[out=-90, in=0] (-2.1,0.6);
\draw [blue](-0.8, 0.8) node{\tiny$\circled{4}$};
\draw [blue][line width=1pt][postaction={on each segment={mid arrow=blue}}]
      (1.9, 0.4) to[out=180, in=-45] (0.9, 0.9);
\draw [blue](1.2, 0.2) node{\tiny$\circled{5}$};
\draw [blue][line width=1pt][postaction={on each segment={mid arrow=blue}}]
      (0.7, 0.7) to[out=135, in=-90] (0.2, 2);
\draw [blue]( 0, 1.3) node{\tiny$\circled{6}$};
\draw [blue][line width=1pt][postaction={on each segment={mid arrow=blue}}]
      (0.8,-2.2) to[out=90, in=180] (1.65,-1.65);
\draw [blue](0.6, -1.8) node{\tiny$\circled{7}$};
\draw ( 2.0,-0.0) node[right]{\small\color{red}$a_1$};
\draw (-2.0, 0.0) node[left]{\small\color{red}$a_1$};
\draw (-0.0,-0.0) node[right]{\small\color{red}$a_2$};
\draw (-0.0, 2.0) node[above]{\small\color{red}$a_3$};
\draw (-0.0,-2.0) node[below]{\small\color{red}$a_3$};
\end{tikzpicture}
\begin{tikzpicture}[scale=0.55]
\filldraw [black!20] (-2, 2) circle (0.5cm);
\filldraw [black!20] ( 2, 2) circle (0.5cm);
\filldraw [black!20] ( 2,-2) circle (0.5cm);
\filldraw [black!20] (-2,-2) circle (0.5cm);
\draw [line width=1pt] (-2, 2) circle (0.5cm);
\draw [line width=1pt] ( 2, 2) circle (0.5cm);
\draw [line width=1pt] ( 2,-2) circle (0.5cm);
\draw [line width=1pt] (-2,-2) circle (0.5cm);
\draw [dash pattern=on 2pt off 2pt] (-2, 2)-- ( 2, 2);
\draw [dash pattern=on 2pt off 2pt] ( 2, 2)-- ( 2,-2);
\draw [dash pattern=on 2pt off 2pt] ( 2,-2)-- (-2,-2);
\draw [dash pattern=on 2pt off 2pt] (-2, 2)-- (-2,-2);
\draw [line width=1pt,color=ffqqqq] (-2.35,1.65)-- (-1.65,-1.65);
\draw [line width=1pt,color=ffqqqq] (-1.65,2.35)-- (1.65,1.65);
\draw [line width=1pt,color=ffqqqq] (1.65,1.65)-- (2.35,-1.65);
\draw [line width=1pt,color=ffqqqq] (1.65,-2.35)-- (-1.65,-1.65);
\draw [line width=1pt,color=ffqqqq] (-1.65,-1.65)-- (1.65,1.65);
\fill [color=qqwuqq] (-1.65, 1.65) circle (2.5pt);
\fill [color=qqwuqq] ( 1.65,-1.65) circle (2.5pt);
\filldraw [color=white, draw=ffqqqq, line width=0.5pt] (-1.65, 2.35) circle (2.5pt);
\fill [color=qqwuqq] (-2.35, 2.35) circle (2.5pt);
\filldraw [color=white, draw=ffqqqq, line width=0.5pt] (-1.65,-1.65) circle (2.5pt);
\filldraw [color=white, draw=ffqqqq, line width=0.5pt] (-2.35,-2.35) circle (2.5pt);
\fill [color=qqwuqq] (-2.35,-1.65) circle (2.5pt);
\fill [color=qqwuqq] (-1.65,-2.35) circle (2.5pt);
\filldraw [color=white, draw=ffqqqq, line width=0.5pt] (-2.35, 1.65) circle (2.5pt);
\fill [color=qqwuqq] ( 1.65, 2.35) circle (2.5pt);
\filldraw [color=white, draw=ffqqqq, line width=0.5pt] ( 1.65, 1.65) circle (2.5pt);
\filldraw [color=white, draw=ffqqqq, line width=0.5pt] ( 2.35, 2.35) circle (2.5pt);
\fill [color=qqwuqq] ( 2.35, 1.65) circle (2.5pt);
\filldraw [color=white, draw=ffqqqq, line width=0.5pt] ( 2.35,-1.65) circle (2.5pt);
\fill [color=qqwuqq] ( 2.35,-2.35) circle (2.5pt);
\filldraw [color=white, draw=ffqqqq, line width=0.5pt] ( 1.65,-2.35) circle (2.5pt);
\draw [color=qqwuqq] ( 1.8,-1.8) node{$p$};
\draw [color=qqwuqq] (-1.8, 1.8) node{$q$};
\draw [color=ffqqqq] (-1.8,-1.8) node{$r$};
\draw [color=ffqqqq] (-1.65, 2.35) node[above]{$r$};
\draw [color=ffqqqq] ( 2.35,-1.65) node[right]{$r$};
\draw [color=ffqqqq] ( 1.8, 1.8) node{$s$};
\draw [color=ffqqqq] ( 1.65,-2.35) node[below]{$s$};
\draw [color=ffqqqq] (-2.35, 1.65) node[left]{$s$};
\draw [blue][line width=1pt][postaction={on each segment={mid arrow=blue}}]
      (1.65,-1.65) to[out=90, in=180] (2.2,-0.8);
\draw [blue](1.8, -0.7) node{\tiny$\circled{1}$};
\draw [blue][line width=1pt][postaction={on each segment={mid arrow=blue}}]
      (-1.9,-0.4) to[out=0, in=135] (-0.9, -0.9);
\draw [blue](-1.2, -0.3) node{\tiny$\circled{2}$};
\draw [blue][line width=1pt][postaction={on each segment={mid arrow=blue}}]
      (-0.7, -0.7) to[out=-45, in=90] (-0.2,-2);
\draw [blue](-0,-1.4) node{\tiny$\circled{3}$};
\draw [blue][line width=1pt][postaction={on each segment={mid arrow=blue}}]
      (-0.6,2.1) to[out=-90, in=0] (-2.1,0.6);
\draw [blue](-0.8, 0.7) node{\tiny$\circled{3}$};
\draw [blue][line width=1pt][postaction={on each segment={mid arrow=blue}}]
      (1.9, 0.4) to[out=180, in=-45] (0.9, 0.9);
\draw [blue](1.2, 0.2) node{\tiny$\circled{4}$};
\draw [blue][line width=1pt][postaction={on each segment={mid arrow=blue}}]
      (0.7, 0.7) to[out=135, in=-90] (0.2, 2);
\draw [blue]( 0, 1.4) node{\tiny$\circled{5}$};
\draw [blue][line width=1pt][postaction={on each segment={mid arrow=blue}}]
      (0.8,-2.2) to[out=90, in=180] (2,0);
\draw [blue](0.8, -0.5) node{\tiny$\circled{7}$};
\draw ( 2.0,-0.0) node[right]{\small\color{red}$a_1$};
\draw (-2.0, 0.0) node[left]{\small\color{red}$a_1$};
\draw (-0.0,-0.0) node[right]{\small\color{red}$a_2$};
\draw (-0.0, 2.0) node[above]{\small\color{red}$a_3$};
\draw (-0.0,-2.0) node[below]{\small\color{red}$a_3$};
\end{tikzpicture}
\caption{Case 1.1}
\label{fig:case 1.1}
\end{center}
\end{figure}

        \item[Case 1.2] $\snum{i_3}{\ws}=3$. If $t=3$, see the first picture of \Pic\ref{fig:case 1.2}, then $\sdii{i_t}{\ws}-\sdii{i_1}{\ws}=1>0$ as required. Now we assume $t>3$. By Lemma~\ref{lem:cancel}, $\snum{i_4}{\ws}=2$ and $\snum{i_5}{\ws}=1$. If $t=5$, since (S.R.E.), we are in the situation shown in the second picture of \Pic\ref{fig:case 1.2}, which contradicts Lemma~\ref{lemm:self-int Type}. So $t>5$. By Lemma~\ref{lem:cancel}, $\snum{i_6}{\ws}=3$ and $\snum{i_7}{\ws}=2$. If $t=7$, then $\ws_{i_7,i_7+1}$ crosses $\ws_{i_5,i_5+1}$, see the third picture of \Pic\ref{fig:case 1.2}, a contradiction. So $t>7$. By Lemma~\ref{lem:cancel}, $\snum{i_8}{\ws}=1$, see the fourth picture of \Pic\ref{fig:case 1.2}, which contradicts Lemma~\ref{lem:four}.

\begin{figure}[htbp]
\begin{center}
\definecolor{ffqqqq}{rgb}{1,0,0}
\definecolor{qqwuqq}{rgb}{0,0,1}
\begin{tikzpicture}[scale=0.55]
\filldraw [black!20] (-2, 2) circle (0.5cm);
\filldraw [black!20] ( 2, 2) circle (0.5cm);
\filldraw [black!20] ( 2,-2) circle (0.5cm);
\filldraw [black!20] (-2,-2) circle (0.5cm);
\draw [line width=1pt] (-2, 2) circle (0.5cm);
\draw [line width=1pt] ( 2, 2) circle (0.5cm);
\draw [line width=1pt] ( 2,-2) circle (0.5cm);
\draw [line width=1pt] (-2,-2) circle (0.5cm);
\draw [dash pattern=on 2pt off 2pt] (-2, 2)-- ( 2, 2);
\draw [dash pattern=on 2pt off 2pt] ( 2, 2)-- ( 2,-2);
\draw [dash pattern=on 2pt off 2pt] ( 2,-2)-- (-2,-2);
\draw [dash pattern=on 2pt off 2pt] (-2, 2)-- (-2,-2);
\draw [line width=1pt,color=ffqqqq] (-2.35,1.65)-- (-1.65,-1.65);
\draw [line width=1pt,color=ffqqqq] (-1.65,2.35)-- (1.65,1.65);
\draw [line width=1pt,color=ffqqqq] (1.65,1.65)-- (2.35,-1.65);
\draw [line width=1pt,color=ffqqqq] (1.65,-2.35)-- (-1.65,-1.65);
\draw [line width=1pt,color=ffqqqq] (-1.65,-1.65)-- (1.65,1.65);
\fill [color=qqwuqq] (-1.65, 1.65) circle (2.5pt);
\fill [color=qqwuqq] ( 1.65,-1.65) circle (2.5pt);
\filldraw [color=white, draw=ffqqqq, line width=0.5pt] (-1.65, 2.35) circle (2.5pt);
\fill [color=qqwuqq] (-2.35, 2.35) circle (2.5pt);
\filldraw [color=white, draw=ffqqqq, line width=0.5pt] (-1.65,-1.65) circle (2.5pt);
\filldraw [color=white, draw=ffqqqq, line width=0.5pt] (-2.35,-2.35) circle (2.5pt);
\fill [color=qqwuqq] (-2.35,-1.65) circle (2.5pt);
\fill [color=qqwuqq] (-1.65,-2.35) circle (2.5pt);
\filldraw [color=white, draw=ffqqqq, line width=0.5pt] (-2.35, 1.65) circle (2.5pt);
\fill [color=qqwuqq] ( 1.65, 2.35) circle (2.5pt);
\filldraw [color=white, draw=ffqqqq, line width=0.5pt] ( 1.65, 1.65) circle (2.5pt);
\filldraw [color=white, draw=ffqqqq, line width=0.5pt] ( 2.35, 2.35) circle (2.5pt);
\fill [color=qqwuqq] ( 2.35, 1.65) circle (2.5pt);
\filldraw [color=white, draw=ffqqqq, line width=0.5pt] ( 2.35,-1.65) circle (2.5pt);
\fill [color=qqwuqq] ( 2.35,-2.35) circle (2.5pt);
\filldraw [color=white, draw=ffqqqq, line width=0.5pt] ( 1.65,-2.35) circle (2.5pt);
\draw [color=qqwuqq] ( 1.8,-1.8) node{$p$};
\draw [color=qqwuqq] (-1.8, 1.8) node{$q$};
\draw [color=ffqqqq] (-1.8,-1.8) node{$r$};
\draw [color=ffqqqq] (-1.65, 2.35) node[above]{$r$};
\draw [color=ffqqqq] ( 2.35,-1.65) node[right]{$r$};
\draw [color=ffqqqq] ( 1.8, 1.8) node{$s$};
\draw [color=ffqqqq] ( 1.65,-2.35) node[below]{$s$};
\draw [color=ffqqqq] (-2.35, 1.65) node[left]{$s$};
\draw [blue][line width=1pt][postaction={on each segment={mid arrow=blue}}]
      (1.65,-1.65) to[out=90,in=180] (2.05,-0.2);
\draw [blue](1.5,-0.2) node{\tiny$\circled{1}$};
\draw [blue][line width=1pt][postaction={on each segment={mid arrow=blue}}]
      (-2.1, 0.4) to[out=0,in=-90] (-0.4, 2.1);
\draw [blue](-0.65, 0.567) node{\tiny$\circled{2}$};
\draw [blue][line width=1pt][postaction={on each segment={mid arrow=blue}}]
      (0.2, -2.05) to[out=90,in=180] (1.65,-1.65);
\draw [blue](0.2, -1.4) node{\tiny$\circled{3}$};
\draw ( 2.0,-0.0) node[right]{\small\color{red}$a_1$};
\draw (-2.0, 0.0) node[left]{\small\color{red}$a_1$};
\draw (-0.0,-0.0) node[right]{\small\color{red}$a_2$};
\draw (-0.0, 2.0) node[above]{\small\color{red}$a_3$};
\draw (-0.0,-2.0) node[below]{\small\color{red}$a_3$};
\end{tikzpicture}
\begin{tikzpicture}[scale=0.55]
\filldraw [black!20] (-2, 2) circle (0.5cm);
\filldraw [black!20] ( 2, 2) circle (0.5cm);
\filldraw [black!20] ( 2,-2) circle (0.5cm);
\filldraw [black!20] (-2,-2) circle (0.5cm);
\draw [line width=1pt] (-2, 2) circle (0.5cm);
\draw [line width=1pt] ( 2, 2) circle (0.5cm);
\draw [line width=1pt] ( 2,-2) circle (0.5cm);
\draw [line width=1pt] (-2,-2) circle (0.5cm);
\draw [dash pattern=on 2pt off 2pt] (-2, 2)-- ( 2, 2);
\draw [dash pattern=on 2pt off 2pt] ( 2, 2)-- ( 2,-2);
\draw [dash pattern=on 2pt off 2pt] ( 2,-2)-- (-2,-2);
\draw [dash pattern=on 2pt off 2pt] (-2, 2)-- (-2,-2);
\draw [line width=1pt,color=ffqqqq] (-2.35,1.65)-- (-1.65,-1.65);
\draw [line width=1pt,color=ffqqqq] (-1.65,2.35)-- (1.65,1.65);
\draw [line width=1pt,color=ffqqqq] (1.65,1.65)-- (2.35,-1.65);
\draw [line width=1pt,color=ffqqqq] (1.65,-2.35)-- (-1.65,-1.65);
\draw [line width=1pt,color=ffqqqq] (-1.65,-1.65)-- (1.65,1.65);
\fill [color=qqwuqq] (-1.65, 1.65) circle (2.5pt);
\fill [color=qqwuqq] ( 1.65,-1.65) circle (2.5pt);
\filldraw [color=white, draw=ffqqqq, line width=0.5pt] (-1.65, 2.35) circle (2.5pt);
\fill [color=qqwuqq] (-2.35, 2.35) circle (2.5pt);
\filldraw [color=white, draw=ffqqqq, line width=0.5pt] (-1.65,-1.65) circle (2.5pt);
\filldraw [color=white, draw=ffqqqq, line width=0.5pt] (-2.35,-2.35) circle (2.5pt);
\fill [color=qqwuqq] (-2.35,-1.65) circle (2.5pt);
\fill [color=qqwuqq] (-1.65,-2.35) circle (2.5pt);
\filldraw [color=white, draw=ffqqqq, line width=0.5pt] (-2.35, 1.65) circle (2.5pt);
\fill [color=qqwuqq] ( 1.65, 2.35) circle (2.5pt);
\filldraw [color=white, draw=ffqqqq, line width=0.5pt] ( 1.65, 1.65) circle (2.5pt);
\filldraw [color=white, draw=ffqqqq, line width=0.5pt] ( 2.35, 2.35) circle (2.5pt);
\fill [color=qqwuqq] ( 2.35, 1.65) circle (2.5pt);
\filldraw [color=white, draw=ffqqqq, line width=0.5pt] ( 2.35,-1.65) circle (2.5pt);
\fill [color=qqwuqq] ( 2.35,-2.35) circle (2.5pt);
\filldraw [color=white, draw=ffqqqq, line width=0.5pt] ( 1.65,-2.35) circle (2.5pt);
\draw [color=qqwuqq] ( 1.8,-1.8) node{$p$};
\draw [color=qqwuqq] (-1.8, 1.8) node{$q$};
\draw [color=ffqqqq] (-1.8,-1.8) node{$r$};
\draw [color=ffqqqq] (-1.65, 2.35) node[above]{$r$};
\draw [color=ffqqqq] ( 2.35,-1.65) node[right]{$r$};
\draw [color=ffqqqq] ( 1.8, 1.8) node{$s$};
\draw [color=ffqqqq] ( 1.65,-2.35) node[below]{$s$};
\draw [color=ffqqqq] (-2.35, 1.65) node[left]{$s$};
\draw [blue][line width=1pt][postaction={on each segment={mid arrow=blue}}]
      (1.65,-1.65) to[out=90,in=180] (2.05,-0.2);
\draw [blue](1.7,-0.1) node{\tiny$\circled{1}$};
\draw [blue][line width=1pt][postaction={on each segment={mid arrow=blue}}]
      (-2.1, 0.4) to[out=0,in=-90] (-0.4, 2.1);
\draw [blue](-0.6, 0.6) node{\tiny$\circled{2}$};
\draw [blue][line width=1pt][postaction={on each segment={mid arrow=blue}}]
      (-0.4, -1.9) to[out=90,in=-45] (-0.8,-0.8);
\draw [blue](-0.1, -1.5) node{\tiny$\circled{3}$};
\draw [blue][line width=1pt][postaction={on each segment={mid arrow=blue}}]
      (-0.6,-0.6) to[out=135,in= 0] (-2,-0.2);
\draw [blue](-1, 0) node{\tiny$\circled{4}$};
\draw [blue][line width=1pt][postaction={on each segment={mid arrow=blue}}]
      ( 2, 0.6) to[out=180,in= 135] (1.65,-1.65);
\draw [blue](0.9,-0.5) node{\tiny$\circled{5}$};
\draw ( 2.0,-0.0) node[right]{\small\color{red}$a_1$};
\draw (-2.0, 0.0) node[left]{\small\color{red}$a_1$};
\draw (-0.0,-0.0) node[right]{\small\color{red}$a_2$};
\draw (-0.0, 2.0) node[above]{\small\color{red}$a_3$};
\draw (-0.0,-2.0) node[below]{\small\color{red}$a_3$};
\end{tikzpicture}
\begin{tikzpicture}[scale=0.55]
\filldraw [black!20] (-2, 2) circle (0.5cm);
\filldraw [black!20] ( 2, 2) circle (0.5cm);
\filldraw [black!20] ( 2,-2) circle (0.5cm);
\filldraw [black!20] (-2,-2) circle (0.5cm);
\draw [line width=1pt] (-2, 2) circle (0.5cm);
\draw [line width=1pt] ( 2, 2) circle (0.5cm);
\draw [line width=1pt] ( 2,-2) circle (0.5cm);
\draw [line width=1pt] (-2,-2) circle (0.5cm);
\draw [dash pattern=on 2pt off 2pt] (-2, 2)-- ( 2, 2);
\draw [dash pattern=on 2pt off 2pt] ( 2, 2)-- ( 2,-2);
\draw [dash pattern=on 2pt off 2pt] ( 2,-2)-- (-2,-2);
\draw [dash pattern=on 2pt off 2pt] (-2, 2)-- (-2,-2);
\draw [line width=1pt,color=ffqqqq] (-2.35,1.65)-- (-1.65,-1.65);
\draw [line width=1pt,color=ffqqqq] (-1.65,2.35)-- (1.65,1.65);
\draw [line width=1pt,color=ffqqqq] (1.65,1.65)-- (2.35,-1.65);
\draw [line width=1pt,color=ffqqqq] (1.65,-2.35)-- (-1.65,-1.65);
\draw [line width=1pt,color=ffqqqq] (-1.65,-1.65)-- (1.65,1.65);
\fill [color=qqwuqq] (-1.65, 1.65) circle (2.5pt);
\fill [color=qqwuqq] ( 1.65,-1.65) circle (2.5pt);
\filldraw [color=white, draw=ffqqqq, line width=0.5pt] (-1.65, 2.35) circle (2.5pt);
\fill [color=qqwuqq] (-2.35, 2.35) circle (2.5pt);
\filldraw [color=white, draw=ffqqqq, line width=0.5pt] (-1.65,-1.65) circle (2.5pt);
\filldraw [color=white, draw=ffqqqq, line width=0.5pt] (-2.35,-2.35) circle (2.5pt);
\fill [color=qqwuqq] (-2.35,-1.65) circle (2.5pt);
\fill [color=qqwuqq] (-1.65,-2.35) circle (2.5pt);
\filldraw [color=white, draw=ffqqqq, line width=0.5pt] (-2.35, 1.65) circle (2.5pt);
\fill [color=qqwuqq] ( 1.65, 2.35) circle (2.5pt);
\filldraw [color=white, draw=ffqqqq, line width=0.5pt] ( 1.65, 1.65) circle (2.5pt);
\filldraw [color=white, draw=ffqqqq, line width=0.5pt] ( 2.35, 2.35) circle (2.5pt);
\fill [color=qqwuqq] ( 2.35, 1.65) circle (2.5pt);
\filldraw [color=white, draw=ffqqqq, line width=0.5pt] ( 2.35,-1.65) circle (2.5pt);
\fill [color=qqwuqq] ( 2.35,-2.35) circle (2.5pt);
\filldraw [color=white, draw=ffqqqq, line width=0.5pt] ( 1.65,-2.35) circle (2.5pt);
\draw [color=qqwuqq] ( 1.8,-1.8) node{$p$};
\draw [color=qqwuqq] (-1.8, 1.8) node{$q$};
\draw [color=ffqqqq] (-1.8,-1.8) node{$r$};
\draw [color=ffqqqq] (-1.65, 2.35) node[above]{$r$};
\draw [color=ffqqqq] ( 2.35,-1.65) node[right]{$r$};
\draw [color=ffqqqq] ( 1.8, 1.8) node{$s$};
\draw [color=ffqqqq] ( 1.65,-2.35) node[below]{$s$};
\draw [color=ffqqqq] (-2.35, 1.65) node[left]{$s$};
\draw [blue][line width=1pt][postaction={on each segment={mid arrow=blue}}]
      (1.65,-1.65) to[out=90,in=180] (2.05,-0.2);
\draw [blue](1.7,-0.1) node{\tiny$\circled{1}$};
\draw [blue][line width=1pt][postaction={on each segment={mid arrow=blue}}]
      (-2.1, 0.4) to[out=0,in=-90] (-0.4, 2.1);
\draw [blue](-0.6, 0.6) node{\tiny$\circled{2}$};
\draw [blue][line width=1pt][postaction={on each segment={mid arrow=blue}}]
      (-0.4, -1.9) to[out=90,in=-45] (-0.8,-0.8);
\draw [blue](-0.1, -1.5) node{\tiny$\circled{3}$};
\draw [blue][line width=1pt][postaction={on each segment={mid arrow=blue}}]
      (-0.6,-0.6) to[out=135,in= 0] (-2,-0.2);
\draw [blue](-1, 0) node{\tiny$\circled{4}$};
\draw [blue][line width=1pt][postaction={on each segment={mid arrow=blue}}]
      ( 2, 0.4) to[out=180,in= 90] (0.5,-2.1);
\draw [blue](0.3,-1) node{\tiny$\circled{5}$};
\draw [blue][line width=1pt][postaction={on each segment={mid arrow=blue}}]
      (0.7, 1.85) to[out=-90,in= 135] (1.1, 1.1);
\draw [blue] (0.5, 1.2) node{\tiny$\circled{6}$};
\draw [blue][line width=1pt][postaction={on each segment={mid arrow=blue}}]
      (0.9, 0.9) -- (1.65,-1.65);
\draw [blue] (1.2,-1.35) node{\tiny$\circled{7}$};
\draw ( 2.0,-0.0) node[right]{\small\color{red}$a_1$};
\draw (-2.0, 0.0) node[left]{\small\color{red}$a_1$};
\draw (-0.0,-0.0) node[right]{\small\color{red}$a_2$};
\draw (-0.0, 2.0) node[above]{\small\color{red}$a_3$};
\draw (-0.0,-2.0) node[below]{\small\color{red}$a_3$};
\end{tikzpicture}
\begin{tikzpicture}[scale=0.55]
\filldraw [black!20] (-2, 2) circle (0.5cm);
\filldraw [black!20] ( 2, 2) circle (0.5cm);
\filldraw [black!20] ( 2,-2) circle (0.5cm);
\filldraw [black!20] (-2,-2) circle (0.5cm);
\draw [line width=1pt] (-2, 2) circle (0.5cm);
\draw [line width=1pt] ( 2, 2) circle (0.5cm);
\draw [line width=1pt] ( 2,-2) circle (0.5cm);
\draw [line width=1pt] (-2,-2) circle (0.5cm);
\draw [dash pattern=on 2pt off 2pt] (-2, 2)-- ( 2, 2);
\draw [dash pattern=on 2pt off 2pt] ( 2, 2)-- ( 2,-2);
\draw [dash pattern=on 2pt off 2pt] ( 2,-2)-- (-2,-2);
\draw [dash pattern=on 2pt off 2pt] (-2, 2)-- (-2,-2);
\draw [line width=1pt,color=ffqqqq] (-2.35,1.65)-- (-1.65,-1.65);
\draw [line width=1pt,color=ffqqqq] (-1.65,2.35)-- (1.65,1.65);
\draw [line width=1pt,color=ffqqqq] (1.65,1.65)-- (2.35,-1.65);
\draw [line width=1pt,color=ffqqqq] (1.65,-2.35)-- (-1.65,-1.65);
\draw [line width=1pt,color=ffqqqq] (-1.65,-1.65)-- (1.65,1.65);
\fill [color=qqwuqq] (-1.65, 1.65) circle (2.5pt);
\fill [color=qqwuqq] ( 1.65,-1.65) circle (2.5pt);
\filldraw [color=white, draw=ffqqqq, line width=0.5pt] (-1.65, 2.35) circle (2.5pt);
\fill [color=qqwuqq] (-2.35, 2.35) circle (2.5pt);
\filldraw [color=white, draw=ffqqqq, line width=0.5pt] (-1.65,-1.65) circle (2.5pt);
\filldraw [color=white, draw=ffqqqq, line width=0.5pt] (-2.35,-2.35) circle (2.5pt);
\fill [color=qqwuqq] (-2.35,-1.65) circle (2.5pt);
\fill [color=qqwuqq] (-1.65,-2.35) circle (2.5pt);
\filldraw [color=white, draw=ffqqqq, line width=0.5pt] (-2.35, 1.65) circle (2.5pt);
\fill [color=qqwuqq] ( 1.65, 2.35) circle (2.5pt);
\filldraw [color=white, draw=ffqqqq, line width=0.5pt] ( 1.65, 1.65) circle (2.5pt);
\filldraw [color=white, draw=ffqqqq, line width=0.5pt] ( 2.35, 2.35) circle (2.5pt);
\fill [color=qqwuqq] ( 2.35, 1.65) circle (2.5pt);
\filldraw [color=white, draw=ffqqqq, line width=0.5pt] ( 2.35,-1.65) circle (2.5pt);
\fill [color=qqwuqq] ( 2.35,-2.35) circle (2.5pt);
\filldraw [color=white, draw=ffqqqq, line width=0.5pt] ( 1.65,-2.35) circle (2.5pt);
\draw [color=qqwuqq] ( 1.8,-1.8) node{$p$};
\draw [color=qqwuqq] (-1.8, 1.8) node{$q$};
\draw [color=ffqqqq] (-1.8,-1.8) node{$r$};
\draw [color=ffqqqq] (-1.65, 2.35) node[above]{$r$};
\draw [color=ffqqqq] ( 2.35,-1.65) node[right]{$r$};
\draw [color=ffqqqq] ( 1.8, 1.8) node{$s$};
\draw [color=ffqqqq] ( 1.65,-2.35) node[below]{$s$};
\draw [color=ffqqqq] (-2.35, 1.65) node[left]{$s$};
\draw [blue][line width=1pt][postaction={on each segment={mid arrow=blue}}]
      (1.65,-1.65) to[out=90,in=180] (2.05,-0.2);
\draw [blue](1.7,-0.1) node{\tiny$\circled{1}$};
\draw [blue][line width=1pt][postaction={on each segment={mid arrow=blue}}]
      (-2.1, 0.4) to[out=0,in=-90] (-0.4, 2.1);
\draw [blue](-0.6, 0.6) node{\tiny$\circled{2}$};
\draw [blue][line width=1pt][postaction={on each segment={mid arrow=blue}}]
      (-0.4, -1.9) to[out=90,in=-45] (-0.8,-0.8);
\draw [blue](-0.1, -1.5) node{\tiny$\circled{3}$};
\draw [blue][line width=1pt][postaction={on each segment={mid arrow=blue}}]
      (-0.6,-0.6) to[out=135,in= 0] (-2,-0.2);
\draw [blue](-1, 0) node{\tiny$\circled{4}$};
\draw [blue][line width=1pt][postaction={on each segment={mid arrow=blue}}]
      ( 2, 0.4) to[out=180,in= 90] (0.5,-2.1);
\draw [blue](0.3,-1) node{\tiny$\circled{5}$};
\draw [blue][line width=1pt][postaction={on each segment={mid arrow=blue}}]
      (0.7, 1.85) to[out=-90,in= 135] (1.1, 1.1);
\draw [blue] (0.5, 1.3) node{\tiny$\circled{6}$};
\draw [blue][line width=1pt][postaction={on each segment={mid arrow=blue}}]
      (0.9, 0.9) to[out=-45,in= 180] (1.85, 0.7);
\draw [blue] (1.5, 0.9) node{\tiny$\circled{7}$};
\draw ( 2.0,-0.0) node[right]{\small\color{red}$a_1$};
\draw (-2.0, 0.0) node[left]{\small\color{red}$a_1$};
\draw (-0.0,-0.0) node[right]{\small\color{red}$a_2$};
\draw (-0.0, 2.0) node[above]{\small\color{red}$a_3$};
\draw (-0.0,-2.0) node[below]{\small\color{red}$a_3$};
\end{tikzpicture}
\caption{Case 1.2}
\label{fig:case 1.2}
\end{center}
\end{figure}

    \end{itemize}
    \item[Case 2] $\snum{i_2}{\ws}=2$.
    \begin{itemize}
        \item[Case 2.1] $\snum{i_3}{\ws}=1$. If $t=3$, see the first picture of \Pic\ref{fig:case 2}, then (S.L.E), a contradiction. So $t>3$. Then by Lemma~\ref{lem:cancel},  $\snum{i_4}{\ws}=3$. However, now $\ws_{i_3,i_3+1}$ crosses $\ws_{i_1,i_1+1}$, see the second picture of Figure~\ref{fig:case 2}, a contradiction.
        \item[Case 2.1] $\snum{i_3}{\ws}=3$. If $t=3$, see the third picture of Figure~\ref{fig:case 2}, then $\sdii{i_t}{\ws}-\sdii{i_1}{\ws}=1>0$ as required. Now we assume $t>3$. By Lemma~\ref{lem:cancel},  $\snum{i_4}{\ws}=1$. However, now $\ws_{i_3,i_3+1}$ crosses $\ws_{i_1,i_1+1}$, see the fourth picture of Figure~\ref{fig:case 2}, a contradiction.
    \end{itemize}

\begin{figure}[htbp]
\begin{center}
\definecolor{ffqqqq}{rgb}{1,0,0}
\definecolor{qqwuqq}{rgb}{0,0,1}
\begin{tikzpicture}[scale=0.55]
\filldraw [black!20] (-2, 2) circle (0.5cm);
\filldraw [black!20] ( 2, 2) circle (0.5cm);
\filldraw [black!20] ( 2,-2) circle (0.5cm);
\filldraw [black!20] (-2,-2) circle (0.5cm);
\draw [line width=1pt] (-2, 2) circle (0.5cm);
\draw [line width=1pt] ( 2, 2) circle (0.5cm);
\draw [line width=1pt] ( 2,-2) circle (0.5cm);
\draw [line width=1pt] (-2,-2) circle (0.5cm);
\draw [dash pattern=on 2pt off 2pt] (-2, 2)-- ( 2, 2);
\draw [dash pattern=on 2pt off 2pt] ( 2, 2)-- ( 2,-2);
\draw [dash pattern=on 2pt off 2pt] ( 2,-2)-- (-2,-2);
\draw [dash pattern=on 2pt off 2pt] (-2, 2)-- (-2,-2);
\draw [line width=1pt,color=ffqqqq] (-2.35,1.65)-- (-1.65,-1.65);
\draw [line width=1pt,color=ffqqqq] (-1.65,2.35)-- (1.65,1.65);
\draw [line width=1pt,color=ffqqqq] (1.65,1.65)-- (2.35,-1.65);
\draw [line width=1pt,color=ffqqqq] (1.65,-2.35)-- (-1.65,-1.65);
\draw [line width=1pt,color=ffqqqq] (-1.65,-1.65)-- (1.65,1.65);
\fill [color=qqwuqq] (-1.65, 1.65) circle (2.5pt);
\fill [color=qqwuqq] ( 1.65,-1.65) circle (2.5pt);
\filldraw [color=white, draw=ffqqqq, line width=0.5pt] (-1.65, 2.35) circle (2.5pt);
\fill [color=qqwuqq] (-2.35, 2.35) circle (2.5pt);
\filldraw [color=white, draw=ffqqqq, line width=0.5pt] (-1.65,-1.65) circle (2.5pt);
\filldraw [color=white, draw=ffqqqq, line width=0.5pt] (-2.35,-2.35) circle (2.5pt);
\fill [color=qqwuqq] (-2.35,-1.65) circle (2.5pt);
\fill [color=qqwuqq] (-1.65,-2.35) circle (2.5pt);
\filldraw [color=white, draw=ffqqqq, line width=0.5pt] (-2.35, 1.65) circle (2.5pt);
\fill [color=qqwuqq] ( 1.65, 2.35) circle (2.5pt);
\filldraw [color=white, draw=ffqqqq, line width=0.5pt] ( 1.65, 1.65) circle (2.5pt);
\filldraw [color=white, draw=ffqqqq, line width=0.5pt] ( 2.35, 2.35) circle (2.5pt);
\fill [color=qqwuqq] ( 2.35, 1.65) circle (2.5pt);
\filldraw [color=white, draw=ffqqqq, line width=0.5pt] ( 2.35,-1.65) circle (2.5pt);
\fill [color=qqwuqq] ( 2.35,-2.35) circle (2.5pt);
\filldraw [color=white, draw=ffqqqq, line width=0.5pt] ( 1.65,-2.35) circle (2.5pt);
\draw [color=qqwuqq] ( 1.8,-1.8) node{$p$};
\draw [color=qqwuqq] (-1.8, 1.8) node{$q$};
\draw [color=ffqqqq] (-1.8,-1.8) node{$r$};
\draw [color=ffqqqq] (-1.65, 2.35) node[above]{$r$};
\draw [color=ffqqqq] ( 2.35,-1.65) node[right]{$r$};
\draw [color=ffqqqq] ( 1.8, 1.8) node{$s$};
\draw [color=ffqqqq] ( 1.65,-2.35) node[below]{$s$};
\draw [color=ffqqqq] (-2.35, 1.65) node[left]{$s$};
\draw [blue][line width=1pt][postaction={on each segment={mid arrow=blue}}]
      (1.65,-1.65) to[out=135,in=-45] (0,0);
\draw [blue](1, -1) node[above]{\tiny$\circled{1}$};
\draw [blue][line width=1pt][postaction={on each segment={mid arrow=blue}}]
      (-0.5,-0.5) to[out=135,in= 0] (-2,-0);
\draw [blue](-1.5, 0.2) node{\tiny$\circled{2}$};
\draw [blue][line width=1pt][postaction={on each segment={mid arrow=blue}}]
      (2.05, -0.2) to[out=180,in=90] (1.65,-1.65);
\draw [blue](1.7, 0) node{\tiny$\circled{3}$};
\draw ( 2.0,-0.0) node[right]{\small\color{red}$a_1$};
\draw (-2.0, 0.0) node[left]{\small\color{red}$a_1$};
\draw (-0.0,-0.0) node[right]{\small\color{red}$a_2$};
\draw (-0.0, 2.0) node[above]{\small\color{red}$a_3$};
\draw (-0.0,-2.0) node[below]{\small\color{red}$a_3$};
\end{tikzpicture}
\begin{tikzpicture}[scale=0.55]
\filldraw [black!20] (-2, 2) circle (0.5cm);
\filldraw [black!20] ( 2, 2) circle (0.5cm);
\filldraw [black!20] ( 2,-2) circle (0.5cm);
\filldraw [black!20] (-2,-2) circle (0.5cm);
\draw [line width=1pt] (-2, 2) circle (0.5cm);
\draw [line width=1pt] ( 2, 2) circle (0.5cm);
\draw [line width=1pt] ( 2,-2) circle (0.5cm);
\draw [line width=1pt] (-2,-2) circle (0.5cm);
\draw [dash pattern=on 2pt off 2pt] (-2, 2)-- ( 2, 2);
\draw [dash pattern=on 2pt off 2pt] ( 2, 2)-- ( 2,-2);
\draw [dash pattern=on 2pt off 2pt] ( 2,-2)-- (-2,-2);
\draw [dash pattern=on 2pt off 2pt] (-2, 2)-- (-2,-2);
\draw [line width=1pt,color=ffqqqq] (-2.35,1.65)-- (-1.65,-1.65);
\draw [line width=1pt,color=ffqqqq] (-1.65,2.35)-- (1.65,1.65);
\draw [line width=1pt,color=ffqqqq] (1.65,1.65)-- (2.35,-1.65);
\draw [line width=1pt,color=ffqqqq] (1.65,-2.35)-- (-1.65,-1.65);
\draw [line width=1pt,color=ffqqqq] (-1.65,-1.65)-- (1.65,1.65);
\fill [color=qqwuqq] (-1.65, 1.65) circle (2.5pt);
\fill [color=qqwuqq] ( 1.65,-1.65) circle (2.5pt);
\filldraw [color=white, draw=ffqqqq, line width=0.5pt] (-1.65, 2.35) circle (2.5pt);
\fill [color=qqwuqq] (-2.35, 2.35) circle (2.5pt);
\filldraw [color=white, draw=ffqqqq, line width=0.5pt] (-1.65,-1.65) circle (2.5pt);
\filldraw [color=white, draw=ffqqqq, line width=0.5pt] (-2.35,-2.35) circle (2.5pt);
\fill [color=qqwuqq] (-2.35,-1.65) circle (2.5pt);
\fill [color=qqwuqq] (-1.65,-2.35) circle (2.5pt);
\filldraw [color=white, draw=ffqqqq, line width=0.5pt] (-2.35, 1.65) circle (2.5pt);
\fill [color=qqwuqq] ( 1.65, 2.35) circle (2.5pt);
\filldraw [color=white, draw=ffqqqq, line width=0.5pt] ( 1.65, 1.65) circle (2.5pt);
\filldraw [color=white, draw=ffqqqq, line width=0.5pt] ( 2.35, 2.35) circle (2.5pt);
\fill [color=qqwuqq] ( 2.35, 1.65) circle (2.5pt);
\filldraw [color=white, draw=ffqqqq, line width=0.5pt] ( 2.35,-1.65) circle (2.5pt);
\fill [color=qqwuqq] ( 2.35,-2.35) circle (2.5pt);
\filldraw [color=white, draw=ffqqqq, line width=0.5pt] ( 1.65,-2.35) circle (2.5pt);
\draw [color=qqwuqq] ( 1.8,-1.8) node{$p$};
\draw [color=qqwuqq] (-1.8, 1.8) node{$q$};
\draw [color=ffqqqq] (-1.8,-1.8) node{$r$};
\draw [color=ffqqqq] (-1.65, 2.35) node[above]{$r$};
\draw [color=ffqqqq] ( 2.35,-1.65) node[right]{$r$};
\draw [color=ffqqqq] ( 1.8, 1.8) node{$s$};
\draw [color=ffqqqq] ( 1.65,-2.35) node[below]{$s$};
\draw [color=ffqqqq] (-2.35, 1.65) node[left]{$s$};
\draw [blue][line width=1pt][postaction={on each segment={mid arrow=blue}}]
      (1.65,-1.65) to[out=135,in=-45] (0,0);
\draw [blue](1.5, -1.5) node[above]{\tiny$\circled{1}$};
\draw [blue][line width=1pt][postaction={on each segment={mid arrow=blue}}]
      (-0.5,-0.5) to[out=135,in= 0] (-2,-0);
\draw [blue](-1.5, 0.2) node{\tiny$\circled{2}$};
\draw [blue][line width=1pt][postaction={on each segment={mid arrow=blue}}]
      (2.05, -0.2) to[out=180,in=90] (0.2,-2.05);
\draw [blue](1.35, 0) node{\tiny$\circled{3}$};
\draw ( 2.0,-0.0) node[right]{\small\color{red}$a_1$};
\draw (-2.0, 0.0) node[left]{\small\color{red}$a_1$};
\draw (-0.0,-0.0) node[right]{\small\color{red}$a_2$};
\draw (-0.0, 2.0) node[above]{\small\color{red}$a_3$};
\draw (-0.0,-2.0) node[below]{\small\color{red}$a_3$};
\end{tikzpicture}
\begin{tikzpicture}[scale=0.55]
\filldraw [black!20] (-2, 2) circle (0.5cm);
\filldraw [black!20] ( 2, 2) circle (0.5cm);
\filldraw [black!20] ( 2,-2) circle (0.5cm);
\filldraw [black!20] (-2,-2) circle (0.5cm);
\draw [line width=1pt] (-2, 2) circle (0.5cm);
\draw [line width=1pt] ( 2, 2) circle (0.5cm);
\draw [line width=1pt] ( 2,-2) circle (0.5cm);
\draw [line width=1pt] (-2,-2) circle (0.5cm);
\draw [dash pattern=on 2pt off 2pt] (-2, 2)-- ( 2, 2);
\draw [dash pattern=on 2pt off 2pt] ( 2, 2)-- ( 2,-2);
\draw [dash pattern=on 2pt off 2pt] ( 2,-2)-- (-2,-2);
\draw [dash pattern=on 2pt off 2pt] (-2, 2)-- (-2,-2);
\draw [line width=1pt,color=ffqqqq] (-2.35,1.65)-- (-1.65,-1.65);
\draw [line width=1pt,color=ffqqqq] (-1.65,2.35)-- (1.65,1.65);
\draw [line width=1pt,color=ffqqqq] (1.65,1.65)-- (2.35,-1.65);
\draw [line width=1pt,color=ffqqqq] (1.65,-2.35)-- (-1.65,-1.65);
\draw [line width=1pt,color=ffqqqq] (-1.65,-1.65)-- (1.65,1.65);
\fill [color=qqwuqq] (-1.65, 1.65) circle (2.5pt);
\fill [color=qqwuqq] ( 1.65,-1.65) circle (2.5pt);
\filldraw [color=white, draw=ffqqqq, line width=0.5pt] (-1.65, 2.35) circle (2.5pt);
\fill [color=qqwuqq] (-2.35, 2.35) circle (2.5pt);
\filldraw [color=white, draw=ffqqqq, line width=0.5pt] (-1.65,-1.65) circle (2.5pt);
\filldraw [color=white, draw=ffqqqq, line width=0.5pt] (-2.35,-2.35) circle (2.5pt);
\fill [color=qqwuqq] (-2.35,-1.65) circle (2.5pt);
\fill [color=qqwuqq] (-1.65,-2.35) circle (2.5pt);
\filldraw [color=white, draw=ffqqqq, line width=0.5pt] (-2.35, 1.65) circle (2.5pt);
\fill [color=qqwuqq] ( 1.65, 2.35) circle (2.5pt);
\filldraw [color=white, draw=ffqqqq, line width=0.5pt] ( 1.65, 1.65) circle (2.5pt);
\filldraw [color=white, draw=ffqqqq, line width=0.5pt] ( 2.35, 2.35) circle (2.5pt);
\fill [color=qqwuqq] ( 2.35, 1.65) circle (2.5pt);
\filldraw [color=white, draw=ffqqqq, line width=0.5pt] ( 2.35,-1.65) circle (2.5pt);
\fill [color=qqwuqq] ( 2.35,-2.35) circle (2.5pt);
\filldraw [color=white, draw=ffqqqq, line width=0.5pt] ( 1.65,-2.35) circle (2.5pt);
\draw [color=qqwuqq] ( 1.8,-1.8) node{$p$};
\draw [color=qqwuqq] (-1.8, 1.8) node{$q$};
\draw [color=ffqqqq] (-1.8,-1.8) node{$r$};
\draw [color=ffqqqq] (-1.65, 2.35) node[above]{$r$};
\draw [color=ffqqqq] ( 2.35,-1.65) node[right]{$r$};
\draw [color=ffqqqq] ( 1.8, 1.8) node{$s$};
\draw [color=ffqqqq] ( 1.65,-2.35) node[below]{$s$};
\draw [color=ffqqqq] (-2.35, 1.65) node[left]{$s$};
\draw [blue][line width=1pt][postaction={on each segment={mid arrow=blue}}]
      (0.5,-2.1) to[out= 90,in=180] (1.65,-1.65);
\draw [blue](0.5, -1.55) node{\tiny$\circled{3}$};
\draw [blue][line width=1pt][postaction={on each segment={mid arrow=blue}}]
      (0.5, 0.5) to[out=135, in=-90] (0,2);
\draw [blue] (-0.2,0.8) node{\tiny$\circled{2}$};
\draw [blue][line width=1pt][postaction={on each segment={mid arrow=blue}}]
      (1.65,-1.65) -- (0,0);
\draw [blue] (1, -1) node[above]{\tiny$\circled{1}$};
\draw ( 2.0,-0.0) node[right]{\small\color{red}$a_1$};
\draw (-2.0, 0.0) node[left]{\small\color{red}$a_1$};
\draw (-0.0,-0.0) node[right]{\small\color{red}$a_2$};
\draw (-0.0, 2.0) node[above]{\small\color{red}$a_3$};
\draw (-0.0,-2.0) node[below]{\small\color{red}$a_3$};
\end{tikzpicture}
\begin{tikzpicture}[scale=0.55]
\filldraw [black!20] (-2, 2) circle (0.5cm);
\filldraw [black!20] ( 2, 2) circle (0.5cm);
\filldraw [black!20] ( 2,-2) circle (0.5cm);
\filldraw [black!20] (-2,-2) circle (0.5cm);
\draw [line width=1pt] (-2, 2) circle (0.5cm);
\draw [line width=1pt] ( 2, 2) circle (0.5cm);
\draw [line width=1pt] ( 2,-2) circle (0.5cm);
\draw [line width=1pt] (-2,-2) circle (0.5cm);
\draw [dash pattern=on 2pt off 2pt] (-2, 2)-- ( 2, 2);
\draw [dash pattern=on 2pt off 2pt] ( 2, 2)-- ( 2,-2);
\draw [dash pattern=on 2pt off 2pt] ( 2,-2)-- (-2,-2);
\draw [dash pattern=on 2pt off 2pt] (-2, 2)-- (-2,-2);
\draw [line width=1pt,color=ffqqqq] (-2.35,1.65)-- (-1.65,-1.65);
\draw [line width=1pt,color=ffqqqq] (-1.65,2.35)-- (1.65,1.65);
\draw [line width=1pt,color=ffqqqq] (1.65,1.65)-- (2.35,-1.65);
\draw [line width=1pt,color=ffqqqq] (1.65,-2.35)-- (-1.65,-1.65);
\draw [line width=1pt,color=ffqqqq] (-1.65,-1.65)-- (1.65,1.65);
\fill [color=qqwuqq] (-1.65, 1.65) circle (2.5pt);
\fill [color=qqwuqq] ( 1.65,-1.65) circle (2.5pt);
\filldraw [color=white, draw=ffqqqq, line width=0.5pt] (-1.65, 2.35) circle (2.5pt);
\fill [color=qqwuqq] (-2.35, 2.35) circle (2.5pt);
\filldraw [color=white, draw=ffqqqq, line width=0.5pt] (-1.65,-1.65) circle (2.5pt);
\filldraw [color=white, draw=ffqqqq, line width=0.5pt] (-2.35,-2.35) circle (2.5pt);
\fill [color=qqwuqq] (-2.35,-1.65) circle (2.5pt);
\fill [color=qqwuqq] (-1.65,-2.35) circle (2.5pt);
\filldraw [color=white, draw=ffqqqq, line width=0.5pt] (-2.35, 1.65) circle (2.5pt);
\fill [color=qqwuqq] ( 1.65, 2.35) circle (2.5pt);
\filldraw [color=white, draw=ffqqqq, line width=0.5pt] ( 1.65, 1.65) circle (2.5pt);
\filldraw [color=white, draw=ffqqqq, line width=0.5pt] ( 2.35, 2.35) circle (2.5pt);
\fill [color=qqwuqq] ( 2.35, 1.65) circle (2.5pt);
\filldraw [color=white, draw=ffqqqq, line width=0.5pt] ( 2.35,-1.65) circle (2.5pt);
\fill [color=qqwuqq] ( 2.35,-2.35) circle (2.5pt);
\filldraw [color=white, draw=ffqqqq, line width=0.5pt] ( 1.65,-2.35) circle (2.5pt);
\draw [color=qqwuqq] ( 1.8,-1.8) node{$p$};
\draw [color=qqwuqq] (-1.8, 1.8) node{$q$};
\draw [color=ffqqqq] (-1.8,-1.8) node{$r$};
\draw [color=ffqqqq] (-1.65, 2.35) node[above]{$r$};
\draw [color=ffqqqq] ( 2.35,-1.65) node[right]{$r$};
\draw [color=ffqqqq] ( 1.8, 1.8) node{$s$};
\draw [color=ffqqqq] ( 1.65,-2.35) node[below]{$s$};
\draw [color=ffqqqq] (-2.35, 1.65) node[left]{$s$};
\draw [blue][line width=1pt][postaction={on each segment={mid arrow=blue}}]
      (0.,-2) to[out= 90,in=180] (2,-0.);
\draw [blue](-0.3, -1.55) node{\tiny$\circled{3}$};
\draw [blue][line width=1pt][postaction={on each segment={mid arrow=blue}}]
      (0.5, 0.5) to[out=135, in=-90] (0,2);
\draw [blue] (-0.2,0.8) node{\tiny$\circled{2}$};
\draw [blue][line width=1pt][postaction={on each segment={mid arrow=blue}}]
      (1.65,-1.65) -- (0,0);
\draw [blue] (1.5, -1.5) node[above]{\tiny$\circled{1}$};
\draw ( 2.0,-0.0) node[right]{\small\color{red}$a_1$};
\draw (-2.0, 0.0) node[left]{\small\color{red}$a_1$};
\draw (-0.0,-0.0) node[right]{\small\color{red}$a_2$};
\draw (-0.0, 2.0) node[above]{\small\color{red}$a_3$};
\draw (-0.0,-2.0) node[below]{\small\color{red}$a_3$};
\end{tikzpicture}
\caption{Case 2}
\label{fig:case 2}
\end{center}
\end{figure}

    \item[Case 3] $\snum{i_2}{\ws}=3$.
    \begin{itemize}
        \item[Case 3.1] $\snum{i_3}{\ws}=1$. If $t=3$, see the first picture of \Pic\ref{fig:case 3.1}, then (S.L.E.), a contradiction. So $t>3$. 
            By Lemma~\ref{lem:cancel}, $\snum{i_4}{\ws}=2$ and $\snum{i_5}{\ws}=3$. If $t=5$, see the second picture of \Pic\ref{fig:case 3.1}, then $\sdii{i_t}{\ws}-\sdii{i_1}{\ws}=1>0$ as required. Now we assume $t>5$.
            By Lemma~\ref{lem:cancel}, $\snum{i_6}{\ws}=1$ and $\snum{i_7}{\ws}=2$. If $t=7$, see the third picture of Figure~\ref{fig:case 3.1},
\begin{figure}[htbp]
\begin{center}
\definecolor{ffqqqq}{rgb}{1,0,0}
\definecolor{qqwuqq}{rgb}{0,0,1}
\begin{tikzpicture}[scale=0.55]
\filldraw [black!20] (-2, 2) circle (0.5cm);
\filldraw [black!20] ( 2, 2) circle (0.5cm);
\filldraw [black!20] ( 2,-2) circle (0.5cm);
\filldraw [black!20] (-2,-2) circle (0.5cm);
\draw [line width=1pt] (-2, 2) circle (0.5cm);
\draw [line width=1pt] ( 2, 2) circle (0.5cm);
\draw [line width=1pt] ( 2,-2) circle (0.5cm);
\draw [line width=1pt] (-2,-2) circle (0.5cm);
\draw [dash pattern=on 2pt off 2pt] (-2, 2)-- ( 2, 2);
\draw [dash pattern=on 2pt off 2pt] ( 2, 2)-- ( 2,-2);
\draw [dash pattern=on 2pt off 2pt] ( 2,-2)-- (-2,-2);
\draw [dash pattern=on 2pt off 2pt] (-2, 2)-- (-2,-2);
\draw [line width=1pt,color=ffqqqq] (-2.35,1.65)-- (-1.65,-1.65);
\draw [line width=1pt,color=ffqqqq] (-1.65,2.35)-- (1.65,1.65);
\draw [line width=1pt,color=ffqqqq] (1.65,1.65)-- (2.35,-1.65);
\draw [line width=1pt,color=ffqqqq] (1.65,-2.35)-- (-1.65,-1.65);
\draw [line width=1pt,color=ffqqqq] (-1.65,-1.65)-- (1.65,1.65);
\fill [color=qqwuqq] (-1.65, 1.65) circle (2.5pt);
\fill [color=qqwuqq] ( 1.65,-1.65) circle (2.5pt);
\filldraw [color=white, draw=ffqqqq, line width=0.5pt] (-1.65, 2.35) circle (2.5pt);
\fill [color=qqwuqq] (-2.35, 2.35) circle (2.5pt);
\filldraw [color=white, draw=ffqqqq, line width=0.5pt] (-1.65,-1.65) circle (2.5pt);
\filldraw [color=white, draw=ffqqqq, line width=0.5pt] (-2.35,-2.35) circle (2.5pt);
\fill [color=qqwuqq] (-2.35,-1.65) circle (2.5pt);
\fill [color=qqwuqq] (-1.65,-2.35) circle (2.5pt);
\filldraw [color=white, draw=ffqqqq, line width=0.5pt] (-2.35, 1.65) circle (2.5pt);
\fill [color=qqwuqq] ( 1.65, 2.35) circle (2.5pt);
\filldraw [color=white, draw=ffqqqq, line width=0.5pt] ( 1.65, 1.65) circle (2.5pt);
\filldraw [color=white, draw=ffqqqq, line width=0.5pt] ( 2.35, 2.35) circle (2.5pt);
\fill [color=qqwuqq] ( 2.35, 1.65) circle (2.5pt);
\filldraw [color=white, draw=ffqqqq, line width=0.5pt] ( 2.35,-1.65) circle (2.5pt);
\fill [color=qqwuqq] ( 2.35,-2.35) circle (2.5pt);
\filldraw [color=white, draw=ffqqqq, line width=0.5pt] ( 1.65,-2.35) circle (2.5pt);
\draw [color=qqwuqq] ( 1.8,-1.8) node{$p$};
\draw [color=qqwuqq] (-1.8, 1.8) node{$q$};
\draw [color=ffqqqq] (-1.8,-1.8) node{$r$};
\draw [color=ffqqqq] (-1.65, 2.35) node[above]{$r$};
\draw [color=ffqqqq] ( 2.35,-1.65) node[right]{$r$};
\draw [color=ffqqqq] ( 1.8, 1.8) node{$s$};
\draw [color=ffqqqq] ( 1.65,-2.35) node[below]{$s$};
\draw [color=ffqqqq] (-2.35, 1.65) node[left]{$s$};
\draw [blue][line width=1pt][postaction={on each segment={mid arrow=blue}}]
      (1.65, -1.65) to[out=180, in=90] (0.5,-2.1);
\draw [blue] (0.5,-1.5) node{\tiny$\circled{1}$};
\draw [blue][line width=1pt][postaction={on each segment={mid arrow=blue}}]
      (-0.5, 2.1) to[out=-90,in= 0] (-2.1,0.5);
\draw [blue] (-0.7, 0.7) node{\tiny$\circled{2}$};
\draw [blue][line width=1pt][postaction={on each segment={mid arrow=blue}}]
      (2.1, -0.5) to[out=180,in= 90] (1.65, -1.65);
\draw [blue] (1.5, -0.5) node{\tiny$\circled{3}$};
\draw ( 2.0,-0.0) node[right]{\small\color{red}$a_1$};
\draw (-2.0, 0.0) node[left]{\small\color{red}$a_1$};
\draw (-0.0,-0.0) node[right]{\small\color{red}$a_2$};
\draw (-0.0, 2.0) node[above]{\small\color{red}$a_3$};
\draw (-0.0,-2.0) node[below]{\small\color{red}$a_3$};
\end{tikzpicture}
\begin{tikzpicture}[scale=0.55]
\filldraw [black!20] (-2, 2) circle (0.5cm);
\filldraw [black!20] ( 2, 2) circle (0.5cm);
\filldraw [black!20] ( 2,-2) circle (0.5cm);
\filldraw [black!20] (-2,-2) circle (0.5cm);
\draw [line width=1pt] (-2, 2) circle (0.5cm);
\draw [line width=1pt] ( 2, 2) circle (0.5cm);
\draw [line width=1pt] ( 2,-2) circle (0.5cm);
\draw [line width=1pt] (-2,-2) circle (0.5cm);
\draw [dash pattern=on 2pt off 2pt] (-2, 2)-- ( 2, 2);
\draw [dash pattern=on 2pt off 2pt] ( 2, 2)-- ( 2,-2);
\draw [dash pattern=on 2pt off 2pt] ( 2,-2)-- (-2,-2);
\draw [dash pattern=on 2pt off 2pt] (-2, 2)-- (-2,-2);
\draw [line width=1pt,color=ffqqqq] (-2.35,1.65)-- (-1.65,-1.65);
\draw [line width=1pt,color=ffqqqq] (-1.65,2.35)-- (1.65,1.65);
\draw [line width=1pt,color=ffqqqq] (1.65,1.65)-- (2.35,-1.65);
\draw [line width=1pt,color=ffqqqq] (1.65,-2.35)-- (-1.65,-1.65);
\draw [line width=1pt,color=ffqqqq] (-1.65,-1.65)-- (1.65,1.65);
\fill [color=qqwuqq] (-1.65, 1.65) circle (2.5pt);
\fill [color=qqwuqq] ( 1.65,-1.65) circle (2.5pt);
\filldraw [color=white, draw=ffqqqq, line width=0.5pt] (-1.65, 2.35) circle (2.5pt);
\fill [color=qqwuqq] (-2.35, 2.35) circle (2.5pt);
\filldraw [color=white, draw=ffqqqq, line width=0.5pt] (-1.65,-1.65) circle (2.5pt);
\filldraw [color=white, draw=ffqqqq, line width=0.5pt] (-2.35,-2.35) circle (2.5pt);
\fill [color=qqwuqq] (-2.35,-1.65) circle (2.5pt);
\fill [color=qqwuqq] (-1.65,-2.35) circle (2.5pt);
\filldraw [color=white, draw=ffqqqq, line width=0.5pt] (-2.35, 1.65) circle (2.5pt);
\fill [color=qqwuqq] ( 1.65, 2.35) circle (2.5pt);
\filldraw [color=white, draw=ffqqqq, line width=0.5pt] ( 1.65, 1.65) circle (2.5pt);
\filldraw [color=white, draw=ffqqqq, line width=0.5pt] ( 2.35, 2.35) circle (2.5pt);
\fill [color=qqwuqq] ( 2.35, 1.65) circle (2.5pt);
\filldraw [color=white, draw=ffqqqq, line width=0.5pt] ( 2.35,-1.65) circle (2.5pt);
\fill [color=qqwuqq] ( 2.35,-2.35) circle (2.5pt);
\filldraw [color=white, draw=ffqqqq, line width=0.5pt] ( 1.65,-2.35) circle (2.5pt);
\draw [color=qqwuqq] ( 1.8,-1.8) node{$p$};
\draw [color=qqwuqq] (-1.8, 1.8) node{$q$};
\draw [color=ffqqqq] (-1.8,-1.8) node{$r$};
\draw [color=ffqqqq] (-1.65, 2.35) node[above]{$r$};
\draw [color=ffqqqq] ( 2.35,-1.65) node[right]{$r$};
\draw [color=ffqqqq] ( 1.8, 1.8) node{$s$};
\draw [color=ffqqqq] ( 1.65,-2.35) node[below]{$s$};
\draw [color=ffqqqq] (-2.35, 1.65) node[left]{$s$};
\draw [blue][line width=1pt][postaction={on each segment={mid arrow=blue}}]
      (1.65, -1.65) to[out=135, in=90] (-0.2,-1.9);
\draw [blue] (-0.3,-1.4) node{\tiny$\circled{1}$};
\draw [blue][line width=1pt][postaction={on each segment={mid arrow=blue}}]
      (-0.5, 2.1) to[out=-90,in= 0] (-2.1,0.5);
\draw [blue] (-0.7, 0.7) node{\tiny$\circled{2}$};
\draw [blue][line width=1pt][postaction={on each segment={mid arrow=blue}}]
      (2, 0) to[out=180,in=-45] (0.6, 0.6);
\draw [blue] (1.4, 0.4) node{\tiny$\circled{3}$};
\draw [blue][line width=1pt][postaction={on each segment={mid arrow=blue}}]
      (0.7, 0.7) to[out=135,in=-90] (0.4, 1.95);
\draw [blue] (0., 1.3) node{\tiny$\circled{4}$};
\draw [blue][line width=1pt][postaction={on each segment={mid arrow=blue}}]
      (0.8,-2.15) to[out=90,in=225] (1.65, -1.65);
\draw [blue] (0.55, -1.8) node{\tiny$\circled{5}$};
\draw ( 2.0,-0.0) node[right]{\small\color{red}$a_1$};
\draw (-2.0, 0.0) node[left]{\small\color{red}$a_1$};
\draw (-0.0,-0.0) node[right]{\small\color{red}$a_2$};
\draw (-0.0, 2.0) node[above]{\small\color{red}$a_3$};
\draw (-0.0,-2.0) node[below]{\small\color{red}$a_3$};
\end{tikzpicture}
\begin{tikzpicture}[scale=0.55]
\filldraw [black!20] (-2, 2) circle (0.5cm);
\filldraw [black!20] ( 2, 2) circle (0.5cm);
\filldraw [black!20] ( 2,-2) circle (0.5cm);
\filldraw [black!20] (-2,-2) circle (0.5cm);
\draw [line width=1pt] (-2, 2) circle (0.5cm);
\draw [line width=1pt] ( 2, 2) circle (0.5cm);
\draw [line width=1pt] ( 2,-2) circle (0.5cm);
\draw [line width=1pt] (-2,-2) circle (0.5cm);
\draw [dash pattern=on 2pt off 2pt] (-2, 2)-- ( 2, 2);
\draw [dash pattern=on 2pt off 2pt] ( 2, 2)-- ( 2,-2);
\draw [dash pattern=on 2pt off 2pt] ( 2,-2)-- (-2,-2);
\draw [dash pattern=on 2pt off 2pt] (-2, 2)-- (-2,-2);
\draw [line width=1pt,color=ffqqqq] (-2.35,1.65)-- (-1.65,-1.65);
\draw [line width=1pt,color=ffqqqq] (-1.65,2.35)-- (1.65,1.65);
\draw [line width=1pt,color=ffqqqq] (1.65,1.65)-- (2.35,-1.65);
\draw [line width=1pt,color=ffqqqq] (1.65,-2.35)-- (-1.65,-1.65);
\draw [line width=1pt,color=ffqqqq] (-1.65,-1.65)-- (1.65,1.65);
\fill [color=qqwuqq] (-1.65, 1.65) circle (2.5pt);
\fill [color=qqwuqq] ( 1.65,-1.65) circle (2.5pt);
\filldraw [color=white, draw=ffqqqq, line width=0.5pt] (-1.65, 2.35) circle (2.5pt);
\fill [color=qqwuqq] (-2.35, 2.35) circle (2.5pt);
\filldraw [color=white, draw=ffqqqq, line width=0.5pt] (-1.65,-1.65) circle (2.5pt);
\filldraw [color=white, draw=ffqqqq, line width=0.5pt] (-2.35,-2.35) circle (2.5pt);
\fill [color=qqwuqq] (-2.35,-1.65) circle (2.5pt);
\fill [color=qqwuqq] (-1.65,-2.35) circle (2.5pt);
\filldraw [color=white, draw=ffqqqq, line width=0.5pt] (-2.35, 1.65) circle (2.5pt);
\fill [color=qqwuqq] ( 1.65, 2.35) circle (2.5pt);
\filldraw [color=white, draw=ffqqqq, line width=0.5pt] ( 1.65, 1.65) circle (2.5pt);
\filldraw [color=white, draw=ffqqqq, line width=0.5pt] ( 2.35, 2.35) circle (2.5pt);
\fill [color=qqwuqq] ( 2.35, 1.65) circle (2.5pt);
\filldraw [color=white, draw=ffqqqq, line width=0.5pt] ( 2.35,-1.65) circle (2.5pt);
\fill [color=qqwuqq] ( 2.35,-2.35) circle (2.5pt);
\filldraw [color=white, draw=ffqqqq, line width=0.5pt] ( 1.65,-2.35) circle (2.5pt);
\draw [color=qqwuqq] ( 1.8,-1.8) node{$p$};
\draw [color=qqwuqq] (-1.8, 1.8) node{$q$};
\draw [color=ffqqqq] (-1.8,-1.8) node{$r$};
\draw [color=ffqqqq] (-1.65, 2.35) node[above]{$r$};
\draw [color=ffqqqq] ( 2.35,-1.65) node[right]{$r$};
\draw [color=ffqqqq] ( 1.8, 1.8) node{$s$};
\draw [color=ffqqqq] ( 1.65,-2.35) node[below]{$s$};
\draw [color=ffqqqq] (-2.35, 1.65) node[left]{$s$};
\draw [blue][line width=1pt][postaction={on each segment={mid arrow=blue}}]
      (1.65, -1.65) to[out=180, in=90] (0.5,-2.1);
\draw [blue] (0.5,-1.65) node{\tiny$\circled{1}$};
\draw [blue][line width=1pt][postaction={on each segment={mid arrow=blue}}]
      (-0.5, 2.1) to[out=-90,in= 0] (-2.1,0.5);
\draw [blue] (-0.7, 0.7) node{\tiny$\circled{2}$};
\draw [blue][line width=1pt][postaction={on each segment={mid arrow=blue}}]
      (2, 0) to[out=180,in=-45] (0.6, 0.6);
\draw [blue] (1.4, 0.4) node{\tiny$\circled{3}$};
\draw [blue][line width=1pt][postaction={on each segment={mid arrow=blue}}]
      (0.7, 0.7) to[out=135,in=-90] (0.4, 1.95);
\draw [blue] (0.0, 1.3) node{\tiny$\circled{4}$};
\draw [blue][line width=1pt][postaction={on each segment={mid arrow=blue}}]
      (-0.3,-1.95) to[out=90,in=180] (2.2, -1);
\draw [blue] (-0.55, -1.5) node{\tiny$\circled{5}$};
\draw [blue][line width=1pt][postaction={on each segment={mid arrow=blue}}]
      (-2, 0) to[out=0,in=135] (-0.5, -0.5);
\draw [blue] (-0.75, 0.) node{\tiny$\circled{6}$};
\draw [blue][line width=1pt][postaction={on each segment={mid arrow=blue}}]
      (-0.4, -0.4) to[out=0,in=90] (1.65, -1.65);
\draw [blue] (1.10,-0.4) node{\tiny$\circled{7}$};
\draw ( 2.0,-0.0) node[right]{\small\color{red}$a_1$};
\draw (-2.0, 0.0) node[left]{\small\color{red}$a_1$};
\draw (-0.0,-0.0) node[right]{\small\color{red}$a_2$};
\draw (-0.0, 2.0) node[above]{\small\color{red}$a_3$};
\draw (-0.0,-2.0) node[below]{\small\color{red}$a_3$};
\end{tikzpicture}
\begin{tikzpicture}[scale=0.55]
\filldraw [black!20] (-2, 2) circle (0.5cm);
\filldraw [black!20] ( 2, 2) circle (0.5cm);
\filldraw [black!20] ( 2,-2) circle (0.5cm);
\filldraw [black!20] (-2,-2) circle (0.5cm);
\draw [line width=1pt] (-2, 2) circle (0.5cm);
\draw [line width=1pt] ( 2, 2) circle (0.5cm);
\draw [line width=1pt] ( 2,-2) circle (0.5cm);
\draw [line width=1pt] (-2,-2) circle (0.5cm);
\draw [dash pattern=on 2pt off 2pt] (-2, 2)-- ( 2, 2);
\draw [dash pattern=on 2pt off 2pt] ( 2, 2)-- ( 2,-2);
\draw [dash pattern=on 2pt off 2pt] ( 2,-2)-- (-2,-2);
\draw [dash pattern=on 2pt off 2pt] (-2, 2)-- (-2,-2);
\draw [line width=1pt,color=ffqqqq] (-2.35,1.65)-- (-1.65,-1.65);
\draw [line width=1pt,color=ffqqqq] (-1.65,2.35)-- (1.65,1.65);
\draw [line width=1pt,color=ffqqqq] (1.65,1.65)-- (2.35,-1.65);
\draw [line width=1pt,color=ffqqqq] (1.65,-2.35)-- (-1.65,-1.65);
\draw [line width=1pt,color=ffqqqq] (-1.65,-1.65)-- (1.65,1.65);
\fill [color=qqwuqq] (-1.65, 1.65) circle (2.5pt);
\fill [color=qqwuqq] ( 1.65,-1.65) circle (2.5pt);
\filldraw [color=white, draw=ffqqqq, line width=0.5pt] (-1.65, 2.35) circle (2.5pt);
\fill [color=qqwuqq] (-2.35, 2.35) circle (2.5pt);
\filldraw [color=white, draw=ffqqqq, line width=0.5pt] (-1.65,-1.65) circle (2.5pt);
\filldraw [color=white, draw=ffqqqq, line width=0.5pt] (-2.35,-2.35) circle (2.5pt);
\fill [color=qqwuqq] (-2.35,-1.65) circle (2.5pt);
\fill [color=qqwuqq] (-1.65,-2.35) circle (2.5pt);
\filldraw [color=white, draw=ffqqqq, line width=0.5pt] (-2.35, 1.65) circle (2.5pt);
\fill [color=qqwuqq] ( 1.65, 2.35) circle (2.5pt);
\filldraw [color=white, draw=ffqqqq, line width=0.5pt] ( 1.65, 1.65) circle (2.5pt);
\filldraw [color=white, draw=ffqqqq, line width=0.5pt] ( 2.35, 2.35) circle (2.5pt);
\fill [color=qqwuqq] ( 2.35, 1.65) circle (2.5pt);
\filldraw [color=white, draw=ffqqqq, line width=0.5pt] ( 2.35,-1.65) circle (2.5pt);
\fill [color=qqwuqq] ( 2.35,-2.35) circle (2.5pt);
\filldraw [color=white, draw=ffqqqq, line width=0.5pt] ( 1.65,-2.35) circle (2.5pt);
\draw [color=qqwuqq] ( 1.8,-1.8) node{$p$};
\draw [color=qqwuqq] (-1.8, 1.8) node{$q$};
\draw [color=ffqqqq] (-1.8,-1.8) node{$r$};
\draw [color=ffqqqq] (-1.65, 2.35) node[above]{$r$};
\draw [color=ffqqqq] ( 2.35,-1.65) node[right]{$r$};
\draw [color=ffqqqq] ( 1.8, 1.8) node{$s$};
\draw [color=ffqqqq] ( 1.65,-2.35) node[below]{$s$};
\draw [color=ffqqqq] (-2.35, 1.65) node[left]{$s$};
\draw [blue][line width=1pt][postaction={on each segment={mid arrow=blue}}]
      (1.65, -1.65) to[out=180, in=90] (0.5,-2.1);
\draw [blue] (0.5,-1.65) node{\tiny$\circled{1}$};
\draw [blue][line width=1pt][postaction={on each segment={mid arrow=blue}}]
      (-0.5, 2.1) to[out=-90,in= 0] (-2.1,0.5);
\draw [blue] (-0.7, 0.7) node{\tiny$\circled{2}$};
\draw [blue][line width=1pt][postaction={on each segment={mid arrow=blue}}]
      (2, 0) to[out=180,in=-45] (0.6, 0.6);
\draw [blue] (1.4, 0.4) node{\tiny$\circled{3}$};
\draw [blue][line width=1pt][postaction={on each segment={mid arrow=blue}}]
      (0.7, 0.7) to[out=135,in=-90] (0.4, 1.95);
\draw [blue] (0.0, 1.3) node{\tiny$\circled{4}$};
\draw [blue][line width=1pt][postaction={on each segment={mid arrow=blue}}]
      (-0.3,-1.95) to[out=90,in=180] (2.2, -1);
\draw [blue] (0.2, -0.85) node{\tiny$\circled{5}$};
\draw [blue][line width=1pt][postaction={on each segment={mid arrow=blue}}]
      (-1.9,-0.5) to[out=0,in=135] (-0.9, -0.9);
\draw [blue] (-1.4,-0.2) node{\tiny$\circled{6}$};
\draw [blue][line width=1pt][postaction={on each segment={mid arrow=blue}}]
      (-0.8, -0.8) to[out=-45,in=45] (-0.9, -1.8);
\draw [blue] (-1.1, -1.6) node{\tiny$\circled{7}$};
\draw ( 2.0,-0.0) node[right]{\small\color{red}$a_1$};
\draw (-2.0, 0.0) node[left]{\small\color{red}$a_1$};
\draw (-0.0,-0.0) node[right]{\small\color{red}$a_2$};
\draw (-0.0, 2.0) node[above]{\small\color{red}$a_3$};
\draw (-0.0,-2.0) node[below]{\small\color{red}$a_3$};
\end{tikzpicture}
\end{center}
\caption{Case 3.1}
\label{fig:case 3.1}
\end{figure}
        then (S.L.E.), a contradiction. So $t>7$. By Lemma~\ref{lem:cancel}, $\snum{i_8}{\ws}=3$, see the fourth picture of Figure~\ref{fig:case 3.1}. However, this contradicts Lemma~\ref{lem:four}.

        \item[Case 3.2] $\snum{i_3}{\ws}=2$. If $t=3$, see the first picture of \Pic\ref{fig:case 3.2}, then (S.L.E.), a contradiction. So $t>3$. By Lemma~\ref{lem:cancel}, $\snum{i_4}{\ws}=1$ and $\snum{i_5}{\ws}=3$. If $t=5$, since (S.R.E.), we are in the situation shown in the second picture of \Pic\ref{fig:case 3.2}, which contradicts Lemma~\ref{lemm:self-int Type}. So $t>5$. By Lemma~\ref{lem:cancel}, $\snum{i_6}{\ws}=2$ and $\snum{i_7}{\ws}=1$. If $t=7$, see the third picture of \Pic\ref{fig:case 3.2},
\begin{figure}[htbp]
\begin{center}
\definecolor{ffqqqq}{rgb}{1,0,0}
\definecolor{qqwuqq}{rgb}{0,0,1}
\begin{tikzpicture}[scale=0.55]
\filldraw [black!20] (-2, 2) circle (0.5cm);
\filldraw [black!20] ( 2, 2) circle (0.5cm);
\filldraw [black!20] ( 2,-2) circle (0.5cm);
\filldraw [black!20] (-2,-2) circle (0.5cm);
\draw [line width=1pt] (-2, 2) circle (0.5cm);
\draw [line width=1pt] ( 2, 2) circle (0.5cm);
\draw [line width=1pt] ( 2,-2) circle (0.5cm);
\draw [line width=1pt] (-2,-2) circle (0.5cm);
\draw [dash pattern=on 2pt off 2pt] (-2, 2)-- ( 2, 2);
\draw [dash pattern=on 2pt off 2pt] ( 2, 2)-- ( 2,-2);
\draw [dash pattern=on 2pt off 2pt] ( 2,-2)-- (-2,-2);
\draw [dash pattern=on 2pt off 2pt] (-2, 2)-- (-2,-2);
\draw [line width=1pt,color=ffqqqq] (-2.35,1.65)-- (-1.65,-1.65);
\draw [line width=1pt,color=ffqqqq] (-1.65,2.35)-- (1.65,1.65);
\draw [line width=1pt,color=ffqqqq] (1.65,1.65)-- (2.35,-1.65);
\draw [line width=1pt,color=ffqqqq] (1.65,-2.35)-- (-1.65,-1.65);
\draw [line width=1pt,color=ffqqqq] (-1.65,-1.65)-- (1.65,1.65);
\fill [color=qqwuqq] (-1.65, 1.65) circle (2.5pt);
\fill [color=qqwuqq] ( 1.65,-1.65) circle (2.5pt);
\filldraw [color=white, draw=ffqqqq, line width=0.5pt] (-1.65, 2.35) circle (2.5pt);
\fill [color=qqwuqq] (-2.35, 2.35) circle (2.5pt);
\filldraw [color=white, draw=ffqqqq, line width=0.5pt] (-1.65,-1.65) circle (2.5pt);
\filldraw [color=white, draw=ffqqqq, line width=0.5pt] (-2.35,-2.35) circle (2.5pt);
\fill [color=qqwuqq] (-2.35,-1.65) circle (2.5pt);
\fill [color=qqwuqq] (-1.65,-2.35) circle (2.5pt);
\filldraw [color=white, draw=ffqqqq, line width=0.5pt] (-2.35, 1.65) circle (2.5pt);
\fill [color=qqwuqq] ( 1.65, 2.35) circle (2.5pt);
\filldraw [color=white, draw=ffqqqq, line width=0.5pt] ( 1.65, 1.65) circle (2.5pt);
\filldraw [color=white, draw=ffqqqq, line width=0.5pt] ( 2.35, 2.35) circle (2.5pt);
\fill [color=qqwuqq] ( 2.35, 1.65) circle (2.5pt);
\filldraw [color=white, draw=ffqqqq, line width=0.5pt] ( 2.35,-1.65) circle (2.5pt);
\fill [color=qqwuqq] ( 2.35,-2.35) circle (2.5pt);
\filldraw [color=white, draw=ffqqqq, line width=0.5pt] ( 1.65,-2.35) circle (2.5pt);
\draw [color=qqwuqq] ( 1.8,-1.8) node{$p$};
\draw [color=qqwuqq] (-1.8, 1.8) node{$q$};
\draw [color=ffqqqq] (-1.8,-1.8) node{$r$};
\draw [color=ffqqqq] (-1.65, 2.35) node[above]{$r$};
\draw [color=ffqqqq] ( 2.35,-1.65) node[right]{$r$};
\draw [color=ffqqqq] ( 1.8, 1.8) node{$s$};
\draw [color=ffqqqq] ( 1.65,-2.35) node[below]{$s$};
\draw [color=ffqqqq] (-2.35, 1.65) node[left]{$s$};
\draw [blue][line width=1pt][postaction={on each segment={mid arrow=blue}}]
      (1.65, -1.65) to[out=180, in=90] (0.5,-2.1);
\draw [blue] (0.5,-1.5) node{\tiny$\circled{1}$};
\draw [blue][line width=1pt][postaction={on each segment={mid arrow=blue}}]
      ( 0, 2) to[out=-90,in=135] (0.35, 0.35);
\draw [blue] ( 0.4, 1.2) node{\tiny$\circled{2}$};
\draw [blue][line width=1pt][postaction={on each segment={mid arrow=blue}}]
      (0.0, 0.0) to[out=-45,in= 135] (1.65, -1.65);
\draw [blue] (1, -1) node[right]{\tiny$\circled{3}$};
\draw ( 2.0,-0.0) node[right]{\small\color{red}$a_1$};
\draw (-2.0, 0.0) node[left]{\small\color{red}$a_1$};
\draw (-0.0,-0.0) node[right]{\small\color{red}$a_2$};
\draw (-0.0, 2.0) node[above]{\small\color{red}$a_3$};
\draw (-0.0,-2.0) node[below]{\small\color{red}$a_3$};
\end{tikzpicture}
\begin{tikzpicture}[scale=0.55]
\filldraw [black!20] (-2, 2) circle (0.5cm);
\filldraw [black!20] ( 2, 2) circle (0.5cm);
\filldraw [black!20] ( 2,-2) circle (0.5cm);
\filldraw [black!20] (-2,-2) circle (0.5cm);
\draw [line width=1pt] (-2, 2) circle (0.5cm);
\draw [line width=1pt] ( 2, 2) circle (0.5cm);
\draw [line width=1pt] ( 2,-2) circle (0.5cm);
\draw [line width=1pt] (-2,-2) circle (0.5cm);
\draw [dash pattern=on 2pt off 2pt] (-2, 2)-- ( 2, 2);
\draw [dash pattern=on 2pt off 2pt] ( 2, 2)-- ( 2,-2);
\draw [dash pattern=on 2pt off 2pt] ( 2,-2)-- (-2,-2);
\draw [dash pattern=on 2pt off 2pt] (-2, 2)-- (-2,-2);
\draw [line width=1pt,color=ffqqqq] (-2.35,1.65)-- (-1.65,-1.65);
\draw [line width=1pt,color=ffqqqq] (-1.65,2.35)-- (1.65,1.65);
\draw [line width=1pt,color=ffqqqq] (1.65,1.65)-- (2.35,-1.65);
\draw [line width=1pt,color=ffqqqq] (1.65,-2.35)-- (-1.65,-1.65);
\draw [line width=1pt,color=ffqqqq] (-1.65,-1.65)-- (1.65,1.65);
\fill [color=qqwuqq] (-1.65, 1.65) circle (2.5pt);
\fill [color=qqwuqq] ( 1.65,-1.65) circle (2.5pt);
\filldraw [color=white, draw=ffqqqq, line width=0.5pt] (-1.65, 2.35) circle (2.5pt);
\fill [color=qqwuqq] (-2.35, 2.35) circle (2.5pt);
\filldraw [color=white, draw=ffqqqq, line width=0.5pt] (-1.65,-1.65) circle (2.5pt);
\filldraw [color=white, draw=ffqqqq, line width=0.5pt] (-2.35,-2.35) circle (2.5pt);
\fill [color=qqwuqq] (-2.35,-1.65) circle (2.5pt);
\fill [color=qqwuqq] (-1.65,-2.35) circle (2.5pt);
\filldraw [color=white, draw=ffqqqq, line width=0.5pt] (-2.35, 1.65) circle (2.5pt);
\fill [color=qqwuqq] ( 1.65, 2.35) circle (2.5pt);
\filldraw [color=white, draw=ffqqqq, line width=0.5pt] ( 1.65, 1.65) circle (2.5pt);
\filldraw [color=white, draw=ffqqqq, line width=0.5pt] ( 2.35, 2.35) circle (2.5pt);
\fill [color=qqwuqq] ( 2.35, 1.65) circle (2.5pt);
\filldraw [color=white, draw=ffqqqq, line width=0.5pt] ( 2.35,-1.65) circle (2.5pt);
\fill [color=qqwuqq] ( 2.35,-2.35) circle (2.5pt);
\filldraw [color=white, draw=ffqqqq, line width=0.5pt] ( 1.65,-2.35) circle (2.5pt);
\draw [color=qqwuqq] ( 1.8,-1.8) node{$p$};
\draw [color=qqwuqq] (-1.8, 1.8) node{$q$};
\draw [color=ffqqqq] (-1.8,-1.8) node{$r$};
\draw [color=ffqqqq] (-1.65, 2.35) node[above]{$r$};
\draw [color=ffqqqq] ( 2.35,-1.65) node[right]{$r$};
\draw [color=ffqqqq] ( 1.8, 1.8) node{$s$};
\draw [color=ffqqqq] ( 1.65,-2.35) node[below]{$s$};
\draw [color=ffqqqq] (-2.35, 1.65) node[left]{$s$};
\draw [blue][line width=1pt][postaction={on each segment={mid arrow=blue}}]
      (1.65, -1.65) to[out=135, in=90] (-0.2,-1.9);
\draw [blue] (-0.1,-1.2) node{\tiny$\circled{1}$};
\draw [blue][line width=1pt][postaction={on each segment={mid arrow=blue}}]
      ( 0.5, 1.9) to[out=-90,in=135] (0.9, 0.9);
\draw [blue] ( 0.2, 1.2) node{\tiny$\circled{2}$};
\draw [blue][line width=1pt][postaction={on each segment={mid arrow=blue}}]
      (0.5, 0.5) to[out=-45,in= 180] (2, 0);
\draw [blue] (1.4, 0.4) node{\tiny$\circled{3}$};
\draw [blue][line width=1pt][postaction={on each segment={mid arrow=blue}}]
      (-2.1, 0.4) to[out=0,in= -90] (-0.4, 2.1);
\draw [blue] (-0.6, 0.6) node{\tiny$\circled{4}$};
\draw [blue][line width=1pt][postaction={on each segment={mid arrow=blue}}]
      (0.4,-2.1) to[out=90,in= 160] (1.65,-1.65);
\draw [blue] (1,-2) node{\tiny$\circled{5}$};
\draw ( 2.0,-0.0) node[right]{\small\color{red}$a_1$};
\draw (-2.0, 0.0) node[left]{\small\color{red}$a_1$};
\draw (-0.0,-0.0) node[right]{\small\color{red}$a_2$};
\draw (-0.0, 2.0) node[above]{\small\color{red}$a_3$};
\draw (-0.0,-2.0) node[below]{\small\color{red}$a_3$};
\end{tikzpicture}
\begin{tikzpicture}[scale=0.55]
\filldraw [black!20] (-2, 2) circle (0.5cm);
\filldraw [black!20] ( 2, 2) circle (0.5cm);
\filldraw [black!20] ( 2,-2) circle (0.5cm);
\filldraw [black!20] (-2,-2) circle (0.5cm);
\draw [line width=1pt] (-2, 2) circle (0.5cm);
\draw [line width=1pt] ( 2, 2) circle (0.5cm);
\draw [line width=1pt] ( 2,-2) circle (0.5cm);
\draw [line width=1pt] (-2,-2) circle (0.5cm);
\draw [dash pattern=on 2pt off 2pt] (-2, 2)-- ( 2, 2);
\draw [dash pattern=on 2pt off 2pt] ( 2, 2)-- ( 2,-2);
\draw [dash pattern=on 2pt off 2pt] ( 2,-2)-- (-2,-2);
\draw [dash pattern=on 2pt off 2pt] (-2, 2)-- (-2,-2);
\draw [line width=1pt,color=ffqqqq] (-2.35,1.65)-- (-1.65,-1.65);
\draw [line width=1pt,color=ffqqqq] (-1.65,2.35)-- (1.65,1.65);
\draw [line width=1pt,color=ffqqqq] (1.65,1.65)-- (2.35,-1.65);
\draw [line width=1pt,color=ffqqqq] (1.65,-2.35)-- (-1.65,-1.65);
\draw [line width=1pt,color=ffqqqq] (-1.65,-1.65)-- (1.65,1.65);
\fill [color=qqwuqq] (-1.65, 1.65) circle (2.5pt);
\fill [color=qqwuqq] ( 1.65,-1.65) circle (2.5pt);
\filldraw [color=white, draw=ffqqqq, line width=0.5pt] (-1.65, 2.35) circle (2.5pt);
\fill [color=qqwuqq] (-2.35, 2.35) circle (2.5pt);
\filldraw [color=white, draw=ffqqqq, line width=0.5pt] (-1.65,-1.65) circle (2.5pt);
\filldraw [color=white, draw=ffqqqq, line width=0.5pt] (-2.35,-2.35) circle (2.5pt);
\fill [color=qqwuqq] (-2.35,-1.65) circle (2.5pt);
\fill [color=qqwuqq] (-1.65,-2.35) circle (2.5pt);
\filldraw [color=white, draw=ffqqqq, line width=0.5pt] (-2.35, 1.65) circle (2.5pt);
\fill [color=qqwuqq] ( 1.65, 2.35) circle (2.5pt);
\filldraw [color=white, draw=ffqqqq, line width=0.5pt] ( 1.65, 1.65) circle (2.5pt);
\filldraw [color=white, draw=ffqqqq, line width=0.5pt] ( 2.35, 2.35) circle (2.5pt);
\fill [color=qqwuqq] ( 2.35, 1.65) circle (2.5pt);
\filldraw [color=white, draw=ffqqqq, line width=0.5pt] ( 2.35,-1.65) circle (2.5pt);
\fill [color=qqwuqq] ( 2.35,-2.35) circle (2.5pt);
\filldraw [color=white, draw=ffqqqq, line width=0.5pt] ( 1.65,-2.35) circle (2.5pt);
\draw [color=qqwuqq] ( 1.8,-1.8) node{$p$};
\draw [color=qqwuqq] (-1.8, 1.8) node{$q$};
\draw [color=ffqqqq] (-1.8,-1.8) node{$r$};
\draw [color=ffqqqq] (-1.65, 2.35) node[above]{$r$};
\draw [color=ffqqqq] ( 2.35,-1.65) node[right]{$r$};
\draw [color=ffqqqq] ( 1.8, 1.8) node{$s$};
\draw [color=ffqqqq] ( 1.65,-2.35) node[below]{$s$};
\draw [color=ffqqqq] (-2.35, 1.65) node[left]{$s$};
\draw [blue][line width=1pt][postaction={on each segment={mid arrow=blue}}]
      (1.65, -1.65) to[out=180, in=90] (0.2,-2.1);
\draw [blue] (0.5,-1.4) node{\tiny$\circled{1}$};
\draw [blue][line width=1pt][postaction={on each segment={mid arrow=blue}}]
      ( 0.5, 1.9) to[out=-90,in=135] (0.9, 0.9);
\draw [blue] ( 0.2, 1.2) node{\tiny$\circled{2}$};
\draw [blue][line width=1pt][postaction={on each segment={mid arrow=blue}}]
      (0.5, 0.5) to[out=-45,in= 180] (2, 0);
\draw [blue] (1.4, 0.4) node{\tiny$\circled{3}$};
\draw [blue][line width=1pt][postaction={on each segment={mid arrow=blue}}]
      (-2.1, 0.4) to[out=0,in= -90] (-0.4, 2.1);
\draw [blue] (-0.6, 0.6) node{\tiny$\circled{4}$};
\draw [blue][line width=1pt][postaction={on each segment={mid arrow=blue}}]
      (-0.5,-1.9) to[out=90,in= -45] (-0.9,-0.9);
\draw [blue] (-0.3,-1.2) node{\tiny$\circled{5}$};
\draw [blue][line width=1pt][postaction={on each segment={mid arrow=blue}}]
      (-0.8,-0.8) to[out=135,in= 0] (-1.9,-0.5);
\draw [blue] (-1.1,-0.2) node{\tiny$\circled{6}$};
\draw [blue][line width=1pt][postaction={on each segment={mid arrow=blue}}]
      (2.1,-0.3) to[out=180,in= 90] (1.65,-1.65);
\draw [blue] (1.2, -0.7) node{\tiny$\circled{7}$};
\draw ( 2.0,-0.0) node[right]{\small\color{red}$a_1$};
\draw (-2.0, 0.0) node[left]{\small\color{red}$a_1$};
\draw (-0.0,-0.0) node[right]{\small\color{red}$a_2$};
\draw (-0.0, 2.0) node[above]{\small\color{red}$a_3$};
\draw (-0.0,-2.0) node[below]{\small\color{red}$a_3$};
\end{tikzpicture}
\begin{tikzpicture}[scale=0.55]
\filldraw [black!20] (-2, 2) circle (0.5cm);
\filldraw [black!20] ( 2, 2) circle (0.5cm);
\filldraw [black!20] ( 2,-2) circle (0.5cm);
\filldraw [black!20] (-2,-2) circle (0.5cm);
\draw [line width=1pt] (-2, 2) circle (0.5cm);
\draw [line width=1pt] ( 2, 2) circle (0.5cm);
\draw [line width=1pt] ( 2,-2) circle (0.5cm);
\draw [line width=1pt] (-2,-2) circle (0.5cm);
\draw [dash pattern=on 2pt off 2pt] (-2, 2)-- ( 2, 2);
\draw [dash pattern=on 2pt off 2pt] ( 2, 2)-- ( 2,-2);
\draw [dash pattern=on 2pt off 2pt] ( 2,-2)-- (-2,-2);
\draw [dash pattern=on 2pt off 2pt] (-2, 2)-- (-2,-2);
\draw [line width=1pt,color=ffqqqq] (-2.35,1.65)-- (-1.65,-1.65);
\draw [line width=1pt,color=ffqqqq] (-1.65,2.35)-- (1.65,1.65);
\draw [line width=1pt,color=ffqqqq] (1.65,1.65)-- (2.35,-1.65);
\draw [line width=1pt,color=ffqqqq] (1.65,-2.35)-- (-1.65,-1.65);
\draw [line width=1pt,color=ffqqqq] (-1.65,-1.65)-- (1.65,1.65);
\fill [color=qqwuqq] (-1.65, 1.65) circle (2.5pt);
\fill [color=qqwuqq] ( 1.65,-1.65) circle (2.5pt);
\filldraw [color=white, draw=ffqqqq, line width=0.5pt] (-1.65, 2.35) circle (2.5pt);
\fill [color=qqwuqq] (-2.35, 2.35) circle (2.5pt);
\filldraw [color=white, draw=ffqqqq, line width=0.5pt] (-1.65,-1.65) circle (2.5pt);
\filldraw [color=white, draw=ffqqqq, line width=0.5pt] (-2.35,-2.35) circle (2.5pt);
\fill [color=qqwuqq] (-2.35,-1.65) circle (2.5pt);
\fill [color=qqwuqq] (-1.65,-2.35) circle (2.5pt);
\filldraw [color=white, draw=ffqqqq, line width=0.5pt] (-2.35, 1.65) circle (2.5pt);
\fill [color=qqwuqq] ( 1.65, 2.35) circle (2.5pt);
\filldraw [color=white, draw=ffqqqq, line width=0.5pt] ( 1.65, 1.65) circle (2.5pt);
\filldraw [color=white, draw=ffqqqq, line width=0.5pt] ( 2.35, 2.35) circle (2.5pt);
\fill [color=qqwuqq] ( 2.35, 1.65) circle (2.5pt);
\filldraw [color=white, draw=ffqqqq, line width=0.5pt] ( 2.35,-1.65) circle (2.5pt);
\fill [color=qqwuqq] ( 2.35,-2.35) circle (2.5pt);
\filldraw [color=white, draw=ffqqqq, line width=0.5pt] ( 1.65,-2.35) circle (2.5pt);
\draw [color=qqwuqq] ( 1.8,-1.8) node{$p$};
\draw [color=qqwuqq] (-1.8, 1.8) node{$q$};
\draw [color=ffqqqq] (-1.8,-1.8) node{$r$};
\draw [color=ffqqqq] (-1.65, 2.35) node[above]{$r$};
\draw [color=ffqqqq] ( 2.35,-1.65) node[right]{$r$};
\draw [color=ffqqqq] ( 1.8, 1.8) node{$s$};
\draw [color=ffqqqq] ( 1.65,-2.35) node[below]{$s$};
\draw [color=ffqqqq] (-2.35, 1.65) node[left]{$s$};
\draw [blue][line width=1pt][postaction={on each segment={mid arrow=blue}}]
      (1.65, -1.65) to[out=180, in=90] (0.3,-2.1);
\draw [blue] (1.25,-1.3) node{\tiny$\circled{1}$};
\draw [blue][line width=1pt][postaction={on each segment={mid arrow=blue}}]
      ( 0.5, 1.9) to[out=-90,in=135] (0.9, 0.9);
\draw [blue] ( 0.2, 1.2) node{\tiny$\circled{2}$};
\draw [blue][line width=1pt][postaction={on each segment={mid arrow=blue}}]
      (0.5, 0.5) to[out=-45,in= 180] (2, 0);
\draw [blue] (1.4, 0.4) node{\tiny$\circled{3}$};
\draw [blue][line width=1pt][postaction={on each segment={mid arrow=blue}}]
      (-2.1, 0.4) to[out=0,in= -90] (-0.4, 2.1);
\draw [blue] (-0.6, 0.6) node{\tiny$\circled{4}$};
\draw [blue][line width=1pt][postaction={on each segment={mid arrow=blue}}]
      (-0.5,-1.9) to[out=90,in= -45] (-0.9,-0.9);
\draw [blue] (-0.3,-1.2) node{\tiny$\circled{5}$};
\draw [blue][line width=1pt][postaction={on each segment={mid arrow=blue}}]
      (-0.8,-0.8) to[out=135,in= 0] (-1.9,-0.5);
\draw [blue] (-1.1,-0.2) node{\tiny$\circled{6}$};
\draw [blue][line width=1pt][postaction={on each segment={mid arrow=blue}}]
      (2.1,-0.3) to[out=180,in= 90] (0,-2);
\draw [blue] (0.3, -0.6) node{\tiny$\circled{7}$};
\end{tikzpicture}
\end{center}
\caption{Case 3.2}
\label{fig:case 3.2}
\end{figure}
        we have (S.L.E.), a contradiction. So $t>7$. By Lemma~\ref{lem:cancel}, $\snum{i_8}{\ws}=3$, see the fourth picture of Figure~\ref{fig:case 3.2}. However, this contradicts Lemma~\ref{lem:four}.
    \end{itemize}
\end{itemize}
\end{proof}

Now we are ready to show the main result.

\begin{proof}[Proof of Proposition~\ref{prop:main}]
By Proposition~\ref{prop:presilt.ind.obj.}, this is equivalent to showing that for any $\ws\in\SGOCM(\SURF)$, if $\ws(0)\neq \ws(1)$ and $\ws$ is not homotopic to $\tgamma$, then $X(\ws)\oplus X(\tgamma)$ is not presilting. Note that $p$ is an orientated intersection from $\tgamma$ to $\ws$, and $q$ is an orientation intersection from $\ws$ to $\tgamma$. By Theorem~\ref{thm:ops}~(3), it suffices to show that either $\ii_p(\tgamma,\tsigma)> 0$ or $\ii_q(\tsigma,\tgamma)>0$. Reversing the direction of $\ws$ if necessary, we may assume $\sigma(0)=p$ and $\sigma(1)=q$. By Lemma~\ref{lem:circle} and the second formula in Lemma~\ref{lem:int ind}, we may assume that $\ws$ does not contain a circle. Assume $\ii_p(\tgamma,\tsigma)\leq 0$. Then $\sdii{1}{\ws}:=\dii{1}{\ws}\geq 2$. To show $\ii_q(\tsigma,\tgamma)>0$, by the second formula in Lemma~\ref{lem:int ind}, we only need to show $\sdii{n(\ws)}{\ws}:=\dii{n(\ws)}{\ws}> 0$. Let $0=i_1<i_2<\cdots<i_t=n(\ws)$ be a simplest sequence of $\ws$. Note that $\sigma(0)\neq \sigma(1)$ implies that $t$ is even. If $t=2$, then  $\sdii{n(\ws)}{\ws}=\sdii{i_2}{\ws}=\sdii{i_1+1}{\ws}=\sdii{1}{\ws}\geq 2>0$ as required. Now we assume $t>2$. There are the following cases.

\begin{itemize}
    \item[Case A] $\snum{i_2}{\ws}=1$.
    \begin{itemize}
        \item [Case A.1] $\snum{i_3}{\ws}=2$. By Lemma~\ref{lem:cancel}, $\snum{i_4}{\ws}=3$. If $t=4$, see the first picture of \Pic\ref{fig:case A.1}, then by the first formula in Lemma~\ref{lem:int ind} (the same below), we have $\sdii{n(\ws)}{\ws}=\sdii{1}{\ws}+2\geq 4>0$ as required. Now we assume $t>4$. By Lemma~\ref{lem:cancel}, $\snum{i_5}{\ws}=1$ and $\snum{i_6}{\ws}=2$. If $t=6$, see the second picture of \Pic\ref{fig:case A.1}, then $\ws_{i_6,i_6+1}$ crosses $\ws_{i_4,i_4+1}$, a contradiction. So $t>6$. By Lemma~\ref{lem:cancel}, $\snum{i_7}{\ws}=3$ and $\snum{i_8}{\ws}=1$, see the third picture of \Pic\ref{fig:case A.1}. However, this contradicts Lemma~\ref{lem:four}.

\begin{figure}[htbp]
\begin{center}
\definecolor{ffqqqq}{rgb}{1,0,0}
\definecolor{qqwuqq}{rgb}{0,0,1}
\begin{tikzpicture}[scale=0.55]
\filldraw [black!20] (-2, 2) circle (0.5cm);
\filldraw [black!20] ( 2, 2) circle (0.5cm);
\filldraw [black!20] ( 2,-2) circle (0.5cm);
\filldraw [black!20] (-2,-2) circle (0.5cm);
\draw [line width=1pt] (-2, 2) circle (0.5cm);
\draw [line width=1pt] ( 2, 2) circle (0.5cm);
\draw [line width=1pt] ( 2,-2) circle (0.5cm);
\draw [line width=1pt] (-2,-2) circle (0.5cm);
\draw [dash pattern=on 2pt off 2pt] (-2, 2)-- ( 2, 2);
\draw [dash pattern=on 2pt off 2pt] ( 2, 2)-- ( 2,-2);
\draw [dash pattern=on 2pt off 2pt] ( 2,-2)-- (-2,-2);
\draw [dash pattern=on 2pt off 2pt] (-2, 2)-- (-2,-2);
\draw [line width=1pt,color=ffqqqq] (-2.35,1.65)-- (-1.65,-1.65);
\draw [line width=1pt,color=ffqqqq] (-1.65,2.35)-- (1.65,1.65);
\draw [line width=1pt,color=ffqqqq] (1.65,1.65)-- (2.35,-1.65);
\draw [line width=1pt,color=ffqqqq] (1.65,-2.35)-- (-1.65,-1.65);
\draw [line width=1pt,color=ffqqqq] (-1.65,-1.65)-- (1.65,1.65);
\fill [color=qqwuqq] (-1.65, 1.65) circle (2.5pt);
\fill [color=qqwuqq] ( 1.65,-1.65) circle (2.5pt);
\filldraw [color=white, draw=ffqqqq, line width=0.5pt] (-1.65, 2.35) circle (2.5pt);
\fill [color=qqwuqq] (-2.35, 2.35) circle (2.5pt);
\filldraw [color=white, draw=ffqqqq, line width=0.5pt] (-1.65,-1.65) circle (2.5pt);
\filldraw [color=white, draw=ffqqqq, line width=0.5pt] (-2.35,-2.35) circle (2.5pt);
\fill [color=qqwuqq] (-2.35,-1.65) circle (2.5pt);
\fill [color=qqwuqq] (-1.65,-2.35) circle (2.5pt);
\filldraw [color=white, draw=ffqqqq, line width=0.5pt] (-2.35, 1.65) circle (2.5pt);
\fill [color=qqwuqq] ( 1.65, 2.35) circle (2.5pt);
\filldraw [color=white, draw=ffqqqq, line width=0.5pt] ( 1.65, 1.65) circle (2.5pt);
\filldraw [color=white, draw=ffqqqq, line width=0.5pt] ( 2.35, 2.35) circle (2.5pt);
\fill [color=qqwuqq] ( 2.35, 1.65) circle (2.5pt);
\filldraw [color=white, draw=ffqqqq, line width=0.5pt] ( 2.35,-1.65) circle (2.5pt);
\fill [color=qqwuqq] ( 2.35,-2.35) circle (2.5pt);
\filldraw [color=white, draw=ffqqqq, line width=0.5pt] ( 1.65,-2.35) circle (2.5pt);
\draw [color=qqwuqq] ( 1.8,-1.8) node{$p$};
\draw [color=qqwuqq] (-1.8, 1.8) node{$q$};
\draw [color=ffqqqq] (-1.8,-1.8) node{$r$};
\draw [color=ffqqqq] (-1.65, 2.35) node[above]{$r$};
\draw [color=ffqqqq] ( 2.35,-1.65) node[right]{$r$};
\draw [color=ffqqqq] ( 1.8, 1.8) node{$s$};
\draw [color=ffqqqq] ( 1.65,-2.35) node[below]{$s$};
\draw [color=ffqqqq] (-2.35, 1.65) node[left]{$s$};
\draw [blue][line width=1pt][postaction={on each segment={mid arrow=blue}}]
      (1.65, -1.65) to[out=90, in=180] (2.1, -0.5);
\draw [blue] (1.5, -0.5) node{\tiny$\circled{1}$};
\draw [blue][line width=1pt][postaction={on each segment={mid arrow=blue}}]
      (-1.9, -0.2) to[out= 0,in=135] (-0.6,-0.6);
\draw [blue] (-1.2, 0.1) node{\tiny$\circled{2}$};
\draw [blue][line width=1pt][postaction={on each segment={mid arrow=blue}}]
      (-0.8,-0.8) to[out=-45,in=90] (-0.3, -1.9);
\draw [blue] (-0. ,-1.5) node{\tiny$\circled{3}$};
\draw [blue][line width=1pt][postaction={on each segment={mid arrow=blue}}]
      (-0.7, 2.2) to[out=-90,in=0] (-1.65, 1.65);
\draw [blue] (-0.6, 1.5) node{\tiny$\circled{4}$};
\draw ( 2.0,-0.0) node[right]{\small\color{red}$a_1$};
\draw (-2.0, 0.0) node[left]{\small\color{red}$a_1$};
\draw (-0.0,-0.0) node[right]{\small\color{red}$a_2$};
\draw (-0.0, 2.0) node[above]{\small\color{red}$a_3$};
\draw (-0.0,-2.0) node[below]{\small\color{red}$a_3$};
\end{tikzpicture}
\
\begin{tikzpicture}[scale=0.55]
\filldraw [black!20] (-2, 2) circle (0.5cm);
\filldraw [black!20] ( 2, 2) circle (0.5cm);
\filldraw [black!20] ( 2,-2) circle (0.5cm);
\filldraw [black!20] (-2,-2) circle (0.5cm);
\draw [line width=1pt] (-2, 2) circle (0.5cm);
\draw [line width=1pt] ( 2, 2) circle (0.5cm);
\draw [line width=1pt] ( 2,-2) circle (0.5cm);
\draw [line width=1pt] (-2,-2) circle (0.5cm);
\draw [dash pattern=on 2pt off 2pt] (-2, 2)-- ( 2, 2);
\draw [dash pattern=on 2pt off 2pt] ( 2, 2)-- ( 2,-2);
\draw [dash pattern=on 2pt off 2pt] ( 2,-2)-- (-2,-2);
\draw [dash pattern=on 2pt off 2pt] (-2, 2)-- (-2,-2);
\draw [line width=1pt,color=ffqqqq] (-2.35,1.65)-- (-1.65,-1.65);
\draw [line width=1pt,color=ffqqqq] (-1.65,2.35)-- (1.65,1.65);
\draw [line width=1pt,color=ffqqqq] (1.65,1.65)-- (2.35,-1.65);
\draw [line width=1pt,color=ffqqqq] (1.65,-2.35)-- (-1.65,-1.65);
\draw [line width=1pt,color=ffqqqq] (-1.65,-1.65)-- (1.65,1.65);
\fill [color=qqwuqq] (-1.65, 1.65) circle (2.5pt);
\fill [color=qqwuqq] ( 1.65,-1.65) circle (2.5pt);
\filldraw [color=white, draw=ffqqqq, line width=0.5pt] (-1.65, 2.35) circle (2.5pt);
\fill [color=qqwuqq] (-2.35, 2.35) circle (2.5pt);
\filldraw [color=white, draw=ffqqqq, line width=0.5pt] (-1.65,-1.65) circle (2.5pt);
\filldraw [color=white, draw=ffqqqq, line width=0.5pt] (-2.35,-2.35) circle (2.5pt);
\fill [color=qqwuqq] (-2.35,-1.65) circle (2.5pt);
\fill [color=qqwuqq] (-1.65,-2.35) circle (2.5pt);
\filldraw [color=white, draw=ffqqqq, line width=0.5pt] (-2.35, 1.65) circle (2.5pt);
\fill [color=qqwuqq] ( 1.65, 2.35) circle (2.5pt);
\filldraw [color=white, draw=ffqqqq, line width=0.5pt] ( 1.65, 1.65) circle (2.5pt);
\filldraw [color=white, draw=ffqqqq, line width=0.5pt] ( 2.35, 2.35) circle (2.5pt);
\fill [color=qqwuqq] ( 2.35, 1.65) circle (2.5pt);
\filldraw [color=white, draw=ffqqqq, line width=0.5pt] ( 2.35,-1.65) circle (2.5pt);
\fill [color=qqwuqq] ( 2.35,-2.35) circle (2.5pt);
\filldraw [color=white, draw=ffqqqq, line width=0.5pt] ( 1.65,-2.35) circle (2.5pt);
\draw [color=qqwuqq] ( 1.8,-1.8) node{$p$};
\draw [color=qqwuqq] (-1.8, 1.8) node{$q$};
\draw [color=ffqqqq] (-1.8,-1.8) node{$r$};
\draw [color=ffqqqq] (-1.65, 2.35) node[above]{$r$};
\draw [color=ffqqqq] ( 2.35,-1.65) node[right]{$r$};
\draw [color=ffqqqq] ( 1.8, 1.8) node{$s$};
\draw [color=ffqqqq] ( 1.65,-2.35) node[below]{$s$};
\draw [color=ffqqqq] (-2.35, 1.65) node[left]{$s$};
\draw [blue][line width=1pt][postaction={on each segment={mid arrow=blue}}]
      (1.65, -1.65) to[out=90, in=180] (2.1, -0.5);
\draw [blue] (1.5, -0.5) node{\tiny$\circled{1}$};
\draw [blue][line width=1pt][postaction={on each segment={mid arrow=blue}}]
      (-1.9, -0.2) to[out= 0,in=135] (-0.6,-0.6);
\draw [blue] (-1.2, 0.1) node{\tiny$\circled{2}$};
\draw [blue][line width=1pt][postaction={on each segment={mid arrow=blue}}]
      (-0.8,-0.8) to[out=-45,in=90] (-0.3, -1.9);
\draw [blue] (-0. ,-1.5) node{\tiny$\circled{3}$};
\draw [blue][line width=1pt][postaction={on each segment={mid arrow=blue}}]
      (-0.7, 2.2) to[out=-90,in=0] (-2.2, 0.7);
\draw [blue] (-0.5, 1.6) node{\tiny$\circled{4}$};
\draw [blue][line width=1pt][postaction={on each segment={mid arrow=blue}}]
      (1.8, 0.8) to[out=180, in=-45] (1.1, 1.1);
\draw [blue] (1.5, 0.4) node{\tiny$\circled{5}$};
\draw [blue][line width=1pt][postaction={on each segment={mid arrow=blue}}]
      (0.8, 0.8) to[out=180, in=-45] (-1.65, 1.65);
\draw [blue] (0, 0.5) node{\tiny$\circled{6}$};
\draw ( 2.0,-0.0) node[right]{\small\color{red}$a_1$};
\draw (-2.0, 0.0) node[left]{\small\color{red}$a_1$};
\draw (-0.0,-0.0) node[right]{\small\color{red}$a_2$};
\draw (-0.0, 2.0) node[above]{\small\color{red}$a_3$};
\draw (-0.0,-2.0) node[below]{\small\color{red}$a_3$};
\end{tikzpicture}
\
\begin{tikzpicture}[scale=0.55]
\filldraw [black!20] (-2, 2) circle (0.5cm);
\filldraw [black!20] ( 2, 2) circle (0.5cm);
\filldraw [black!20] ( 2,-2) circle (0.5cm);
\filldraw [black!20] (-2,-2) circle (0.5cm);
\draw [line width=1pt] (-2, 2) circle (0.5cm);
\draw [line width=1pt] ( 2, 2) circle (0.5cm);
\draw [line width=1pt] ( 2,-2) circle (0.5cm);
\draw [line width=1pt] (-2,-2) circle (0.5cm);
\draw [dash pattern=on 2pt off 2pt] (-2, 2)-- ( 2, 2);
\draw [dash pattern=on 2pt off 2pt] ( 2, 2)-- ( 2,-2);
\draw [dash pattern=on 2pt off 2pt] ( 2,-2)-- (-2,-2);
\draw [dash pattern=on 2pt off 2pt] (-2, 2)-- (-2,-2);
\draw [line width=1pt,color=ffqqqq] (-2.35,1.65)-- (-1.65,-1.65);
\draw [line width=1pt,color=ffqqqq] (-1.65,2.35)-- (1.65,1.65);
\draw [line width=1pt,color=ffqqqq] (1.65,1.65)-- (2.35,-1.65);
\draw [line width=1pt,color=ffqqqq] (1.65,-2.35)-- (-1.65,-1.65);
\draw [line width=1pt,color=ffqqqq] (-1.65,-1.65)-- (1.65,1.65);
\fill [color=qqwuqq] (-1.65, 1.65) circle (2.5pt);
\fill [color=qqwuqq] ( 1.65,-1.65) circle (2.5pt);
\filldraw [color=white, draw=ffqqqq, line width=0.5pt] (-1.65, 2.35) circle (2.5pt);
\fill [color=qqwuqq] (-2.35, 2.35) circle (2.5pt);
\filldraw [color=white, draw=ffqqqq, line width=0.5pt] (-1.65,-1.65) circle (2.5pt);
\filldraw [color=white, draw=ffqqqq, line width=0.5pt] (-2.35,-2.35) circle (2.5pt);
\fill [color=qqwuqq] (-2.35,-1.65) circle (2.5pt);
\fill [color=qqwuqq] (-1.65,-2.35) circle (2.5pt);
\filldraw [color=white, draw=ffqqqq, line width=0.5pt] (-2.35, 1.65) circle (2.5pt);
\fill [color=qqwuqq] ( 1.65, 2.35) circle (2.5pt);
\filldraw [color=white, draw=ffqqqq, line width=0.5pt] ( 1.65, 1.65) circle (2.5pt);
\filldraw [color=white, draw=ffqqqq, line width=0.5pt] ( 2.35, 2.35) circle (2.5pt);
\fill [color=qqwuqq] ( 2.35, 1.65) circle (2.5pt);
\filldraw [color=white, draw=ffqqqq, line width=0.5pt] ( 2.35,-1.65) circle (2.5pt);
\fill [color=qqwuqq] ( 2.35,-2.35) circle (2.5pt);
\filldraw [color=white, draw=ffqqqq, line width=0.5pt] ( 1.65,-2.35) circle (2.5pt);
\draw [color=qqwuqq] ( 1.8,-1.8) node{$p$};
\draw [color=qqwuqq] (-1.8, 1.8) node{$q$};
\draw [color=ffqqqq] (-1.8,-1.8) node{$r$};
\draw [color=ffqqqq] (-1.65, 2.35) node[above]{$r$};
\draw [color=ffqqqq] ( 2.35,-1.65) node[right]{$r$};
\draw [color=ffqqqq] ( 1.8, 1.8) node{$s$};
\draw [color=ffqqqq] ( 1.65,-2.35) node[below]{$s$};
\draw [color=ffqqqq] (-2.35, 1.65) node[left]{$s$};
\draw [blue][line width=1pt][postaction={on each segment={mid arrow=blue}}]
      (1.65, -1.65) to[out=90, in=180] (2.1, -0.5);
\draw [blue] (1.4, -0.8) node{\tiny$\circled{1}$};
\draw [blue][line width=1pt][postaction={on each segment={mid arrow=blue}}]
      (-1.9, -0.2) to[out= 0,in=135] (-0.6,-0.6);
\draw [blue] (-1.2, 0.1) node{\tiny$\circled{2}$};
\draw [blue][line width=1pt][postaction={on each segment={mid arrow=blue}}]
      (-0.8,-0.8) to[out=-45,in=90] (-0.3, -1.9);
\draw [blue] (-0.8,-1.4) node{\tiny$\circled{3}$};
\draw [blue][line width=1pt][postaction={on each segment={mid arrow=blue}}]
      (-0.7, 2.2) to[out=-90,in=0] (-2.2, 0.7);
\draw [blue] (-0.5, 1.6) node{\tiny$\circled{4}$};
\draw [blue][line width=1pt][postaction={on each segment={mid arrow=blue}}]
      (1.8, 0.8) to[out=180, in=-45] (1.1, 1.1);
\draw [blue] (1.5, 0.4) node{\tiny$\circled{5}$};
\draw [blue][line width=1pt][postaction={on each segment={mid arrow=blue}}]
      (0.8, 0.8) to[out=135, in=-90] (0.3, 1.9);
\draw [blue] (0.1, 1) node{\tiny$\circled{6}$};
\draw [blue][line width=1pt][postaction={on each segment={mid arrow=blue}}]
      (0.2, -2) to[out=90, in=180] (2,-0.2);
\draw [blue] (0.8, -1.3) node{\tiny$\circled{7}$};
\draw ( 2.0,-0.0) node[right]{\small\color{red}$a_1$};
\draw (-2.0, 0.0) node[left]{\small\color{red}$a_1$};
\draw (-0.0,-0.0) node[right]{\small\color{red}$a_2$};
\draw (-0.0, 2.0) node[above]{\small\color{red}$a_3$};
\draw (-0.0,-2.0) node[below]{\small\color{red}$a_3$};
\end{tikzpicture}
\end{center}
\caption{Case A.1}
\label{fig:case A.1}
\end{figure}

        \item [Case A.2] $\snum{i_3}{\ws}=3$. By Lemma~\ref{lem:cancel}, $\snum{i_4}{\ws}=2$. If $t=4$, then $\ws_{i_4,i_4+1}$ crosses $\ws_{i_2,i_2+1}$, see the first picture of \Pic\ref{fig:case A.2}, a contradiction. So $t>4$. By Lemma~\ref{lem:cancel}, $\snum{i_5}{\ws}=1$ and $\snum{i_6}{\ws}=3$. If $t=6$, see the second picture of \Pic\ref{fig:case A.2}, then $\sdii{n(\ws)}{\ws}=\sdii{1}{\ws}\geq 2>0$ as required. Now assume $t>6$. By Lemma~\ref{lem:cancel}, $\snum{i_7}{\ws}=2$ and $\snum{i_8}{\ws}=1$, see the third picture of \Pic\ref{fig:case A.2}. However, this contradicts Lemma~\ref{lem:four}.
\begin{figure}[htbp]
\begin{center}
\definecolor{ffqqqq}{rgb}{1,0,0}
\definecolor{qqwuqq}{rgb}{0,0,1}
\begin{tikzpicture}[scale=0.55]
\filldraw [black!20] (-2, 2) circle (0.5cm);
\filldraw [black!20] ( 2, 2) circle (0.5cm);
\filldraw [black!20] ( 2,-2) circle (0.5cm);
\filldraw [black!20] (-2,-2) circle (0.5cm);
\draw [line width=1pt] (-2, 2) circle (0.5cm);
\draw [line width=1pt] ( 2, 2) circle (0.5cm);
\draw [line width=1pt] ( 2,-2) circle (0.5cm);
\draw [line width=1pt] (-2,-2) circle (0.5cm);
\draw [dash pattern=on 2pt off 2pt] (-2, 2)-- ( 2, 2);
\draw [dash pattern=on 2pt off 2pt] ( 2, 2)-- ( 2,-2);
\draw [dash pattern=on 2pt off 2pt] ( 2,-2)-- (-2,-2);
\draw [dash pattern=on 2pt off 2pt] (-2, 2)-- (-2,-2);
\draw [line width=1pt,color=ffqqqq] (-2.35,1.65)-- (-1.65,-1.65);
\draw [line width=1pt,color=ffqqqq] (-1.65,2.35)-- (1.65,1.65);
\draw [line width=1pt,color=ffqqqq] (1.65,1.65)-- (2.35,-1.65);
\draw [line width=1pt,color=ffqqqq] (1.65,-2.35)-- (-1.65,-1.65);
\draw [line width=1pt,color=ffqqqq] (-1.65,-1.65)-- (1.65,1.65);
\fill [color=qqwuqq] (-1.65, 1.65) circle (2.5pt);
\fill [color=qqwuqq] ( 1.65,-1.65) circle (2.5pt);
\filldraw [color=white, draw=ffqqqq, line width=0.5pt] (-1.65, 2.35) circle (2.5pt);
\fill [color=qqwuqq] (-2.35, 2.35) circle (2.5pt);
\filldraw [color=white, draw=ffqqqq, line width=0.5pt] (-1.65,-1.65) circle (2.5pt);
\filldraw [color=white, draw=ffqqqq, line width=0.5pt] (-2.35,-2.35) circle (2.5pt);
\fill [color=qqwuqq] (-2.35,-1.65) circle (2.5pt);
\fill [color=qqwuqq] (-1.65,-2.35) circle (2.5pt);
\filldraw [color=white, draw=ffqqqq, line width=0.5pt] (-2.35, 1.65) circle (2.5pt);
\fill [color=qqwuqq] ( 1.65, 2.35) circle (2.5pt);
\filldraw [color=white, draw=ffqqqq, line width=0.5pt] ( 1.65, 1.65) circle (2.5pt);
\filldraw [color=white, draw=ffqqqq, line width=0.5pt] ( 2.35, 2.35) circle (2.5pt);
\fill [color=qqwuqq] ( 2.35, 1.65) circle (2.5pt);
\filldraw [color=white, draw=ffqqqq, line width=0.5pt] ( 2.35,-1.65) circle (2.5pt);
\fill [color=qqwuqq] ( 2.35,-2.35) circle (2.5pt);
\filldraw [color=white, draw=ffqqqq, line width=0.5pt] ( 1.65,-2.35) circle (2.5pt);
\draw [color=qqwuqq] ( 1.8,-1.8) node{$p$};
\draw [color=qqwuqq] (-1.8, 1.8) node{$q$};
\draw [color=ffqqqq] (-1.8,-1.8) node{$r$};
\draw [color=ffqqqq] (-1.65, 2.35) node[above]{$r$};
\draw [color=ffqqqq] ( 2.35,-1.65) node[right]{$r$};
\draw [color=ffqqqq] ( 1.8, 1.8) node{$s$};
\draw [color=ffqqqq] ( 1.65,-2.35) node[below]{$s$};
\draw [color=ffqqqq] (-2.35, 1.65) node[left]{$s$};
\draw [blue][line width=1pt][postaction={on each segment={mid arrow=blue}}]
      (1.65, -1.65) to[out=90, in=180] (2.1, -0.5);
\draw [blue] (1.5, -0.6) node{\tiny$\circled{1}$};
\draw [blue][line width=1pt][postaction={on each segment={mid arrow=blue}}]
      (-2, 0.5) to[out= 0,in=-90] (-0.2, 2);
\draw [blue] (-0.1, 1.1) node{\tiny$\circled{2}$};
\draw [blue][line width=1pt][postaction={on each segment={mid arrow=blue}}]
      (-0.4,-1.9) to[out=90, in=-45] (-0.8,-0.8);
\draw [blue] (-0.3,-1.1) node{\tiny$\circled{3}$};
\draw [blue][line width=1pt][postaction={on each segment={mid arrow=blue}}]
      (-0.2,-0.2) to[out=90, in=-45] (-1.65, 1.65);
\draw [blue] (-0.6,-0. ) node{\tiny$\circled{4}$};
\draw ( 2.0,-0.0) node[right]{\small\color{red}$a_1$};
\draw (-2.0, 0.0) node[left]{\small\color{red}$a_1$};
\draw (-0.0,-0.0) node[right]{\small\color{red}$a_2$};
\draw (-0.0, 2.0) node[above]{\small\color{red}$a_3$};
\draw (-0.0,-2.0) node[below]{\small\color{red}$a_3$};
\end{tikzpicture}
\
\begin{tikzpicture}[scale=0.55]
\filldraw [black!20] (-2, 2) circle (0.5cm);
\filldraw [black!20] ( 2, 2) circle (0.5cm);
\filldraw [black!20] ( 2,-2) circle (0.5cm);
\filldraw [black!20] (-2,-2) circle (0.5cm);
\draw [line width=1pt] (-2, 2) circle (0.5cm);
\draw [line width=1pt] ( 2, 2) circle (0.5cm);
\draw [line width=1pt] ( 2,-2) circle (0.5cm);
\draw [line width=1pt] (-2,-2) circle (0.5cm);
\draw [dash pattern=on 2pt off 2pt] (-2, 2)-- ( 2, 2);
\draw [dash pattern=on 2pt off 2pt] ( 2, 2)-- ( 2,-2);
\draw [dash pattern=on 2pt off 2pt] ( 2,-2)-- (-2,-2);
\draw [dash pattern=on 2pt off 2pt] (-2, 2)-- (-2,-2);
\draw [line width=1pt,color=ffqqqq] (-2.35,1.65)-- (-1.65,-1.65);
\draw [line width=1pt,color=ffqqqq] (-1.65,2.35)-- (1.65,1.65);
\draw [line width=1pt,color=ffqqqq] (1.65,1.65)-- (2.35,-1.65);
\draw [line width=1pt,color=ffqqqq] (1.65,-2.35)-- (-1.65,-1.65);
\draw [line width=1pt,color=ffqqqq] (-1.65,-1.65)-- (1.65,1.65);
\fill [color=qqwuqq] (-1.65, 1.65) circle (2.5pt);
\fill [color=qqwuqq] ( 1.65,-1.65) circle (2.5pt);
\filldraw [color=white, draw=ffqqqq, line width=0.5pt] (-1.65, 2.35) circle (2.5pt);
\fill [color=qqwuqq] (-2.35, 2.35) circle (2.5pt);
\filldraw [color=white, draw=ffqqqq, line width=0.5pt] (-1.65,-1.65) circle (2.5pt);
\filldraw [color=white, draw=ffqqqq, line width=0.5pt] (-2.35,-2.35) circle (2.5pt);
\fill [color=qqwuqq] (-2.35,-1.65) circle (2.5pt);
\fill [color=qqwuqq] (-1.65,-2.35) circle (2.5pt);
\filldraw [color=white, draw=ffqqqq, line width=0.5pt] (-2.35, 1.65) circle (2.5pt);
\fill [color=qqwuqq] ( 1.65, 2.35) circle (2.5pt);
\filldraw [color=white, draw=ffqqqq, line width=0.5pt] ( 1.65, 1.65) circle (2.5pt);
\filldraw [color=white, draw=ffqqqq, line width=0.5pt] ( 2.35, 2.35) circle (2.5pt);
\fill [color=qqwuqq] ( 2.35, 1.65) circle (2.5pt);
\filldraw [color=white, draw=ffqqqq, line width=0.5pt] ( 2.35,-1.65) circle (2.5pt);
\fill [color=qqwuqq] ( 2.35,-2.35) circle (2.5pt);
\filldraw [color=white, draw=ffqqqq, line width=0.5pt] ( 1.65,-2.35) circle (2.5pt);
\draw [color=qqwuqq] ( 1.8,-1.8) node{$p$};
\draw [color=qqwuqq] (-1.8, 1.8) node{$q$};
\draw [color=ffqqqq] (-1.8,-1.8) node{$r$};
\draw [color=ffqqqq] (-1.65, 2.35) node[above]{$r$};
\draw [color=ffqqqq] ( 2.35,-1.65) node[right]{$r$};
\draw [color=ffqqqq] ( 1.8, 1.8) node{$s$};
\draw [color=ffqqqq] ( 1.65,-2.35) node[below]{$s$};
\draw [color=ffqqqq] (-2.35, 1.65) node[left]{$s$};
\draw [blue][line width=1pt][postaction={on each segment={mid arrow=blue}}]
      (1.65, -1.65) to[out=90, in=180] (2.1, -0.5);
\draw [blue] (1.25, -1.25) node{\tiny$\circled{1}$};
\draw [blue][line width=1pt][postaction={on each segment={mid arrow=blue}}]
      (-2, 0.5) to[out= 0,in=-90] (-0.2, 2);
\draw [blue] (-0.1, 1.1) node{\tiny$\circled{2}$};
\draw [blue][line width=1pt][postaction={on each segment={mid arrow=blue}}]
      (-0.4,-1.9) to[out=90, in=-45] (-0.8,-0.8);
\draw [blue] (-0.3,-1.1) node{\tiny$\circled{3}$};
\draw [blue][line width=1pt][postaction={on each segment={mid arrow=blue}}]
      (-0.5,-0.5) to[out=135, in=0] (-2,-0.1);
\draw [blue] (-1,0.2) node{\tiny$\circled{4}$};
\draw [blue][line width=1pt][postaction={on each segment={mid arrow=blue}}]
      ( 2, 0) to[out=180, in=90] (0,-2);
\draw [blue] (0.85,-0.85) node{\tiny$\circled{5}$};
\draw [blue][line width=1pt][postaction={on each segment={mid arrow=blue}}]
      (-0.85, 2.2) to[out=-90, in=0] (-1.65, 1.65);
\draw [blue] (-1.3, 1.3) node{\tiny$\circled{6}$};
\draw ( 2.0,-0.0) node[right]{\small\color{red}$a_1$};
\draw (-2.0, 0.0) node[left]{\small\color{red}$a_1$};
\draw (-0.0,-0.0) node[right]{\small\color{red}$a_2$};
\draw (-0.0, 2.0) node[above]{\small\color{red}$a_3$};
\draw (-0.0,-2.0) node[below]{\small\color{red}$a_3$};
\end{tikzpicture}
\
\begin{tikzpicture}[scale=0.55]
\filldraw [black!20] (-2, 2) circle (0.5cm);
\filldraw [black!20] ( 2, 2) circle (0.5cm);
\filldraw [black!20] ( 2,-2) circle (0.5cm);
\filldraw [black!20] (-2,-2) circle (0.5cm);
\draw [line width=1pt] (-2, 2) circle (0.5cm);
\draw [line width=1pt] ( 2, 2) circle (0.5cm);
\draw [line width=1pt] ( 2,-2) circle (0.5cm);
\draw [line width=1pt] (-2,-2) circle (0.5cm);
\draw [dash pattern=on 2pt off 2pt] (-2, 2)-- ( 2, 2);
\draw [dash pattern=on 2pt off 2pt] ( 2, 2)-- ( 2,-2);
\draw [dash pattern=on 2pt off 2pt] ( 2,-2)-- (-2,-2);
\draw [dash pattern=on 2pt off 2pt] (-2, 2)-- (-2,-2);
\draw [line width=1pt,color=ffqqqq] (-2.35,1.65)-- (-1.65,-1.65);
\draw [line width=1pt,color=ffqqqq] (-1.65,2.35)-- (1.65,1.65);
\draw [line width=1pt,color=ffqqqq] (1.65,1.65)-- (2.35,-1.65);
\draw [line width=1pt,color=ffqqqq] (1.65,-2.35)-- (-1.65,-1.65);
\draw [line width=1pt,color=ffqqqq] (-1.65,-1.65)-- (1.65,1.65);
\fill [color=qqwuqq] (-1.65, 1.65) circle (2.5pt);
\fill [color=qqwuqq] ( 1.65,-1.65) circle (2.5pt);
\filldraw [color=white, draw=ffqqqq, line width=0.5pt] (-1.65, 2.35) circle (2.5pt);
\fill [color=qqwuqq] (-2.35, 2.35) circle (2.5pt);
\filldraw [color=white, draw=ffqqqq, line width=0.5pt] (-1.65,-1.65) circle (2.5pt);
\filldraw [color=white, draw=ffqqqq, line width=0.5pt] (-2.35,-2.35) circle (2.5pt);
\fill [color=qqwuqq] (-2.35,-1.65) circle (2.5pt);
\fill [color=qqwuqq] (-1.65,-2.35) circle (2.5pt);
\filldraw [color=white, draw=ffqqqq, line width=0.5pt] (-2.35, 1.65) circle (2.5pt);
\fill [color=qqwuqq] ( 1.65, 2.35) circle (2.5pt);
\filldraw [color=white, draw=ffqqqq, line width=0.5pt] ( 1.65, 1.65) circle (2.5pt);
\filldraw [color=white, draw=ffqqqq, line width=0.5pt] ( 2.35, 2.35) circle (2.5pt);
\fill [color=qqwuqq] ( 2.35, 1.65) circle (2.5pt);
\filldraw [color=white, draw=ffqqqq, line width=0.5pt] ( 2.35,-1.65) circle (2.5pt);
\fill [color=qqwuqq] ( 2.35,-2.35) circle (2.5pt);
\filldraw [color=white, draw=ffqqqq, line width=0.5pt] ( 1.65,-2.35) circle (2.5pt);
\draw [color=qqwuqq] ( 1.8,-1.8) node{$p$};
\draw [color=qqwuqq] (-1.8, 1.8) node{$q$};
\draw [color=ffqqqq] (-1.8,-1.8) node{$r$};
\draw [color=ffqqqq] (-1.65, 2.35) node[above]{$r$};
\draw [color=ffqqqq] ( 2.35,-1.65) node[right]{$r$};
\draw [color=ffqqqq] ( 1.8, 1.8) node{$s$};
\draw [color=ffqqqq] ( 1.65,-2.35) node[below]{$s$};
\draw [color=ffqqqq] (-2.35, 1.65) node[left]{$s$};
\draw [blue][line width=1pt][postaction={on each segment={mid arrow=blue}}]
      (1.65, -1.65) to[out=90, in=180] (2.1, -0.5);
\draw [blue] (1.25, -1.25) node{\tiny$\circled{1}$};
\draw [blue][line width=1pt][postaction={on each segment={mid arrow=blue}}]
      (-2, 0.5) to[out= 0,in=-90] (-0.2, 2);
\draw [blue] (-0.2, 1.1) node{\tiny$\circled{2}$};
\draw [blue][line width=1pt][postaction={on each segment={mid arrow=blue}}]
      (-0.4,-1.9) to[out=90, in=-45] (-0.8,-0.8);
\draw [blue] (-0.3,-1.1) node{\tiny$\circled{3}$};
\draw [blue][line width=1pt][postaction={on each segment={mid arrow=blue}}]
      (-0.5,-0.5) to[out=135, in=0] (-2,-0.1);
\draw [blue] (-1,0.2) node{\tiny$\circled{4}$};
\draw [blue][line width=1pt][postaction={on each segment={mid arrow=blue}}]
      ( 2, 0) to[out=180, in=90] (0,-2);
\draw [blue] (0.85,-0.85) node{\tiny$\circled{5}$};
\draw [blue][line width=1pt][postaction={on each segment={mid arrow=blue}}]
      ( 0.3, 1.95) to[out=-90, in=135] ( 0.8, 0.8);
\draw [blue] ( 0.75, 1.35) node{\tiny$\circled{6}$};
\draw [blue][line width=1pt][postaction={on each segment={mid arrow=blue}}]
      ( 1., 1.) to[out=-45, in=180] (1.9, 0.6);
\draw [blue] (1.1, 0.4) node{\tiny$\circled{7}$};
\draw ( 2.0,-0.0) node[right]{\small\color{red}$a_1$};
\draw (-2.0, 0.0) node[left]{\small\color{red}$a_1$};
\draw (-0.0,-0.0) node[right]{\small\color{red}$a_2$};
\draw (-0.0, 2.0) node[above]{\small\color{red}$a_3$};
\draw (-0.0,-2.0) node[below]{\small\color{red}$a_3$};
\end{tikzpicture}
\end{center}
\caption{Case A.2}
\label{fig:case A.2}
\end{figure}
    \end{itemize}

    \item[Case B] $\snum{i_2}{\ws}=2$. If $\snum{i_3}{\ws}=1$, by Lemma~\ref{lem:cancel}, $\snum{i_4}{\ws}=3$. Then $\ws_{i_3,i_3+1}$ crosses $\ws_{i_1,i_1+1}$, see the first picture of \Pic\ref{fig:case B}, a contradiction. Similarly, if $\snum{i_3}{\ws}=3$, we also have that $\ws_{i_3,i_3+1}$ crosses $\ws_{i_1,i_1+1}$, see the second picture of \Pic\ref{fig:case B}, a contradiction.
\begin{figure}[htbp]
\begin{center}
\definecolor{ffqqqq}{rgb}{1,0,0}
\definecolor{qqwuqq}{rgb}{0,0,1}
\begin{tikzpicture}[scale=0.55]
\filldraw [black!20] (-2, 2) circle (0.5cm);
\filldraw [black!20] ( 2, 2) circle (0.5cm);
\filldraw [black!20] ( 2,-2) circle (0.5cm);
\filldraw [black!20] (-2,-2) circle (0.5cm);
\draw [line width=1pt] (-2, 2) circle (0.5cm);
\draw [line width=1pt] ( 2, 2) circle (0.5cm);
\draw [line width=1pt] ( 2,-2) circle (0.5cm);
\draw [line width=1pt] (-2,-2) circle (0.5cm);
\draw [dash pattern=on 2pt off 2pt] (-2, 2)-- ( 2, 2);
\draw [dash pattern=on 2pt off 2pt] ( 2, 2)-- ( 2,-2);
\draw [dash pattern=on 2pt off 2pt] ( 2,-2)-- (-2,-2);
\draw [dash pattern=on 2pt off 2pt] (-2, 2)-- (-2,-2);
\draw [line width=1pt,color=ffqqqq] (-2.35,1.65)-- (-1.65,-1.65);
\draw [line width=1pt,color=ffqqqq] (-1.65,2.35)-- (1.65,1.65);
\draw [line width=1pt,color=ffqqqq] (1.65,1.65)-- (2.35,-1.65);
\draw [line width=1pt,color=ffqqqq] (1.65,-2.35)-- (-1.65,-1.65);
\draw [line width=1pt,color=ffqqqq] (-1.65,-1.65)-- (1.65,1.65);
\fill [color=qqwuqq] (-1.65, 1.65) circle (2.5pt);
\fill [color=qqwuqq] ( 1.65,-1.65) circle (2.5pt);
\filldraw [color=white, draw=ffqqqq, line width=0.5pt] (-1.65, 2.35) circle (2.5pt);
\fill [color=qqwuqq] (-2.35, 2.35) circle (2.5pt);
\filldraw [color=white, draw=ffqqqq, line width=0.5pt] (-1.65,-1.65) circle (2.5pt);
\filldraw [color=white, draw=ffqqqq, line width=0.5pt] (-2.35,-2.35) circle (2.5pt);
\fill [color=qqwuqq] (-2.35,-1.65) circle (2.5pt);
\fill [color=qqwuqq] (-1.65,-2.35) circle (2.5pt);
\filldraw [color=white, draw=ffqqqq, line width=0.5pt] (-2.35, 1.65) circle (2.5pt);
\fill [color=qqwuqq] ( 1.65, 2.35) circle (2.5pt);
\filldraw [color=white, draw=ffqqqq, line width=0.5pt] ( 1.65, 1.65) circle (2.5pt);
\filldraw [color=white, draw=ffqqqq, line width=0.5pt] ( 2.35, 2.35) circle (2.5pt);
\fill [color=qqwuqq] ( 2.35, 1.65) circle (2.5pt);
\filldraw [color=white, draw=ffqqqq, line width=0.5pt] ( 2.35,-1.65) circle (2.5pt);
\fill [color=qqwuqq] ( 2.35,-2.35) circle (2.5pt);
\filldraw [color=white, draw=ffqqqq, line width=0.5pt] ( 1.65,-2.35) circle (2.5pt);
\draw [color=qqwuqq] ( 1.8,-1.8) node{$p$};
\draw [color=qqwuqq] (-1.8, 1.8) node{$q$};
\draw [color=ffqqqq] (-1.8,-1.8) node{$r$};
\draw [color=ffqqqq] (-1.65, 2.35) node[above]{$r$};
\draw [color=ffqqqq] ( 2.35,-1.65) node[right]{$r$};
\draw [color=ffqqqq] ( 1.8, 1.8) node{$s$};
\draw [color=ffqqqq] ( 1.65,-2.35) node[below]{$s$};
\draw [color=ffqqqq] (-2.35, 1.65) node[left]{$s$};
\draw [blue][line width=1pt][postaction={on each segment={mid arrow=blue}}]
      (1.65, -1.65) -- (0, 0);
\draw [blue] (1.2,-0.8) node{\tiny$\circled{1}$};
\draw [blue][line width=1pt][postaction={on each segment={mid arrow=blue}}]
      (-0.5,-0.5) to[out=135,in=0] (-2,-0.);
\draw [blue] (-1, 0.2) node{\tiny$\circled{2}$};
\draw [blue][line width=1pt][postaction={on each segment={mid arrow=blue}}]
      (2,-0.) to[out=180,in=90] (0,-2);
\draw [blue] (1,-0.)node{\tiny$\circled{3}$};
\draw ( 2.0,-0.0) node[right]{\small\color{red}$a_1$};
\draw (-2.0, 0.0) node[left]{\small\color{red}$a_1$};
\draw ( 0.5, 0.5) node[right]{\small\color{red}$a_2$};
\draw (-0.0, 2.0) node[above]{\small\color{red}$a_3$};
\draw (-0.0,-2.0) node[below]{\small\color{red}$a_3$};
\end{tikzpicture}
\
\begin{tikzpicture}[scale=0.55]
\filldraw [black!20] (-2, 2) circle (0.5cm);
\filldraw [black!20] ( 2, 2) circle (0.5cm);
\filldraw [black!20] ( 2,-2) circle (0.5cm);
\filldraw [black!20] (-2,-2) circle (0.5cm);
\draw [line width=1pt] (-2, 2) circle (0.5cm);
\draw [line width=1pt] ( 2, 2) circle (0.5cm);
\draw [line width=1pt] ( 2,-2) circle (0.5cm);
\draw [line width=1pt] (-2,-2) circle (0.5cm);
\draw [dash pattern=on 2pt off 2pt] (-2, 2)-- ( 2, 2);
\draw [dash pattern=on 2pt off 2pt] ( 2, 2)-- ( 2,-2);
\draw [dash pattern=on 2pt off 2pt] ( 2,-2)-- (-2,-2);
\draw [dash pattern=on 2pt off 2pt] (-2, 2)-- (-2,-2);
\draw [line width=1pt,color=ffqqqq] (-2.35,1.65)-- (-1.65,-1.65);
\draw [line width=1pt,color=ffqqqq] (-1.65,2.35)-- (1.65,1.65);
\draw [line width=1pt,color=ffqqqq] (1.65,1.65)-- (2.35,-1.65);
\draw [line width=1pt,color=ffqqqq] (1.65,-2.35)-- (-1.65,-1.65);
\draw [line width=1pt,color=ffqqqq] (-1.65,-1.65)-- (1.65,1.65);
\fill [color=qqwuqq] (-1.65, 1.65) circle (2.5pt);
\fill [color=qqwuqq] ( 1.65,-1.65) circle (2.5pt);
\filldraw [color=white, draw=ffqqqq, line width=0.5pt] (-1.65, 2.35) circle (2.5pt);
\fill [color=qqwuqq] (-2.35, 2.35) circle (2.5pt);
\filldraw [color=white, draw=ffqqqq, line width=0.5pt] (-1.65,-1.65) circle (2.5pt);
\filldraw [color=white, draw=ffqqqq, line width=0.5pt] (-2.35,-2.35) circle (2.5pt);
\fill [color=qqwuqq] (-2.35,-1.65) circle (2.5pt);
\fill [color=qqwuqq] (-1.65,-2.35) circle (2.5pt);
\filldraw [color=white, draw=ffqqqq, line width=0.5pt] (-2.35, 1.65) circle (2.5pt);
\fill [color=qqwuqq] ( 1.65, 2.35) circle (2.5pt);
\filldraw [color=white, draw=ffqqqq, line width=0.5pt] ( 1.65, 1.65) circle (2.5pt);
\filldraw [color=white, draw=ffqqqq, line width=0.5pt] ( 2.35, 2.35) circle (2.5pt);
\fill [color=qqwuqq] ( 2.35, 1.65) circle (2.5pt);
\filldraw [color=white, draw=ffqqqq, line width=0.5pt] ( 2.35,-1.65) circle (2.5pt);
\fill [color=qqwuqq] ( 2.35,-2.35) circle (2.5pt);
\filldraw [color=white, draw=ffqqqq, line width=0.5pt] ( 1.65,-2.35) circle (2.5pt);
\draw [color=qqwuqq] ( 1.8,-1.8) node{$p$};
\draw [color=qqwuqq] (-1.8, 1.8) node{$q$};
\draw [color=ffqqqq] (-1.8,-1.8) node{$r$};
\draw [color=ffqqqq] (-1.65, 2.35) node[above]{$r$};
\draw [color=ffqqqq] ( 2.35,-1.65) node[right]{$r$};
\draw [color=ffqqqq] ( 1.8, 1.8) node{$s$};
\draw [color=ffqqqq] ( 1.65,-2.35) node[below]{$s$};
\draw [color=ffqqqq] (-2.35, 1.65) node[left]{$s$};
\draw [blue][line width=1pt][postaction={on each segment={mid arrow=blue}}]
      (1.65, -1.65) -- (0, 0);
\draw [blue] (1.2,-0.8) node{\tiny$\circled{1}$};
\draw [blue][line width=1pt][postaction={on each segment={mid arrow=blue}}]
      ( 0.5, 0.5) to[out=135,in=-90] (0, 2);
\draw [blue] (-0.2, 1) node{\tiny$\circled{2}$};
\draw [blue][line width=1pt][postaction={on each segment={mid arrow=blue}}]
      (0,-2) to[out=90,in=180] (2,-0);
\draw [blue] (1,-0.)node{\tiny$\circled{3}$};
\draw ( 2.0,-0.0) node[right]{\small\color{red}$a_1$};
\draw (-2.0, 0.0) node[left]{\small\color{red}$a_1$};
\draw (-0.0,-0.0) node[right]{\small\color{red}$a_2$};
\draw (-0.0, 2.0) node[above]{\small\color{red}$a_3$};
\draw (-0.0,-2.0) node[below]{\small\color{red}$a_3$};
\end{tikzpicture}
\end{center}
\caption{Case B}
\label{fig:case B}
\end{figure}
    \item[Case C] $\snum{i_2}{\ws}=3$. This case is dual to Case A. Hence the possible simplest sequences are those shown in \Pic\ref{fig:case C},
\begin{figure}[htbp]
\begin{center}
\definecolor{ffqqqq}{rgb}{1,0,0}
\definecolor{qqwuqq}{rgb}{0,0,1}
\begin{tikzpicture}[scale=0.55]
\filldraw [black!20] (-2, 2) circle (0.5cm);
\filldraw [black!20] ( 2, 2) circle (0.5cm);
\filldraw [black!20] ( 2,-2) circle (0.5cm);
\filldraw [black!20] (-2,-2) circle (0.5cm);
\draw [line width=1pt] (-2, 2) circle (0.5cm);
\draw [line width=1pt] ( 2, 2) circle (0.5cm);
\draw [line width=1pt] ( 2,-2) circle (0.5cm);
\draw [line width=1pt] (-2,-2) circle (0.5cm);
\draw [dash pattern=on 2pt off 2pt] (-2, 2)-- ( 2, 2);
\draw [dash pattern=on 2pt off 2pt] ( 2, 2)-- ( 2,-2);
\draw [dash pattern=on 2pt off 2pt] ( 2,-2)-- (-2,-2);
\draw [dash pattern=on 2pt off 2pt] (-2, 2)-- (-2,-2);
\draw [line width=1pt,color=ffqqqq] (-2.35,1.65)-- (-1.65,-1.65);
\draw [line width=1pt,color=ffqqqq] (-1.65,2.35)-- (1.65,1.65);
\draw [line width=1pt,color=ffqqqq] (1.65,1.65)-- (2.35,-1.65);
\draw [line width=1pt,color=ffqqqq] (1.65,-2.35)-- (-1.65,-1.65);
\draw [line width=1pt,color=ffqqqq] (-1.65,-1.65)-- (1.65,1.65);
\fill [color=qqwuqq] (-1.65, 1.65) circle (2.5pt);
\fill [color=qqwuqq] ( 1.65,-1.65) circle (2.5pt);
\filldraw [color=white, draw=ffqqqq, line width=0.5pt] (-1.65, 2.35) circle (2.5pt);
\fill [color=qqwuqq] (-2.35, 2.35) circle (2.5pt);
\filldraw [color=white, draw=ffqqqq, line width=0.5pt] (-1.65,-1.65) circle (2.5pt);
\filldraw [color=white, draw=ffqqqq, line width=0.5pt] (-2.35,-2.35) circle (2.5pt);
\fill [color=qqwuqq] (-2.35,-1.65) circle (2.5pt);
\fill [color=qqwuqq] (-1.65,-2.35) circle (2.5pt);
\filldraw [color=white, draw=ffqqqq, line width=0.5pt] (-2.35, 1.65) circle (2.5pt);
\fill [color=qqwuqq] ( 1.65, 2.35) circle (2.5pt);
\filldraw [color=white, draw=ffqqqq, line width=0.5pt] ( 1.65, 1.65) circle (2.5pt);
\filldraw [color=white, draw=ffqqqq, line width=0.5pt] ( 2.35, 2.35) circle (2.5pt);
\fill [color=qqwuqq] ( 2.35, 1.65) circle (2.5pt);
\filldraw [color=white, draw=ffqqqq, line width=0.5pt] ( 2.35,-1.65) circle (2.5pt);
\fill [color=qqwuqq] ( 2.35,-2.35) circle (2.5pt);
\filldraw [color=white, draw=ffqqqq, line width=0.5pt] ( 1.65,-2.35) circle (2.5pt);
\draw [color=qqwuqq] ( 1.8,-1.8) node{$p$};
\draw [color=qqwuqq] (-1.8, 1.8) node{$q$};
\draw [color=ffqqqq] (-1.8,-1.8) node{$r$};
\draw [color=ffqqqq] (-1.65, 2.35) node[above]{$r$};
\draw [color=ffqqqq] ( 2.35,-1.65) node[right]{$r$};
\draw [color=ffqqqq] ( 1.8, 1.8) node{$s$};
\draw [color=ffqqqq] ( 1.65,-2.35) node[below]{$s$};
\draw [color=ffqqqq] (-2.35, 1.65) node[left]{$s$};
\draw [blue][line width=1pt][postaction={on each segment={mid arrow=blue}}]
      (1.65,-1.65) to[out=180,in=90] (0.5,-2);
\draw [blue] (1,-1.35) node{\tiny$\circled{1}$};
\draw [blue][line width=1pt][postaction={on each segment={mid arrow=blue}}]
      (0.3,1.9) to[out=-90,in=135] (0.8,0.8);
\draw [blue] (-0.1,1.3) node{\tiny$\circled{2}$};
\draw [blue][line width=1pt][postaction={on each segment={mid arrow=blue}}]
      (0.6,0.6) to[out=-45,in=180] (2,0);
\draw [blue] (1.5,-0.3) node{\tiny$\circled{3}$};
\draw [blue][line width=1pt][postaction={on each segment={mid arrow=blue}}]
      (-2,0.5) to[out=0,in=-90] (-1.65,1.65);
\draw [blue] (-1.3, 1) node{\tiny$\circled{4}$};
\draw ( 2.0,-0.0) node[right]{\small\color{red}$a_1$};
\draw (-2.0, 0.0) node[left]{\small\color{red}$a_1$};
\draw ( 0.0, 0.0) node[right]{\small\color{red}$a_2$};
\draw (-0.0, 2.0) node[above]{\small\color{red}$a_3$};
\draw (-0.0,-2.0) node[below]{\small\color{red}$a_3$};
\end{tikzpicture}
\
\begin{tikzpicture}[scale=0.5]
\filldraw [black!20] (-2, 2) circle (0.5cm);
\filldraw [black!20] ( 2, 2) circle (0.5cm);
\filldraw [black!20] ( 2,-2) circle (0.5cm);
\filldraw [black!20] (-2,-2) circle (0.5cm);
\draw [line width=1pt] (-2, 2) circle (0.5cm);
\draw [line width=1pt] ( 2, 2) circle (0.5cm);
\draw [line width=1pt] ( 2,-2) circle (0.5cm);
\draw [line width=1pt] (-2,-2) circle (0.5cm);
\draw [dash pattern=on 2pt off 2pt] (-2, 2)-- ( 2, 2);
\draw [dash pattern=on 2pt off 2pt] ( 2, 2)-- ( 2,-2);
\draw [dash pattern=on 2pt off 2pt] ( 2,-2)-- (-2,-2);
\draw [dash pattern=on 2pt off 2pt] (-2, 2)-- (-2,-2);
\draw [line width=1pt,color=ffqqqq] (-2.35,1.65)-- (-1.65,-1.65);
\draw [line width=1pt,color=ffqqqq] (-1.65,2.35)-- (1.65,1.65);
\draw [line width=1pt,color=ffqqqq] (1.65,1.65)-- (2.35,-1.65);
\draw [line width=1pt,color=ffqqqq] (1.65,-2.35)-- (-1.65,-1.65);
\draw [line width=1pt,color=ffqqqq] (-1.65,-1.65)-- (1.65,1.65);
\fill [color=qqwuqq] (-1.65, 1.65) circle (2.5pt);
\fill [color=qqwuqq] ( 1.65,-1.65) circle (2.5pt);
\filldraw [color=white, draw=ffqqqq, line width=0.5pt] (-1.65, 2.35) circle (2.5pt);
\fill [color=qqwuqq] (-2.35, 2.35) circle (2.5pt);
\filldraw [color=white, draw=ffqqqq, line width=0.5pt] (-1.65,-1.65) circle (2.5pt);
\filldraw [color=white, draw=ffqqqq, line width=0.5pt] (-2.35,-2.35) circle (2.5pt);
\fill [color=qqwuqq] (-2.35,-1.65) circle (2.5pt);
\fill [color=qqwuqq] (-1.65,-2.35) circle (2.5pt);
\filldraw [color=white, draw=ffqqqq, line width=0.5pt] (-2.35, 1.65) circle (2.5pt);
\fill [color=qqwuqq] ( 1.65, 2.35) circle (2.5pt);
\filldraw [color=white, draw=ffqqqq, line width=0.5pt] ( 1.65, 1.65) circle (2.5pt);
\filldraw [color=white, draw=ffqqqq, line width=0.5pt] ( 2.35, 2.35) circle (2.5pt);
\fill [color=qqwuqq] ( 2.35, 1.65) circle (2.5pt);
\filldraw [color=white, draw=ffqqqq, line width=0.5pt] ( 2.35,-1.65) circle (2.5pt);
\fill [color=qqwuqq] ( 2.35,-2.35) circle (2.5pt);
\filldraw [color=white, draw=ffqqqq, line width=0.5pt] ( 1.65,-2.35) circle (2.5pt);
\draw [color=qqwuqq] ( 1.8,-1.8) node{$p$};
\draw [color=qqwuqq] (-1.8, 1.8) node{$q$};
\draw [color=ffqqqq] (-1.8,-1.8) node{$r$};
\draw [color=ffqqqq] (-1.65, 2.35) node[above]{$r$};
\draw [color=ffqqqq] ( 2.35,-1.65) node[right]{$r$};
\draw [color=ffqqqq] ( 1.8, 1.8) node{$s$};
\draw [color=ffqqqq] ( 1.65,-2.35) node[below]{$s$};
\draw [color=ffqqqq] (-2.35, 1.65) node[left]{$s$};
\draw [blue][line width=1pt][postaction={on each segment={mid arrow=blue}}]
      (1.65,-1.65) to[out=180,in=90] (0.5,-2);
\draw [blue] ( 1.0,-1.35) node{\tiny$\circled{1}$};
\draw [blue][line width=1pt][postaction={on each segment={mid arrow=blue}}]
      (0,2) to[out=-90,in= 0] (-2,-0.3);
\draw [blue] (-0.75, 0.75) node{\tiny$\circled{2}$};
\draw [blue][line width=1pt][postaction={on each segment={mid arrow=blue}}]
      (2,0.5) to[out=180,in=-45] (0.9, 0.9);
\draw [blue] ( 1.45, 0.85) node{\tiny$\circled{3}$};
\draw [blue][line width=1pt][postaction={on each segment={mid arrow=blue}}]
      (0.8, 0.8) to[out=135,in=-90] (0.35, 1.9);
\draw [blue] ( 0.9, 1.4) node{\tiny$\circled{4}$};
\draw [blue][line width=1pt][postaction={on each segment={mid arrow=blue}}]
      (0,-2) to[out=90,in=180] (2,0);
\draw [blue] ( 1.5,-0.55 ) node{\tiny$\circled{5}$};
\draw [blue][line width=1pt][postaction={on each segment={mid arrow=blue}}]
      (-2,0.3) to[out=0,in=-90] (-1.65,1.65);
\draw [blue] (-1.30, 1.25) node{\tiny$\circled{6}$};
\draw ( 2.0,-0.0) node[right]{\small\color{red}$a_1$};
\draw (-2.0, 0.0) node[left]{\small\color{red}$a_1$};
\draw (-0.0,-0.0) node[right]{\small\color{red}$a_2$};
\draw (-0.0, 2.0) node[above]{\small\color{red}$a_3$};
\draw (-0.0,-2.0) node[below]{\small\color{red}$a_3$};
\end{tikzpicture}
\end{center}
\caption{Case C}
\label{fig:case C}
\end{figure}
    which are dual to the first picture of \Pic\ref{fig:case A.1} and the second picture of 
    \Pic\ref{fig:case A.2}, respectively. In the first picture, by Lemma~\ref{lem:wg}, $\ws$ is homotopic to $\tgamma$, a contradiction; 
    in the second picture, we have $\sdii{n(\ws)}{\ws}=\sdii{1}{\ws}\geq 2>0$ as required.
\end{itemize}
Thus, the proof is complete.
\end{proof}




\def\cprime{$'$}

\end{document}